\documentclass[11pt]{amsart}

\usepackage[usenames,dvipsnames]{xcolor}
\usepackage{textcomp}
\usepackage{tikz}
\usetikzlibrary{shapes}
\usetikzlibrary{decorations.pathreplacing}
\usepackage{enumerate}
\usepackage[shortlabels]{enumitem}
\usepackage{verbatim}
\usepackage[margin=1in]{geometry}
\usepackage{amssymb}
\usepackage{caption}
\usepackage[normalem]{ulem}
\usepackage{bbm}
\usepackage{amsmath}
\usepackage{amsthm}

\usepackage{enumitem}
\usepackage{mdframed}
\usepackage[numbers,comma,square,sort&compress]{natbib}
\usepackage{scalerel}    
\usepackage{stmaryrd}   

\usetikzlibrary{arrows.meta}

\DeclareFontFamily{U}{mathx}{\hyphenchar\font45}
\DeclareFontShape{U}{mathx}{m}{n}{
      <5> <6> <7> <8> <9> <10>
      <10.95> <12> <14.4> <17.28> <20.74> <24.88>
      mathx10
      }{}
\DeclareSymbolFont{mathx}{U}{mathx}{m}{n}
\DeclareFontSubstitution{U}{mathx}{m}{n}
\DeclareMathAccent{\widecheck}{0}{mathx}{"71}
\DeclareMathAccent{\wideparen}{0}{mathx}{"75}

\usepackage{rotating,centernot,cancel}    

\usepackage{float}

\usepackage{tcolorbox}   
\PassOptionsToPackage{hyphens}{url}
\definecolor{my-link}{rgb}{0.5,0.0,0.0}
\definecolor{my-blue}{rgb}{0.0,0.0,0.6}
\definecolor{my-red}{rgb}{0.5,0.0,0.0}
\definecolor{my-green}{rgb}{0.2,0.5,0.2}
\definecolor{darkgreen}{rgb}{0.0,0.5,0.0}
\definecolor{darkblue}{rgb}{0.0,0.0,0.3}
\definecolor{light-gray}{gray}{0.7}
\RequirePackage[colorlinks,urlcolor=my-blue,linkcolor=my-red,citecolor=my-green]{hyperref}
\usepackage{babel}
\theoremstyle{plain}

\input amssym.def
\input amssym.tex

\DeclareRobustCommand{\chiup}{{\mathpalette\irchi\relax}}
\newcommand{\irchi}[2]{\raisebox{\depth}{$#1\chi$}} 

\def\ncouple{\mathcal{H}}

\def\beqn{\begin{eqnarray}}
\def\eeqn{\end{eqnarray}}

\def\beq{\begin{equation}}
\def\eeq{\end{equation}}
\usepackage{mathtools}
\DeclarePairedDelimiter\floor{\lfloor}{\rfloor}

\newtheorem{theorem}{Theorem}[section]
\newtheorem*{definition*}{Definition}
\newtheorem{corollary}[theorem]{Corollary} 
\newtheorem{lemma}[theorem]{Lemma} 
\newtheorem*{lemma*}{Lemma}
\newtheorem{proposition}[theorem]{Proposition}

\theoremstyle{remark}
\newtheorem{remark}[theorem]{Remark}

\theoremstyle{definition}

\numberwithin{figure}{section}
\numberwithin{equation}{section}

\def\R{\mathbb R}

\def\Z{\mathbb Z}

\def\to{\rightarrow}
\def\dlim[#1][#2]{\lim_{#1 \to #2, #1 \neq #2}}

\def\Var{\text{$\mathbb{V}$ar}}

\newcommand{\be}{\begin{equation}}
\newcommand{\ee}{\end{equation}}
\newcommand{\lzb}{\llbracket}   
\newcommand{\rzb}{\rrbracket}   

\providecommand{\abs}[1]{\vert#1\vert}

\def\w{\omega} 

\def\wt{\widetilde}    \def\wc{\widecheck}

\usepackage{xargs}
\usepackage[colorinlistoftodos,prependcaption,textsize=footnotesize]{todonotes}
\newcommandx{\addmath}[2][1=]{\todo[linecolor=red,backgroundcolor=red!25,bordercolor=red,#1]{#2}}
\newcommandx{\fixtext}[2][1=]{\todo[linecolor=blue,backgroundcolor=blue!25,bordercolor=blue,#1]{#2}}
\newcommandx{\note}[2][1=]{\todo[linecolor=yellow,backgroundcolor=yellow!25,bordercolor=yellow,#1]{#2}}

\def\cD{\mathcal{D}}

\def\cL{\mathcal{L}}

\newcommand\abullet{\hspace{0.6pt}{\raisebox{1pt}{\scaleobj{0.6}{\bullet}}}\hspace{0.8pt}}  
\newcommand\bbullet{{{\scaleobj{0.6}{\bullet}}}} 
\newcommand\babullet{{\raisebox{0.5pt}{\scaleobj{0.6}{\bullet}}}} 
\newcommand\brbullet{{\raisebox{-1pt}{\scaleobj{0.5}{\bullet}}}} 

\def\evec{e}    

\def\tspb{\hspace{0.9pt}}

\title[Coalescence of semi-infinite polymers]{Coalescence and total-variation distance\\ of semi-infinite inverse-gamma polymers}

\author[F.~Rassoul-Agha]{Firas Rassoul-Agha}
\address{Firas Rassoul-Agha\\ University of Utah\\  Mathematics Department\\ 155S 1400E\\   Salt Lake City, UT 84112\\ USA.}
\email{firas@math.utah.edu}
\urladdr{http://www.math.utah.edu/~firas}
\thanks{F.\ Rassoul-Agha was partially supported by National Science Foundation grants DMS-1811090 and DMS-2054630 and MPS-Simons Fellowship grant 823136}

\author[T.~Sepp\"al\"ainen]{Timo Sepp\"al\"ainen} \address{Timo Sepp\"al\"ainen\\ University of Wisconsin--Madison\\  Department of Mathematics\\ 480 Lincoln Drive \\  Madison, WI 53706\\ USA.}  \email{seppalai@math.wisc.edu}\urladdr{http://www.math.wisc.edu/~seppalai} \thanks{T.~Sepp\"al\"ainen was partially supported by National Science Foundation grants  DMS-1854619 and DMS-2152362 and by the Wisconsin Alumni Research Foundation.}  

\author[X.~Shen]{Xiao Shen}
\address{Xiao Shen\\ University of Utah\\  Mathematics Department\\ 155S 1400E\\   Salt Lake City, UT 84112\\ USA.}
\email{xiao.shen@utah.edu}
\urladdr{https://www.xshen.org}


\date{\today}
\begin{document}

\begin{abstract}
We show that two semi-infinite positive temperature polymers coalesce on the scale predicted by KPZ (Kardar-Parisi-Zhang) universality. The two polymer paths have the same asymptotic direction and evolve  in the same environment, independently until coalescence. If they start at distance $k$ apart, their coalescence occurs on the scale $k^{3/2}$. It follows that the total variation distance of two semi-infinite polymer measures decays on this same scale. Our results are  upper and lower  bounds on probabilities and expectations that match, up to constant factors and occasional logarithmic corrections. Our proofs are done in the context of the solvable inverse-gamma polymer model, but without appeal to integrable probability. With minor modifications, our proofs give also bounds on transversal fluctuations of the polymer path. Since the free energy of a directed polymer is a discretization of a stochastically forced viscous Hamilton-Jacobi equation, our results suggest that the hyperbolicity phenomenon of such equations obeys the KPZ exponent. 
\end{abstract}

\maketitle

\allowdisplaybreaks


\tableofcontents

\addtocontents{toc}{\protect\setcounter{tocdepth}{1}}
\section{Introduction}
This paper focuses on a probability model for nearest-neighbor up-right random walk paths on the two-dimensional square lattice. The lattice vertices are assigned independent and identically distributed random variables called {\sl weights}, and the {\sl energy} of a path is defined as the sum of the weights along the path. The {\sl point-to-point quenched polymer measures} are probability measures on admissible paths connecting pairs of sites. The probability of a path is proportional to the exponential of its energy. 

This model is known as the {\sl two-dimensional directed lattice polymer with bulk disorder} and was introduced 
in the statistical physics literature by Huse and Henley \cite{Hus-Hen-85} in 1985 to represent the domain wall in the ferromagnetic Ising
model with random impurities. 
This model is expected to be a member of  the Kardar-Parisi-Zhang (KPZ) universality class and has been extensively studied over the past three decades, becoming a paradigmatic model in the field of nonequilibrium statistical mechanics. 
See the surveys \cite{Com-17,Gia-07,Hol-09,Cor-12,Cor-16,Hal-Tak-15,Qua-12,Qua-Spo-15,Zyg-22}.

The {\sl directed last-passage percolation model} (LPP)  on the square lattice is a {\sl zero-temperature} version of the random polymer model. In LPP, we consider the {\sl ground states}, which are admissible paths that maximize the energy, and are referred to as {\sl geodesics}. This particular LPP model with up-right nearest-neighbor lattice paths is also called the {\sl corner growth model}. 

In LPP, a path that starts from a given lattice vertex and only moves up or right is called a {\sl semi-infinite geodesic} if each finite piece of the path is a geodesic between its endpoints. The existence, directedness, and uniqueness or non-uniqueness of semi-infinite geodesics have been well studied and understood (see \cite{Geo-Ras-Sep-17-ptrf-1,Geo-Ras-Sep-17-ptrf-2,Jan-Ras-Sep-22-} for details). Notably, it has been demonstrated in \cite{Geo-Ras-Sep-17-ptrf-2} that these semi-infinite geodesics can be obtained as limits of finite geodesics, as the endpoint moves off towards infinity in a particular direction. Furthermore, it has been shown in the same paper that semi-infinite geodesics starting at different vertices but having the same asymptotic direction eventually coalesce, i.e., they intersect and then move together. 

The study of semi-infinite polymer measures in the case of random directed lattice  polymers was carried out in \cite{Jan-Ras-20-aop,Geo-etal-15}. Similar to LPP, \cite{Jan-Ras-20-aop} established that semi-infinite polymer measures that start from different vertices and share the same asymptotic velocity can be coupled in such a way that their paths coalesce with probability one. As a consequence, the marginals of any two semi-infinite polymer measures that correspond to the same asymptotic velocity are asymptotic to each other. This phenomenon, known as {\sl hyperbolicity}, has been found to be linked to various phenomena such as {\sl stochastic synchronization} and the {\sl one force--one solution principle} (see, for example, \cite{Bak-Li-19,Jan-Ras-Sep-22-1F1S-}). In this work, our focus is on providing precise quantitative bounds on the convergence rates, showcasing how this hyperbolicity obeys the KPZ exponents.
Currently, such sharp estimates are only available in the so-called {\sl solvable} cases, where the weight distribution is chosen in a specific way, allowing for explicit analytic computations. 

The only known solvable LPP models are the ones with either exponential  or geometric weight distribution. 
In the only known solvable directed polymer model, the weights have a negative  log-gamma distribution. This solvable directed polymer model was first introduced by the second author in \cite{Sep-12-corr} and has since been referred to as the {\sl inverse-gamma} or {\sl log-gamma} {\sl polymer}.


Our main contributions in this paper are sharp quantitative bounds on the rates of coalescence of the coupled paths and convergence of the marginals in the inverse-gamma polymer model. The corresponding estimates for LPP with exponential weights were obtained in \cite{Bas-Sar-Sly-19} using integrable probability methods, and in \cite{Sep-She-20} using coupling with stationary versions of the model, which relies less on the solvability of the model. In this paper, we adopt the latter approach and further develop it to handle the additional layer of randomness that arises in the case of semi-infinite polymer measures, where the random environment only determines the path measures. Along the way, we provide various new estimates on the exit point of stationary polymers and we improve one existing estimate, namely the last inequality in \eqref{improved}.

\subsection*{Organization of the paper.}
In Section \ref{sec:main}, we present the setting and our main results concerning the coalescence point, total variation distance, and transversal fluctuations. The connection to hyperbolicity in stochastic Hamilton-Jacobi equations  is addressed briefly in  Remark \ref{rmk:hyp}. Exit time estimates in the stationary inverse-gamma polymer are a crucial tool in our proofs. We introduce the stationary polymer in Section \ref{stat_poly} and provide the exit time estimates in Section \ref{exit_sec}. The proofs of the coalescence results are presented in Section \ref{dual_coal}, while the proofs of the total variation distance estimates can be found in Section \ref{TV_dist}. The proofs of the transversal fluctuations results are provided in Section \ref{sec:trans}. Various auxiliary results are gathered in the appendixes.

\subsection*{Notation and conventions.} 
Subscripts indicate restricted subsets of the reals and integers: for example,   
 $\Z_{>0}=\{1,2,3,\dotsc\}$ and $\Z_{>0}^2=(\Z_{>0})^2$ is the strictly positive first quadrant of the planar integer lattice.    
 
 On $\R^2$ we have the following conventions for points $x=(x_1,x_2)$ and $y=(y_1,y_2)$.   Coordinatewise order: $x\le y$ iff  $x_1\le y_1$  and $x_2\le y_2$.    The $\ell^1$ norm is $\abs{x}_1=\abs{x_1}+\abs{x_2}$.  The origin of $\R^2$ is denoted by both $0$ and $(0,0)$.  The two standard basis vectors are $\evec_1=(1,0)$ and  $\evec_2=(0, 1)$. 
 
 For integers $m\le n$, the integer interval is denoted by $\lzb m,n\rzb=\{m,m+1,\dotsc,n\}$. For planar points 
 $a\le b$  in $\Z^2 $,  $\lzb a, b\rzb =\{x\in\Z^2: a\le x\le b\}$ is the rectangle in $\Z^2$ with corners $a$ and $b$.
The northeast boundary of a rectangle $[\![a, b ]\!]$, denoted by $\partial^{\textup{NE}}[\![a, b]\!]$, is the set of  vertices $v\in [\![a, b ]\!]$ such that $v\cdot e_1 = b\cdot e_1$ or $v\cdot e_2 = b\cdot e_2$. 
 $\lzb a, b\rzb$ is an integer line segment in $\Z^2$  if $a$ and $b$ are on the same horizontal or vertical line. In particular, 
   $\lzb a - e_1, a\rzb $ and $\lzb a - e_2, a\rzb $ denote unit edges.
 
The total variation distance between two probability measures $\mu$ and $\nu$ on $(\Omega, \mathcal{F})$ is 
$d_{\textup{TV}}(\mu, \nu)  = \sup_{A\in \mathcal{F}}|\mu(A) - \nu(A)|.$ For a probability measure $\mu$, $X\sim\mu$ means the random variable $X$ has distribution $\mu$.

\addtocontents{toc}{\protect\setcounter{tocdepth}{2}}
\section{Main results}\label{sec:main}

\subsection{Directed polymer model} 
Let $\{Y_z\}_{z\in \mathbb{Z}^2}$ be a collection of positive weights on the sites of the planar integer square lattice. For vertices $u\leq v$ in $\mathbb{Z}^2$, $\mathbb{X}_{u,v}$ denotes the collection of up-right paths  $x_\bbullet = \{x_i\}_{0\leq i \leq n}$ where $n = |u-v|_1$, $x_0 = u$, $x_n = v$ and $x_{i+1}-x_i\in\{e_1,e_2\}$ for all $i\in \lzb 0, n-1\rzb $.
Define the \textit{point-to-point polymer partition function} between the two vertices $u\leq v$ by 
$$Z_{u,v} = \sum_{x_\brbullet \in \mathbb{X}_{u,v}} \prod_{i=0}^{|u-v|_1} Y_{x_i}.$$
We use the convention $Z_{u,v} =  0$ if $u\leq v$ fails. The \textit{quenched polymer measure} is a probability measure on the set $\mathbb{X}_{u,v}$ and is defined by 
$$Q_{u,v}\{x_{\bbullet} \} = \frac{1}{Z_{u,v}} \prod_{i=0}^{|u-v|_1} Y_{x_i}.$$
In general, the positive weights $\{Y_z\}_{z\in \mathbb{Z}^2}$ can be seen as a random environment if they are chosen as i.i.d.~positive random variables defined on some probability space $(\Omega, \mathcal{F}, \mathbb{P})$. 
Under the moment assumption 
$$\mathbb{E}[|\log Y_x|^p] < \infty \quad  \text{ for some $p>2$},$$
 there exists a concave, positively homogeneous, nonrandom continuous function $\Lambda: \mathbb{R}^2_{\geq 0} \rightarrow \mathbb{R}$ that satisfies the \textit{shape theorem} (see \cite[Section 2.3]{Jan-Ras-20-aop}):
\be\label{La8} \lim_{n\rightarrow \infty} \sup_{z\in \Z^2_{\geq 0} : |z|_1 \geq n} \frac{|\log Z_{0, z} - \Lambda(z)|}{|z|_1} = 0 \qquad \text{$\mathbb{P}$-almost surely}.\ee
$\Lambda$ is  called the (\textit{limiting}) \textit{free energy density} or, by analogy with stochastic growth models, the    \textit{shape function}.  
Regularity  properties of $\Lambda$ such as strict convexity or differentiability are not known in general. 

Fix a base point $v\in\Z^2$ and let $x_N\ge v$ in $\Z^2$  be a sequence of lattice points going to infinity in a deterministic  direction $\xi$, i.e.\ $x_N/|x_N|_1\xrightarrow[N\to \infty]{} \xi/|\xi|_1$. The $\xi$-directed semi-infinite polymer measure is obtained as the weak limit 
\be\label{Q41} Q_{v, x_N} \xrightharpoonup[N\to\infty]{} \Pi^{\xi}_v,\ee
provided this weak limit exists $\mathbb{P}$-a.s. The probability measure $\Pi^{\xi}_v$ is the  quenched path measure of a random walk in a random environment (RWRE) on $\Z^2$   started at $v$. An RWRE is Markov chain whose transition probability depends on the environment in a translation-covariant way.  In the polymer case these  transition probabilities are given by limiting ratios of partition functions.
If the shape function $\Lambda$ (as a function of directions) has sufficient local regularity around the direction $\xi$,  then the limiting measure $\Pi^\xi_v$ exists  \cite[Theorem 3.8]{Jan-Ras-20-aop}.  

\subsection{Inverse-gamma polymer}
This paper focuses exclusively on the  \textit{inverse-gamma polymer}.  A real random variable $X$ has the inverse-gamma distribution with shape parameter $\mu\in(0,\infty)$, abbreviated as $X\sim \text{Ga}^{-1}(\mu)$, if its reciprocal $X^{-1}$ has the gamma distribution with shape parameter $a$. Equivalently, $X$ has probability density function  
$$f_{X}(x) = \frac{1}{\Gamma(\mu)} x^{-1-\mu}e^{-x^{-1}}\mathbbm{1}_{(0, \infty)}(x)$$
where $\Gamma(a) = \int_0^\infty s^{a-1}e^{-s} ds$ is the gamma function. 
The inverse-gamma polymer is defined by letting $\{Y_z\}_{z\in \mathbb{Z}^2}$ be i.i.d.~inverse-gamma distributed random variables.   We will fix the shape parameter $\mu$ in the rest of the paper. While many of the constants in the proofs depend on $\mu$, we will not explicitly mention this fact.

In the current state of the subject, $\Lambda$ in \eqref{La8}  can be  written down explicitly only in the  inverse-gamma case.  Then  the   regularity of $\Lambda$ required  for \eqref{Q41} can be verified explicitly.  Hence   for each given direction $\xi$ in the open
first quadrant and each initial vertex $v\in\Z^2$, the measure $\Pi^{\xi}_v$ exists almost surely \cite[Theorem 7.1]{Geo-etal-15}. Its transition probability is given in equation \eqref{polymerRWRE} below. 

Let $\Psi_0$ and 
$\Psi_1$ be the digamma and  trigamma functions,  defined by  $\Psi_0(z) = \frac{d}{dz} \log\Gamma(z)$ and  $\Psi_1(z) =\Psi_0'(z)= \frac{d^2}{dz^2} \log\Gamma(z)$.
In the study of the inverse-gamma polymer, it is convenient  to index the spatial directions $\xi$ by the parameter ${{\rho}} \in (0, \mu)$ through 
\begin{equation}\label{char_dir}
\xi[{{\rho}}] =  \Big(\tfrac{\Psi_1({{\rho}})}{\Psi_1({{\rho}}) + \Psi_1({{\mu - \rho}})}, \tfrac{\Psi_1({{\mu-\rho}})}{\Psi_1({{\rho}}) + \Psi_1({{\mu - \rho}})}\Big). 
\end{equation}
We call $\xi[{{\rho}}]$ the \textit{characteristic direction} associated to the parameter ${{\rho}}$.  This notion acquires its full meaning when we discuss the stationary inverse-gamma polymer in Section \ref{stat_poly}.
The formula for the shape function $\Lambda$ is cleanest  in terms of the characteristic direction: from (2.16) in \cite{Sep-12-corr}
\[\Lambda(\xi[\rho]) = -\tfrac{\Psi_1({{\rho}})}{\Psi_1({{\rho}}) + \Psi_1({{\mu - \rho}})}\cdot  \Psi_0( \mu-\rho)  -  \tfrac{\Psi_1({{\mu-\rho}})}{\Psi_1({{\rho}}) + \Psi_1({{\mu - \rho}})}  \Psi_0(\rho).\]    
Throughout the paper, $N$ is a scaling parameter
that goes to infinity. We define the particular sequence of lattice points 
\be \label{v_N} v_N = \bigl(\floor{N\xi[\rho] \cdot e_1},\floor{N \xi[\rho] \cdot e_2}  \bigr)\;\in\; \Z_{\ge0}^2  \ee
that  go to infinity in the characteristic direction $\xi[{{\rho}}]$. We simplify the notation for the semi-infinite polymer distribution to $\Pi_v^\rho=\Pi_v^{\xi[\rho]}$.

\begin{figure}[t]
\captionsetup{width=0.8\textwidth}
\begin{center}
 
\begin{tikzpicture}[>=latex, scale=1]

\draw[color=lightgray,line width = 0.3mm, ->] (0,-1)--(0,4);
\draw[color=lightgray,line width = 0.3mm, ->] (0,-1)--(5,-1);

\draw[lightgray,  line width = 0.5mm ]  (0,2) -- (3.5,2)--(3.5,-1);


\draw[ fill=black] (3,3)circle(1.2mm);

\draw[color=black, dotted, line width = 1.2mm, ->] (0,0.3)-- (0.5,0.3) --(0.5,1)-- (1,1) -- (1,1.5)--(2,1.5)--(2.5,1.5) -- (2.5,3) -- (4,3) -- (4,4) ;

\draw[color=gray, dotted, line width = 1.2mm] (1.3,-1)-- (1.3,-0.5) --(1.8,-0.5) -- (1.8, 0) -- (2.3,0) -- (2.3,1) -- (3,1) -- (3,3);

\fill[color=white] (0,-1)circle(1.7mm); 
\draw[ fill=lightgray](0,-1)circle(1mm);
\node at (-0.7,0-1) {$(0,0)$};

\draw[ fill=black] (3,3)circle(1.2mm);
\draw[ fill=white](3,3)circle(0.7mm);

\fill[color=white] (3.5,2)circle(1.7mm); 
\draw[ fill=lightgray](3.5,2)circle(1mm);
\node at (4.1,2) {$v_N$};

\draw[color=lightgray,line width = 0.3mm, ->] (0+7,-1)--(0+7,4);
\draw[color=lightgray,line width = 0.3mm, ->] (0+7,-1)--(5+7,-1);

\draw[lightgray,  line width = 0.5mm ]  (7,2) -- (3.5+7,2)--(3.5+7,-1);


\draw[color=black, dotted, line width = 1.2mm, ->] (0+7,-0.3) --(0.5+7, -0.3)-- (0.5+7,0.2) --(1+7,0.2)-- (1+7,0.5) --(1+7,1) -- (1+7,1) -- (1.5+7 ,1)-- (1.5+7,1.5)--(2+7,1.5)--(2.5+7,1.5) -- (2.5+7,2)  -- (2.5+7,3) -- (3+7+1,3) -- (3+7+1,4);
\draw[color=gray, dotted, line width = 1.2mm] (0.7+7,-1) -- (8, -1) -- (8, -0.5) -- (8.5, -0.5) -- (8.5, 0.7) -- (9, 0.7) -- (9,1.5);
\fill[color=white] (0+7,-1)circle(1.7mm); 
\draw[ fill=lightgray](0+7,-1)circle(1mm);
\node at (-0.7+7,0-1) { $(0,0)$};

\draw[ fill=black] (2+7,1.5)circle(1.2mm);
\draw[ fill=white](2+7,1.5)circle(0.7mm);

\fill[color=white] (3.5+7,2)circle(1.7mm); 
\draw[ fill=lightgray](3.5+7,2)circle(1mm);
\node at (4.1+7,2) {$v_N$};

\draw[ fill=lightgray](0, 0.3)circle(1mm);

\node at (1.3, -1.4) {\footnotesize$r N^{2/3}$};

\node at (-0.65, 0.3) {\footnotesize$r N^{2/3}$};
 
\draw[ fill=lightgray](1.3, -1)circle(1mm);

\draw[ fill=lightgray](7, -0.3)circle(1mm);

\draw[ fill=lightgray](0.7+7, -1)circle(1mm);
\node at (0.8+7, -1.4) {\footnotesize${\delta N^{2/3}}$};

\node at (-0.98+7.3, -0.3) {\footnotesize${\delta N^{2/3}}$};

\end{tikzpicture}

\end{center}
\caption{\small These pictures  illustrate the likely events which are  the complements of the rare events bounded in Theorems \ref{d_coal} and \ref{r_coal}. The open circle marks the  coalescence point of two $\xi[\rho]$-directed semi-infinite polymer paths.  On the left $r$ is large and the initial points are far apart on the scale $N^{2/3}$. Consequently the two paths are unlikely to coalesce before exiting the rectangle. On the right $\delta$ is small and coalescence inside the rectangle is likely.  }
\label{sec2fig1}
\end{figure}
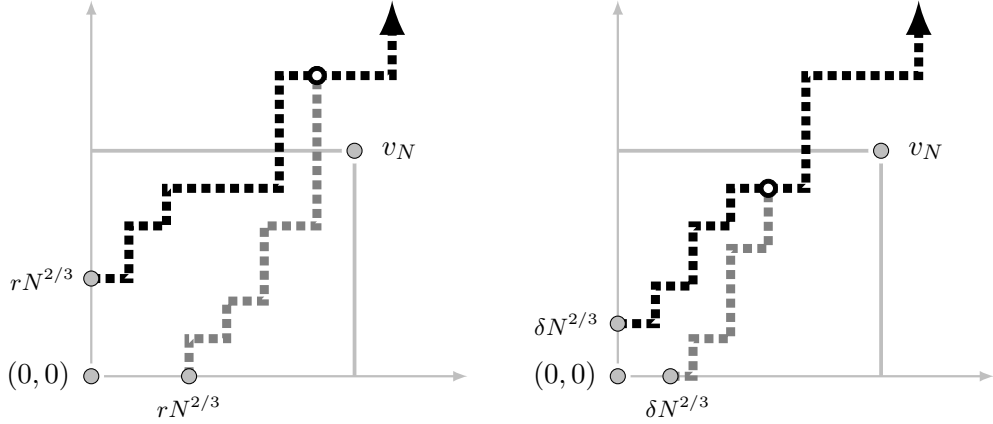

\subsection{Coalescence bounds}
For two initial vertices $a,b\in\Z^2$, let $\ncouple^\rho_{a, b}$ denote the classical coupling measure of the Markov chains $\Pi_a^\rho$ and $\Pi_b^\rho$, as defined by Thorisson \cite[Chapter 2]{Tho-00}.  
Under the distribution $\ncouple^\rho_{a, b}$, the two paths evolve jointly as a Markov chain on $\Z^2\times\Z^2$ with marginal distributions  $\Pi_a^\rho$ and $\Pi_b^\rho$. The joint transition probability is defined on $\Z^2\times\Z^2$ so that the two paths move 
independently until they meet, after which they move together. When this meeting  happens we say that the two paths coalesced.
By \cite[Theorem A.1]{Jan-Ras-20-aop}, for a given $\rho$, coalescence happens $\ncouple^\rho_{a, b}$-almost surely, for almost every environment.


We quantify the speed of coalescence by specifying the lattice subset in which the coalescence first happens.  
For $A\subset \mathbb{Z}^2$, let $\Gamma^A$ denote the collection of pairs of
semi-infinite up-right paths in $\Z^2$ that first meet at a vertex inside the set $A$.
Then, $\ncouple^\rho_{a,b}\bigl(\Gamma^{[\![0, v_N]\!]}\bigr)$ is the quenched probability that the coalescence of the paths from $a$ and $b$ happens  inside the set $[\![0, v_N]\!]$. Similarly, 
$\ncouple^\rho_{a,b}\bigl(\Gamma^{\mathbb{Z}^2\setminus [\![0, v_N]\!]}\bigr)$ is  
the quenched probability that the coalescence happens   outside  $[\![0, v_N]\!]$. 
The two theorems below give upper and lower bounds on the expectations of these quenched probabilities in two distinct cases: when   the initial  points are close together  and when they are  far apart  on the scale $N^{2/3}$.

\begin{theorem}\label{d_coal}
Let $\varepsilon \in(0,\mu/2)$. There exist positive constants $C_1, C_2, N_0, \delta_0$ depending only on $\varepsilon$ such that for each $\rho \in [\varepsilon, \mu-\varepsilon]$, $N\geq N_0$ and $N^{-2/3}\leq \delta \leq \delta_0$, we have 
$$C_1\delta \leq \mathbb{E} \Big[ \ncouple^\rho_{\floor{\delta N^{2/3}}e_1,\floor{\delta N^{2/3}}e_2}\Big(\Gamma^{\mathbb{Z}^2\setminus [\![0, v_N]\!]}\Big) \Big] \leq C_2|\log \delta|^{10}\delta .$$ 
\end{theorem}

\begin{remark} The restriction $\delta \geq N^{-2/3}$ is needed only for the lower bound of the theorem and only for the trivial reason that   the expectation vanishes when  $\delta < N^{-2/3}$ because then the two paths start together at the origin.  
\end{remark}

\begin{theorem} \label{r_coal}
Let $\varepsilon \in (0, \mu/2)$. There exist positive constants $C_1, C_2, r_0, c_0, N_0$ that depend only on $\varepsilon$ such that for each $\rho \in [\varepsilon, \mu-\varepsilon]$, $N\geq N_0$ and $r_0 \leq r \leq c_0N^{1/3}$, we have  
$$e^{-C_1 r^3} \leq \mathbb{E} \Big[ \ncouple^\rho_{\floor{r N^{2/3}}e_1,\floor{r N^{2/3}}e_2}\Big(\Gamma^{[\![0, v_N]\!]}\Big) \Big] \leq e^{-C_2 r^3}.$$ 
\end{theorem}
\begin{remark}
Again, the upper bound $r\leq c_0 N^{1/3}$ is only needed for the lower bound in the theorem. 
\end{remark}

The estimates above do not depend on starting the paths on an antidiagonal. 
The following corollary gives two of the four additional estimates.  The other two follow from the theorems. Also, $e_1$ and $e_2$ are interchangeable by symmetry. 
\begin{corollary}\label{h_start}
Let $\varepsilon \in(0,\mu/2)$. There exist positive constants $C, N_0, \delta_0, r_0$ that depend only on $\varepsilon$ such that for each $\rho \in [\varepsilon, \mu-\varepsilon]$, $N\geq N_0$, $r\geq r_0$ and $N^{-2/3}\leq \delta \leq \delta_0$, we have 
\[\mathbb{E} \Big[ \ncouple^\rho_{0,\floor{rN^{2/3}}e_1}\Big(\Gamma^{[\![0, v_N]\!]}\Big) \Big] \leq e^{-C r^3} \quad\text{and}\quad
\mathbb{E} \Big[ \ncouple^\rho_{0,\floor{\delta N^{2/3}}e_1}\Big(\Gamma^{\mathbb{Z}^2\setminus [\![0, v_N]\!]}\Big) \Big] \geq C\delta .\]
\end{corollary}

By planar monotonicity and a change of variable, our estimates can also be stated for two semi-infinite polymer paths that start at fixed locations.  If the initial points are of order  $k$ apart, then their meeting takes place on the scale $k^{3/2}$, as captured in the  corollary below. We shift the rectangle with the initial points so that the constants do not depend at all on the initial points.  The coordinatewise minimum of two lattice points $a=(a_1,a_2)$ and $b=(b_1,b_2)$ is denoted by 
$a\wedge b=(a_1\wedge b_1, a_2\wedge b_2)$.

\begin{corollary}\label{cor:coal} 
Let $\varepsilon \in (0, \mu/2)$ and $a\ne  b$ in $\mathbb{Z}^2$.  Let  $k= |a-b|_1 \ge 1$. There exist positive constants $C_1, C_2, r_0, c_0$ that depend only on $\varepsilon$ such that for each $\rho \in [\varepsilon, \mu-\varepsilon]$, $k\geq 1$, $r \geq r_0$ and $\delta \geq c_0 k^{-1/2}$ we have  
\begin{align*}C_1r^{-2/3} &\leq \mathbb{E} \Big[ \ncouple^\rho_{a,b}\Big(\Gamma^{\mathbb{Z}^2\setminus \{a\wedge b\,+\, [\![0, v_{rk^{3/2}}]\!]\}}\Big) \Big] \leq C_2(\log r )^{10}r^{-2/3}\quad\text{and} \\
e^{-C_2 \delta^{-2}} &\leq \mathbb{E} \Big[ \ncouple^\rho_{a,b}\Big(\Gamma^{ a\wedge b\,+\, [\![0,  v_{\delta k^{3/2}}]\!]}\Big) \Big] \leq e^{-C_1 \delta^{-2}}.
\end{align*}
\end{corollary}

The next result gives   tail bounds for the quenched probability of fast  coalescence, of optimal exponential order.

\begin{theorem} \label{r_coal_tail}
Fix $\varepsilon \in (0, \mu/2)$. There exist positive constants $C_1, C_2, C_3, C_4, r_0, c_0, N_0$ that depend only on $\varepsilon$ such that for each $\rho \in [\varepsilon, \mu-\varepsilon]$, $N\geq N_0$ and $r_0 \leq r \leq c_0N^{1/3}$, we have  
\begin{align*} 
e^{-C_1 r^3} &\leq \mathbb{P} \Big( \ncouple^\rho_{\floor{r N^{2/3}}e_1,\floor{r N^{2/3}}e_2}\Big(\Gamma^{[\![0, v_N]\!]}\Big) \geq 1-e^{-C_2 r^2 N^{1/3}} \Big) \\
&\leq \mathbb{P}\Big( \ncouple^\rho_{\floor{r N^{2/3}}e_1,\floor{r N^{2/3}}e_2}\Big(\Gamma^{[\![0, v_N]\!]}\Big) \geq e^{-C_3 r^2 N^{1/3}} \Big)  \leq e^{-C_4 r^3}.
\end{align*}
\end{theorem}

\medskip 

\subsection{Coupling and total variation distance}
 Since the quenched non-coalescence probability $\ncouple^\rho_{a,b}(\Gamma^{\mathbb{Z}^2\setminus(a\wedge b\,+\, [\![0, \,  v_{rk^{3/2}}]\!])})$ is nonincreasing in $r$,  Corollary \ref{cor:coal} implies the almost sure convergence  $\ncouple^\rho_{a,b}(\Gamma^{\mathbb{Z}^2\setminus(a\wedge b\,+\, [\![0, \,  v_{rk^{3/2}}]\!])})\to0$ as $r\to\infty$.  This says that   the   polymer distributions $\Pi^\rho_a$  and $\Pi^\rho_b $ couple almost surely. 
 To state this precisely, let 
 $\chiup_N=\chiup_N(\gamma)$ denote the vertex
where a semi-infinite up-right path $\gamma$ started inside $[\![0, v_N]\!]$  first meets the northeast boundary $\partial^{\textup{NE}}[\![0, v_N]\!]$. 
If $(\gamma^a, \gamma^b)$ denote the paths under $\ncouple^\rho_{a,b}$, then for $a,b\in\Z_{\ge0}^2$ we have 
\be\label{H500} \ncouple^\rho_{a,b}\{\chiup_N(\gamma^a)=\chiup_N(\gamma^b) \text{ for large enough } N\}=1.
\ee
 The standard coupling inequality (stated in \eqref{coupineq} in Section \ref{TV_dist}) implies that the total variation distance between the distributions induced on $\partial^{\textup{NE}}[\![0, v_N]\!]$
converges to zero almost surely:
\be\label{dTV6} 
\lim_{N\to\infty}  d_{\textup{TV}}\bigl( \Pi^\rho_a\{\chiup_N \in \abullet\tspb\}\,,\, \Pi^\rho_b\{\chiup_N \in \abullet\tspb\} \bigr)  \;=\;0 
\quad \mathbb P\text{-a.s.} 
\ee

The next two theorems establish bounds on this convergence.  In the same spirit as in the earlier results,  when the initial  points are close on the scale $N^{2/3}$, the  total variation distance on the northeast boundary of a rectangle of size  $N$  is small.  In the opposite  case  the starting points are far apart on the scale $N^{2/3}$ and  the total variation distance is close to $1$.


\begin{theorem}\label{tv_d_upper}
Let $\varepsilon\in (0, \mu/2)$.  There exist finite strictly positive constants $\delta_0, N_0, C$ that  depend on $\varepsilon$ such that, whenever $0< \delta \leq \delta_0$, $N\geq N_0$ and $\rho \in [\varepsilon, \mu-\varepsilon]$, 
$$\mathbb{E}\Big[d_{\textup{TV}}\Big(\Pi^{{\rho}}_{\floor{\delta N^{2/3}}e_1} (\chiup_N \in \abullet), \Pi^{{\rho}}_{\floor{\delta N^{2/3} }e_2} (\chiup_N \in \abullet)\Big)\Big] \leq C |\log \delta|^{10}\delta.  $$
\end{theorem}

\begin{theorem}\label{tv_r_upper}
Let $\varepsilon\in (0, \mu/2)$. There exist finite positive constants $r_0, N_0, C$ depending on $\varepsilon$ such that whenever $N\geq N_0$, $r_0 \leq r \leq N^{1/3}$ and $\rho \in [\varepsilon, \mu-\varepsilon]$, we have 
$$ \mathbb{E}\Big[d_{\textup{TV}}\Big(\Pi^{{\rho}}_{\floor{r N^{2/3}}e_1} (\chiup_N \in \abullet), \Pi^{{\rho}}_{\floor{r N^{2/3}} e_2} (\chiup_N \in \abullet)\Big)\Big] \geq 1- e^{-C r^{3}}.  $$
\end{theorem}

The proofs of the two theorems are given in Section \ref{TV_dist}. 

\begin{remark}[Hyperbolicity in stochastic equations]  \label{rmk:hyp} 
 The free energy of a directed polymer can be viewed as a discretization of a stochastically forced viscous Hamilton-Jacobi equation.
This connection goes back to \cite{Hus-Hen-Fis-85,Imb-Spe-88}. In this vein, semi-infinite polymer measures can be used to construct stationary eternal solutions to such equations. Article \cite{Bak-Li-19} treats  a semidiscrete case and \cite{Jan-Ras-Sep-22-1F1S-}  the KPZ equation.    In particular, the limit \eqref{dTV6} is a version of {\sl hyperbolicity} that appears in   {\sl stochastic synchronization} (also called the {\sl one force--one solution principle}) of such equations. This is the positive temperature analogue of the inviscid phenomenon whereby action minimizers are asymptotic to each other in the infinite past.  See for example Theorem 4.4 of \cite{Bak-Li-19}.   Our results above show that, in the case at hand,  this form of hyperbolicity obeys the KPZ wandering exponent. On universality grounds one can predict that this is  true in some generality in one space dimension  for stochastically forced viscous Hamilton-Jacobi equations with nonlinear Hamiltonians.


\end{remark}

\subsection{Transversal fluctuations}
Finally, we present a result concerning the transversal fluctuation of the finite i.i.d.\ polymer. This result is derived by making a slight modification to the proof of the upper bound for fast coalescence, as stated in Theorem \ref{d_coal}.
It is expected for the midpoint of polymer from $(0,0)$ to $(N, N)$ to fluctuate around the diagonal on the scale $N^{2/3}$.  The upper bound on the transversal fluctuation was first proved in the work \cite{Sep-12-corr}, and we provide here the lower bound, i.e.\ we show that it is rare for the midpoint of the polymer to be too close to the diagonal. 

To state the result, let us introduce some notation. Let $\{\textsf{mid}\leq k\}$ denote the collection of directed paths between $-v_N$ and $v_N$ that intersect the $\ell^\infty$ ball of radius $k$, centered at the origin.

\begin{theorem}\label{fluc_lb}
Let $\varepsilon\in (0, \mu/2)$.  There exist finite strictly positive constants $\delta_0, N_0, C$ that  depend on $\varepsilon$ such that, whenever $0< \delta \leq \delta_0$, $N\geq N_0$ and $\rho \in [\varepsilon, \mu-\varepsilon]$, 
$$\mathbb{E}\Big[Q_{-v_N,v_N}\{\textup{\textsf{mid}} \leq \delta N^{2/3}\} \Big] \leq C |\log \delta|^{10}\delta.  $$
\end{theorem}

\begin{remark}
The midpoint transversal fluctuation can be generalized to other positions along the path, as long as they are order $N$ away from $-v_N$ and $v_N$.
\end{remark}

\begin{remark}
Our proof technique also yields the following lower bound on the fluctuation of the endpoint of the point-to-line polymer. Let $Q^{\textup{p2l}}_{0, N}$ denote the point-to-line quenched path measure on the collection of directed paths from $(0,0)$ to the anti-diagonal line $x+y = 2N$. And let $\{\textsf{end}\leq k\}$ denote the sub-collection of these paths that intersect the $\ell^\infty$ ball of radius $k$, centered at $(N, N)$. It holds that 
\begin{equation}\mathbb{E}\Big[Q^{\textup{p2l}}_{0, N}\{\textup{\textsf{end}} \leq \delta N^{2/3}\} \Big] \leq C |\log {\delta}|^{10}\sqrt{\delta}.
\end{equation}
We get the weaker $\sqrt\delta$ instead of $\delta$ because 
the antidiagonal version of the independence property of Busemann increments on horizontal or vertical lines for two different directions is not known. 
\end{remark}

\section{Stationary inverse-gamma polymer}\label{stat_poly}

One of the main tools we use in our proofs is a stationary version of the polymer model, which we now describe. 

The  stationary inverse-gamma polymer with southwest boundary is defined on a quadrant instead of the entire $\mathbb{Z}^2$. It requires a parameter parameter ${{\rho}}\in (0,\mu)$ and a base vertex $v\in \mathbb{Z}^2$. To each $z\in v+ \mathbb{Z}_{>0}^2$ we attach a weight $Y_z \sim \text{Ga}^{-1}(\mu)$. On the $e_1$- and $e_2$-boundary of $v+\mathbb{Z}^2_{\ge0}$, we place (edge) weights
\begin{align}\label{stat_weights}
\begin{aligned}
I^\rho_{v+ke_1}\sim \text{Ga}^{-1}(\mu-{{\rho}})\quad\text{and}\quad 
J^\rho_{v+ke_2}\sim \text{Ga}^{-1}({{\rho}}),\quad k\ge1.
\end{aligned}
\end{align}
All these weights in the quadrant are independent. 
We refer to the $Y$ weights as the \textit{bulk} weights and to the $I^\rho$ and $J^\rho$ weights as the $\rho$-boundary weights. Section \ref{Bus+Poly} below explains the reason behind thinking of $I^\rho$ and $J^\rho$ as edge weights instead of vertex weights.

We use the same  $\mathbb P$ to denote the joint distribution 
of the weights $(Y,I^\rho,J^\rho)$.   
For $w\in v+ \mathbb{Z}_{\geq 0}^2$, we define  the partition function of the stationary polymer by 
$$Z^{{{\rho}}}_{v, w} = \sum_{x_\bbullet\in \mathbb{X}_{v,w}} \prod_{i=0}^{|w-v|_1} \wt{Y}_{x_i},\text{ where for $x\in v+\mathbb Z^2_{\ge0}$,}\quad\wt{Y}_{x} = \begin{cases}
1 \quad & \text{if $x = v$,}\\
I^{{\rho}}_{x-e_1, x} \quad & \text{if $x\in v+\Z_{>0}e_1$,}\\
J^{{\rho}}_{x-e_2, x} \quad & \text{if $x\in v+\Z_{>0}e_2$,}\\
Y_x \quad & \text{for $x\in v+\Z^2_{>0}$.}
\end{cases}$$ 
The corresponding quenched polymer measure is defined as
$$Q^{{{\rho}}}_{v, w}(x_\bbullet) = \frac{1}{Z^{{{\rho}}}_{v, w}}  \prod_{i=0}^{|w-v|_1} \wt{Y}_{x_i}, \qquad x_\bbullet\in \mathbb{X}_{v,w}.$$



Next  we state the  theorem that  explains why the process $Z^\rho$  is called \textit{ratio-stationary}, or simply \textit{stationary}. 
For a subset $A\subset\mathbb Z^2$, let $A^>=\cup_{x\in A}(x+\mathbb Z^2_{>0})$.

\begin{theorem}[{\cite[Thm.\ 3.3]{Sep-12-corr} and \cite[Eqn.\ (3.6)]{Geo-etal-15}}]
\label{stat} 
Fix $\rho \in (0,\mu)$. For each $u\in v + (\mathbb Z_{>0}\times\mathbb Z_{\ge0})$,
$w\in v + (\mathbb Z_{\ge0}\times\mathbb Z_{>0})$,  and $x\in v+\mathbb Z^2_{>0}$ we have 
$$\frac{Z^\rho_{v, u}}{Z^\rho_{v, u- e_1}} \sim \textup{Ga}^{-1}(\mu-\rho), \quad \frac{Z^\rho_{v, w}}{Z^\rho_{v, w- e_2 }}\sim \textup{Ga}^{-1}(\rho),\quad
\text{and}\quad \frac1{Z^\rho_{v,x}/Z^\rho_{v,x-e_1}+Z^\rho_{v,x}/Z^\rho_{v,x-e_2}}\sim\textup{Ga}^{-1}(\mu).$$
Translation invariance: the distribution of the process 
\[\biggl\{\frac{Z^\rho_{v, z+u}}{Z^\rho_{v, z+u- e_1}}\,,\,\frac{Z^\rho_{v, z+w}}{Z^\rho_{v, z+w- e_2}}: u\in  \mathbb Z_{>0}\times\mathbb Z_{\ge0}, \; 
w\in  \mathbb Z_{\ge0}\times\mathbb Z_{>0}\biggr\}\] 
does not depend on the translation $z\in v+\mathbb Z^2_{\ge0}$.
Furthermore, let $A=\{y_i\}_{i\in \mathcal I}$ be any finite or infinite down-right path in $v+ \mathbb{Z}^2_{\geq 0}$, indexed by an interval $\mathcal I\subset\Z$. {\rm(}This means that each increment satisfies $y_{i+1} - y_i\in\{e_1,-e_2\}$.{\rm)} Then, the nearest-neighbor  ratios $\{Z^\rho_{v, y_{i+1}}/Z^\rho_{v, y_{i}}\}$ along the path and the weights $\bigl\{\bigl(Z^\rho_{v,x}/Z^\rho_{v,x-e_1}+Z^\rho_{v,x}/Z^\rho_{v,x-e_2}\bigr)^{-1}:x\in A^>\bigr\}$ are mutually independent.
\end{theorem}

A key quantity in the coupling approach to polymers and LPP models is the exit time. 
For an up-right path $\gamma$, we define $\tau(\gamma)\in\mathbb{Z}\setminus \{0\}$ as the signed number of steps taken before the first turn, where the plus sign corresponds to $e_1$ steps and the minus sign to $e_2$ steps.  For example,   $\tau(\gamma) = -3$ means that the first four steps of $\gamma$ consist of three consecutive  $e_2$ steps followed by an $e_1$ step. 
For  $v, w \in \mathbb{Z}$, when additional clarity is needed, we use the notation $\tau_{v, w}$ to denote  the restriction of the function $\tau$ to the domain $\mathbb{X}_{v, w}$.   When the path $\gamma$ starts at the base vertex $v$ of the stationary polymer process, $|\tau|$ equals  the number of boundary weights seen by the path before it exits the boundary. This justifies  the term \textit{exit time} for  $\tau(\gamma)$.


With the function $\tau$, we define the restricted partition function $Z_{v,w}(a \leq\tau\leq b)$ similarly to $Z_{v,w}$, except that we sum only over the subset of paths $\{ x_\bbullet\in \mathbb{X}_{v,w}: a\leq \tau_{v,w}(x_\bbullet) \leq b\}$.

Because the weights on the boundary are stochastically larger than the bulk weights, the path prefers to stay on the boundary. For each  ${{\rho}}\in(0,\mu)$  the characteristic direction $\xi[{{\rho}}]$ is the unique direction  in which the pulls of the $e_1$- and $e_2$-boundaries balance out.   The sampled path   between the origin and $v_N$ tends to take order $N^{2/3}$  steps on the boundary.   Precise  exit time estimates are stated in in Section \ref{exit_sec}.

\medskip

The stationary inverse-gamma polymer with northeast boundary is analogous  to the previously defined model, except that it is defined on a third quadrant and uses boundary edge weights placed on the northeast boundary. Thus, it also requires a parameter ${{\rho}}\in (0,\mu)$ and a base vertex $v\in \mathbb{Z}^2$, but it is defined on the quadrant $v-\mathbb{Z}^2_{\ge0}$. To each $z\in v+ \mathbb{Z}_{<0}^2$ we attach a bulk (vertex) weight $Y_z \sim \text{Ga}^{-1}(\mu)$. On the $e_1$- and $e_2$-boundary of $v-\mathbb{Z}^2_{\ge0}$, we place edge weights
\begin{align}\label{stat_weights_1}
\begin{aligned}
I^{{\rho}}_{[\![v+(k-1)ke_1, v+ke_1]\!]} &=  I^{{\rho}}_{v+(k-1)ke_1, v+ke_1} \sim \text{Ga}^{-1}(\mu-{{\rho}}),\\ 
J^{{\rho}}_{[\![v+(k-1)ke_2, v+ke_2]\!]} &= J^{{\rho}}_{v+(k-1)ke_2, v+ke_2}\sim \text{Ga}^{-1}({{\rho}}),\quad k\le0.
\end{aligned}
\end{align}
All these weights in the quadrant are independent. Here too, we  use $\mathbb P$ to denote the joint distribution 
of $(Y,I^\rho,J^\rho)$ and write $Z^{\rho,\textup{NE}}_{u,v}$ and $Q^{\rho,\textup{NE}}_{u,v}$ for, respectively, the partition function and quenched measure for the polymer with northeast boundary. 
Precisely, for $u\in v-\mathbb{Z}_{\geq 0}^2$, define  
$$Z^{\rho,\textup{NE}}_{u, v} = \sum_{x_\bbullet\in \mathbb{X}_{u,v}} \prod_{i=0}^{|v-u|_1} \wt{Y}_{x_i},\text{ where for $x\in v-\mathbb Z^2_{\ge0}$,}\quad\wt{Y}_{x} = \begin{cases}
1 \quad & \text{if $x = v$,}\\
I^{{\rho}}_{x, x+e_1} \quad & \text{if $x\in v-\Z_{>0}e_1$,}\\
J^{{\rho}}_{x, x+e_2} \quad & \text{if $x\in v-\Z_{>0}e_2$,}\\
Y_x \quad & \text{for $x\in v-\Z^2_{>0}$.}
\end{cases}$$ 
The quenched polymer measure is defined by
$$Q^{\rho,\textup{NE}}_{u, v}(x_\bbullet) = \frac{1}{Z^{\rho,\textup{NE}}_{v, w}}  \prod_{i=0}^{|v-u|_1} \wt{Y}_{x_i}.$$

\begin{remark}
We work mostly with the stationary model with southwest boundary and, therefore, we only flesh out the location of the boundary when it is the northeast boundary that is being used.
\end{remark}

By symmetry, the analogous version of Theorem \ref{stat} holds for the stationary polymer with northeast boundary.

\section{Exit time estimates}\label{exit_sec}

In this section, we prove exit time estimates for the stationary polymer model with southwest boundary, introduced in Section \ref{stat_poly}.  These results will be used to derive the coalescence estimate in Section \ref{dual_coal} and the total variation bounds in Section \ref{TV_dist}.

The first theorem below concerns the case when the polymer paths have an unusually large exit time. The upper bound for the annealed measure is proved in \cite{Lan-Sos-22-a-,Emr-Jan-Xie-23-}. We improve this estimate into a bound for the quenched tail.
The related upper bound in the zero-temperature model is \cite[Theorem 2.4]{Bha-20}. The proof in \cite{Bha-20} uses a technical result from \cite[Theorem 10.5]{Bas-Sid-Sly-14-}.   
We will present a simpler proof in this paper.

\begin{theorem} \label{r_up_low}
Fix $\varepsilon\in(0,\mu/2)$. There exist positive constants $r_0$, $N_0$, $c_0$, and $C_i$, $i\in[\![1,6]\!]$, 
that depend only on $\varepsilon$ such that for all ${{\rho}} \in [\varepsilon, \mu-\varepsilon]$, $N\geq N_0$ and $r_0 \leq r \leq c_0N^{1/3}$, we have  
\begin{align}
e^{-C_1r^{3}} &\leq \mathbb{P}\Big(\min_{x\not\in [\![0, v_N]\!]} Q^\rho_{0, x}\{|\tau| > rN^{2/3}\} \geq 1-e^{-C_2r^2 N^{1/3}}\Big)\notag\\
& 
\leq \mathbb{P}\Big(Q^{{\rho}}_{0, v_N+(1,1)}\{|\tau| > rN^{2/3}\} \geq e^{-C_3r^2 N^{1/3}}\Big) \leq e^{-C_4r^{3}}\label{improved}
\end{align}
and
$$ e^{-C_5r^3} \leq \mathbb{E}\Big[\min_{x\not\in [\![0, v_N]\!]}Q^{{{\rho}}}_{0, v_N}\{|\tau| > rN^{2/3}\}  \Big]  \leq \mathbb{E}\Big[Q^{{\rho}}_{0, v_N+(1,1)}\{|\tau| > rN^{2/3}\}  \Big]\leq e^{-C_6 r^3}.$$
\end{theorem}

Lemma \ref{nestedpoly} allows us to obtain the following corollary from Theorem \ref{r_up_low}. The proof of Corollary \ref{sec3cor} is by now standard and is summarized in Figure \ref{sec3fig4} and its caption.


\begin{corollary}\label{sec3cor}
Fix $\varepsilon\in(0,\mu/2)$. There exist positive constants $C_1, C_2 , C_3, r_0, N_0$ that depend only on $\varepsilon$ such that for for all ${{\rho}} \in [\varepsilon, \mu-\varepsilon]$, 
$N\geq N_0$ and $r\ge r_0$, we have  
$$\mathbb{P}\bigl(Q^{{\rho}}_{0, v_N- rN^{2/3}e_1}\{\tau \geq 1\} \geq e^{-C_1r^2 N^{1/3}}\bigr) \leq e^{-C_2r^{3}}$$
and
$$ \mathbb{E}\Big[Q^{{\rho}}_{0, v_N-rN^{2/3}e_1}\{\tau \geq 1\}  \Big]\leq e^{-C_3 r^3}.$$
The same result holds when $v_N-rN^{2/3}e_1$ is replaced by $v_N+rN^{2/3}e_2$.
\end{corollary}

The next theorem is about the polymer paths having unusually small exit times. 
The estimate improves upon the result from \cite{Bus-Sep-22-ejp} where these types of estimates were used to  rule out the existence of non-trivial bi-infinite polymer measures. This technique was first developed for the non-existence of bi-infinite geodesics in the corner growth model  \cite{Bal-Bus-Sep-20} and subsequently applied to coalescence estimates for semi-infinite geodesics in \cite{Sep-She-20}. 

\begin{theorem}\label{dupper}
Fix $\varepsilon\in(0,\mu/2)$. 
There exist positive constants $C_1,C_2, N_0, \delta_0$ that depend only on $\varepsilon$ such that for all ${{\rho}} \in [\varepsilon, \mu-\varepsilon]$, $N\geq N_0$, $N^{-2/3}<\delta \leq \delta_0$, we have 
\begin{align}\label{dupper:claim1}
    \mathbb{P}\Big(\max_{x\not\in [\![0, v_N]\!]} Q^\rho_{0,x}\{ |\tau| \leq \delta N^{2/3}\} \geq e^{-|\log \delta|^2 \sqrt \delta N^{1/3}} \Big) \leq C_1 |\log\delta|^{10}\delta
    \end{align}
and
\begin{align}\label{dupper:claim2}
C_1\delta\le\mathbb{E}\Big[\max_{x\not\in [\![0, v_N]\!]} Q^\rho_{0,x}\{|\tau|\leq \delta N^{2/3}\} \Big] \leq C_2 |\log \delta|^{10}\delta.
\end{align}
\end{theorem}

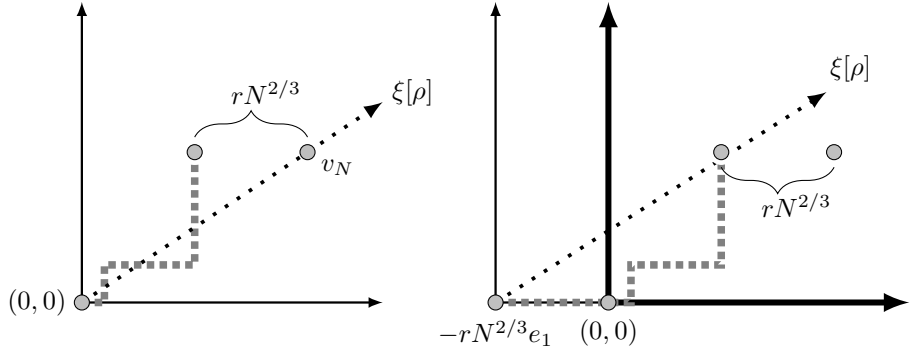
\begin{figure}[t]
\captionsetup{width=0.8\textwidth}
\begin{center}
 
\begin{tikzpicture}[>=latex, scale=1]

\draw[line width = 0.3mm, ->] (1,0)--(1,4);
\draw[line width = 0.3mm, ->] (1,0)--(5,0);

\draw[dotted, color=gray, line width = 1mm] (1,0)--(8.3-7,0) -- (8.3-7,0.5) --(9.5-7, 0.5) -- (9.5-7,2);

\fill[color=white] (2.5,2)circle(1.4mm);

\draw[ line width = 0.3mm, ->] (6.5,0)--(6.5,4);
\draw[line width = 0.3mm] (6.5,0)--(8,0);
\draw[line width = 0.8mm, ->] (8,0)--(12,0);
\draw[line width = 0.8mm, ->] (8,0)--(8,4);

\node at (5.4,2.8) {\small $\xi[{{\rho}}]$};

\draw[dotted, color=gray, line width = 1mm] (6.5,0)--(8.3,0) -- (8.3,0.5) --(9.5, 0.5) -- (9.5,2);

\draw[loosely dotted, line width = 0.5mm, ->] (1,0)--(5,2+2/3);

\draw[loosely dotted, line width = 0.5mm, ->] (6.5,0)--(6.5+0.8*5.5, 0.8*3.5);
\node at (6.5+0.8*5.9, 0.8*3.9) {\small $\xi[{{\rho}}]$};
\fill[color=white] (2.5,2)circle(1.7mm);
\draw[ fill=lightgray](2.5,2)circle(1mm);

\draw [decorate,decoration={brace,amplitude=10pt}, xshift=0pt,yshift=0pt]
(2.5,2+0.2) -- (4,2+0.2) ;

\node at (3.25+0.2,2.8) {\small $rN^{2/3}$};

\draw[ fill=lightgray](4,2)circle(1mm);
\node at (4.4,1.8) {\small $v_N$};

\draw[ fill=lightgray](1,0)circle(1mm);

\draw[ fill=lightgray](6.5,0)circle(1mm);

\draw[ fill=lightgray](11,2)circle(1mm);

\draw[ fill=lightgray](8,0)circle(1mm);



\draw [decorate,decoration={brace,amplitude=10pt, mirror}, xshift=0pt,yshift=0pt]
(6.5+3,-0.2+2) -- (8+3,-0.2+2) ;

\node at (7.5+3,-0.7+2) {\small $rN^{2/3}$};

\fill[color=white] (9.5,2)circle(1.7mm);
\draw[ fill=lightgray](9.5,2)circle(1mm);

\node at (0.4,0) {\small $(0,0)$};

\node at (8,-0.4) {\small $(0,0)$};

\node at (6.5,-0.4) {\small $-rN^{2/3}e_1$};

\end{tikzpicture}
 
\end{center}
\caption{\small  Sketch of Corollary \ref{sec3cor}. On the left is a path in  the event $\tau_{0, v_N- \floor{rN^{2/3}}e_1} \geq 1$.    On the right, a second base point is placed at  $- \floor{rN^{2/3}}e_1$ and the edge weights on the $e_2$-axis based at $0$ are determined by the ratio variables of the polymer based at $- \floor{rN^{2/3}}e_1$. By Lemma  \ref{nestedpoly}, $Q_{0, v_N- \floor{rN^{2/3}}e_1}\{\tau \geq 1\} = Q_{- \floor{rN^{2/3}}e_1, v_N - \floor{rN^{2/3}}e_1 }\{\tau \geq \floor{rN^{2/3}}+ 1\}$, and Theorem \ref{r_up_low} can be applied. }
\label{sec3fig4}
\end{figure}

We close this section by extending the above estimates to any coupling of stationary polymer measures. 
Let $\wt{Q}^\rho_{0, A}$ be any coupling of the measures $\{Q^\rho_{0,x}:x\in A\}$. This is then a probability measure on the product space $\prod_{y\in A} \mathbb{X}_{0, y}$. We view the elements of this product space as vectors and then for $x\in A$, the $x$-th coordinate of such a vector would be the path that ends at $x$. For $x\in A$, define $\{\wt \tau_{0,x} = k\}\subset \prod_{y\in A} \mathbb{X}_{0, y}$ to be the collection of vectors whose $x$-th coordinate is in $\{\tau = k\}$.

\begin{theorem}\label{joint_r}
Fix $\varepsilon \in (0, \mu/2)$. There exist positive constants $C_1, C_2, r_0, c_0, N_0$ that depend only on $\varepsilon$ such that for each $\rho \in [\varepsilon, \mu-\varepsilon]$, $N\geq N_0$ and $r_0 \leq r \leq c_0N^{1/3}$, we have  
$$\mathbb{P}\Big(\wt{Q}^{{\rho}}_{0,\partial^{\textup{NE}}[\![0, v_N]\!]}\Big( \bigcap_{x\in\partial^{\textup{NE}}[\![0, v_N]\!]}\{ |\wt \tau_{0,x}| \geq r N^{2/3}\}\Big) \geq 1-e^{-C_1 r^2 N^{1/3}}\Big)  \geq e^{-C_2 r^3}$$
and
$$ \mathbb{E}\Big[\wt{Q}^{{\rho}}_{0,\partial^{\textup{NE}}[\![0, v_N]\!]}\Big( \bigcap_{x\in\partial^{\textup{NE}}[\![0, v_N]\!]}\{ |\wt \tau_{0,x}| \geq r N^{2/3}\}\Big)\Big]  \geq e^{-C r^3}.$$
\end{theorem}

\begin{theorem}\label{joint_d}
Fix $\varepsilon\in(0, \mu/2)$. There exist positive constants $C, N_0, \delta_0$ that depend only on $\varepsilon$ such that for each $\rho \in [\varepsilon, \mu-\varepsilon]$, $N\geq N_0$, $K \geq 1$ and $0<\delta \leq \delta_0$, we have 
$$\mathbb{E}\Big[\wt{Q}^\rho_{0,\partial^{\textup{NE}}[\![0, v_N]\!]}\Big( \bigcup_{x\in\partial^{\textup{NE}}[\![0, v_N]\!]}\{ |\wt \tau_{0,x}| \leq \delta N^{2/3}\}\Big)\Big] \leq C |\log \delta|^{10}\delta.$$
\end{theorem}

\subsection{Proof of Theorem \ref{r_up_low} }

The expectation bounds in Theorem \ref{r_up_low} follow directly from the tail bounds. We split the proof of the tail bounds  into the following two lemmas. 

\begin{lemma}\label{rtail_ub}
Fix $\varepsilon\in(0,\mu/2)$. There exist positive constants $C_1, C_2$, $r_0$, $N_0$
depending only on $\varepsilon$ such that for all ${{\rho}} \in [\varepsilon, \mu-\varepsilon]$, $N\geq N_0$ and $r\ge r_0$, we have  
\begin{align*}
&\mathbb{P}\Big(Q^\rho_{0, v_N}\{|\tau| > rN^{2/3}\} \geq e^{-C_1 r^2 N^{1/3}}\Big) \leq e^{-C_2r^{3}}.
\end{align*}
\end{lemma}

\begin{lemma}\label{rtail_lb}
Fix $\varepsilon\in(0,\mu/2)$. There exist positive constants $C_1, C_2$, $r_0$, $N_0$, $c_0$ 
depending only on $\varepsilon$ such that for all ${{\rho}} \in [\varepsilon, \mu-\varepsilon]$, $N\geq N_0$ and $r_0 \leq r \leq c_0N^{1/3}$, we have  
\begin{align*}
& \mathbb{P}\Big(\min_{x\not\in [\![0, v_N]\!]} Q^\rho_{0, x}\{|\tau| > rN^{2/3}\} \geq 1-e^{-C_1r^2 N^{1/3}}\Big)\geq e^{-C_2 r^3}
\end{align*}

\end{lemma}

\subsubsection{Proof of Lemma \ref{rtail_ub}}
We start with two calculations for the shape function $\Lambda$. Their proofs use Taylor expansions and are thus postponed to Appendix \ref{exp_calc}.

The first proposition below captures the loss of free energy due to curvature.
\begin{proposition}\label{far_s}
Fix $\varepsilon\in(0,\mu/2)$. There exist positive constants $C_1$, $N_0$, $c_0$
depending only on $\varepsilon$ such that for each $\rho \in [\varepsilon, \mu- \varepsilon]$, 
$N \geq N_0$, $1 \leq s \leq c_0 N^{1/3}$, we have
$$\Lambda\bigl(v_N -\floor{sN^{2/3}}e_1  + \floor{sN^{2/3}}e_2\bigr) - \floor{ sN^{2/3}} \Psi_0(\mu-\rho) + \floor{sN^{2/3}} \Psi_0(\rho)   - \Lambda(v_N)  \leq -C_1 s^2 N^{1/3}.$$
\end{proposition}
The second proposition is essentially a bound on the non-random fluctuation when the endpoint varies around $v_N$.
\begin{proposition}\label{close_s}
Fix $\varepsilon\in(0,\mu/2)$. There exist positive constants $C_1$, $N_0$, $c_0$
depending only on $\varepsilon$ such that for each $\rho \in [\varepsilon, \mu- \varepsilon]$,
$N \geq N_0$, $0\leq s \leq 3$, we have 
$$\Big|\Lambda\bigl(v_N -\floor{sN^{2/3}}e_1  + \floor{sN^{2/3}}e_2\bigr) - \floor{ sN^{2/3}} \Psi_0(\mu-\rho) + \floor{sN^{2/3}} \Psi_0(\rho)  - \Lambda(v_N) \Big| \leq C_1 N^{1/3}.$$
\end{proposition}

With these two propositions, we obtain the following estimate for the maximum free energy. 
\begin{proposition} \label{interval_bd}
For each $\varepsilon\in (0, \mu/2)$, there exist positive constants $C_1$, $C_2, N_0, c_0$ depending on $\varepsilon$ such that for each $N \geq N_0$ and $1\leq r\leq c_0 N^{2/3}$, we have 
$$
\mathbb{P}\Big( \max_{k \in [\![0, 3\floor{N^{2/3}}]\!]} \bigl\{\log Z_{0, v_N + (-k,k)} - \Lambda(v_N + (-k,k))\bigr\} \geq C_1rN^{1/3} \Big)  \leq e^{-C_2r^{3/2}}.
$$
\end{proposition}

\begin{proof}
To start, let us separate the probability that we are trying to bound into two parts.
\begin{align}
&\mathbb{P}\Big( \max_{k \in [\![0, 3\floor{N^{2/3}}]\!]} \bigl\{\log Z_{0, v_N + (-k,k)} - \Lambda(v_N + (-k,k))\bigr\}\geq C'rN^{1/3} \Big)\nonumber\\
&\leq \mathbb{P}\Big( \max_{k \in [\![0, 3\floor{N^{2/3}}]\!]} \bigl\{\log Z_{0, v_N + (-k,k)} - \log Z_{0, v_N} -  [\Lambda(v_N + (-k,k)) - \Lambda(v_N)] \bigr\}\geq \tfrac{C'}{2}rN^{1/3} \Big) \label{rw_est} \\
& \qquad + \mathbb{P}\Big(  \log Z_{0, v_N}  - \Lambda(v_N) \geq \tfrac{C'}{2}rN^{1/3} \Big)\label{ptp_est} 
\end{align}
Using Proposition \ref{up_ub}, $\eqref{ptp_est} \leq e^{-Cr^{3/2}}$. To bound \eqref{rw_est} we reformulate the problem into a bound for  running maxima of random walks. First, by Proposition \ref{close_s}, if $C'\ge 4C_1$ and $r\ge1$, then
\begin{equation}\label{rw_est1}
\eqref{rw_est} \leq \mathbb{P}\Big( \max_{k \in [\![0, 3\floor{N^{2/3}}]\!]} \bigl\{\log Z_{0, v_N + (-k,k)} - \log Z_{0, v_N} -  [-k\Psi_0(\mu-\rho) + k \Psi_0(\rho)] \bigr\}\geq \tfrac{C'}{4}rN^{1/3} \Big).
\end{equation}

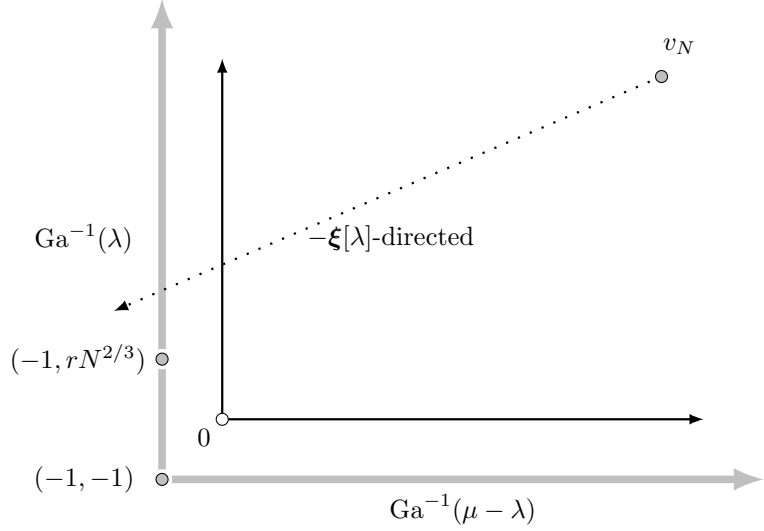
\begin{figure}[t]
\captionsetup{width=0.8\textwidth}
\begin{center}

\begin{tikzpicture}[>=latex, scale=0.8]  

\draw[line width=1mm, lightgray, ->] (-1,-1) -- (-1,7);
\draw[line width=1mm, lightgray, ->] (-1,-1) -- (9,-1);

\draw[line width=0.3mm, ->] (0,0) -- (8,0);
\draw[line width=0.3mm, ->] (0,0) -- (0,6);

\draw[line width=0.3mm, loosely dotted, ->] (6.5+ 0.8,6.5- 0.8) -- (-1.8,1.8);
\node at (2.8,3) {\small$-{\boldsymbol\xi}[\lambda]$-directed};

\draw[ fill=white](0,0)circle(1mm);
\node at (-0.3,-0.3) {\small$0$};

\fill[color=white] (6.5+ 0.8,6.5- 0.8)circle(1.7mm);
\draw[ fill=lightgray](6.5+ 0.8,6.5- 0.8)circle(1mm);
3\node at (6.5+ 0.8+ 0.3,6.5- 0.8+ 0.5) {\small $v_N$};

\fill[color=white] (-1,1)circle(1.7mm);
\draw[ fill=lightgray](-1,1)circle(1mm);
\node at (-2.4,1) {\small $(-1, rN^{2/3})$};

\fill[color=white] (-1,-1)circle(1.7mm);
\draw[ fill=lightgray](-1,-1)circle(1mm);
\node at (-2.3,-1) {\small $(-1,-1)$};

\node at (4,-1.5) {\small $\textup{Ga}^{-1}(\mu - \lambda)$};
\node at (-2.3,3) {\small $\textup{Ga}^{-1}(\lambda)$};

\end{tikzpicture}
	\end{center}	
	\caption{\small The random walk set up in Proposition \ref{interval_bd}.}
	\label{fig_rw}
\end{figure}

Next, we will show that the quantity $\log Z_{0, v_N + (-k,k)} - \log Z_{0, v_N}$ can be compared to a random walk with i.i.d.~steps. To do this, we will place boundary weights on the south-west boundary of $(-1, -1)+\mathbb Z_{\ge0}^2$ with parameters $\lambda = \rho- q_0\sqrt{r}N^{-1/3}$ and $\mu-\lambda$. Here, $q_0$ will be fix large so that the situation from Figure \ref{fig_rw} happens. Then the $c$ from the statement of our theorem can be now fixed sufficiently small so that $\lambda$ stays between $(0, \mu)$. These choices depend only on $\varepsilon$. 

Because $v_N$ 
is far away from $(-1,-1)$ on the scale $N^{2/3}$, by a similar argument to Corollary \ref{sec3cor}, we have 
\begin{equation}\label{exit_rw}
\mathbb{P}(Q^\lambda_{(-1,-1), v_N} \{\tau\geq 1 \} \geq 1/10) \leq e^{-Cr^3}.
\end{equation}
Let us denote the complement of the event above as 
$$A = \big\{Q^\lambda_{(-1,-1), v_N} \{\tau\leq -1 \} \geq 9/10\big\}$$
In the calculation below, let $Z^{\lambda,\textup{W}}_{(-1,0), x}$ denote the partition function for up-right paths from $(-1,0)$ to $x$, which uses the same weights as $Z^{\lambda,}_{(-1,-1),x}$ does on the west boundary but uses the original (bulk) weights on $\Z_{\ge0}^2$. For each $i = 0,1, \dots, 3\floor{N^{2/3}}-1$, we have 
\begin{align*}
&e^{\log Z_{0, v_N + (-i-1, i+1)} - \log Z_{0, v_N + (-i, i)}} \\
&= \frac{Z_{0, v_N + (-i-1, i+1)}}{ Z_{0, v_N + (-i, i)}} \\
&\leq \frac{Z^{\lambda, \textup{W}}_{(-1,0), v_N + (-i-1, i+1)}}{ Z^{\lambda, \textup{west}}_{(-1,0), v_N + (-i, i)}}
 = \frac{Z^{\lambda, \textup{W}}_{(-1,0), v_N + (-i-1, i+1)}}{ Z^{\lambda, \textup{W}}_{(-1,0), v_N + (-i, i)}}\cdot \frac{I^{\lambda}_{[\![(-1,-1), (-1,0)]\!]}}{I^{\lambda}_{[\![(-1,-1), (-1,0)]\!]}} \qquad \textup{by Proposition \ref{mono_ratio}}\\
 &= \frac{Z^{\lambda }_{(-1,-1), v_N + (-i-1, i+1)}(\tau \leq -1)}{ Z^{\lambda}_{(-1,-1),v_N + (-i, i)}(\tau \leq -1)}
= \frac{Q^{\lambda }_{(-1,-1), v_N + (-i, i)}(\tau\leq  -1)}{Q^{\lambda }_{(-1, -1),v_N + (-i, i)}(\tau\leq  -1)}\cdot 
\frac{Z^{\lambda}_{(-1,-1), v_N + (-i-1, i+1)}}{Z^\lambda_{(-1,-1),v_N + (-i, i)}}\\
& \leq \frac{10}{9}\frac{Z^{\lambda }_{(-1, -1), v_N + (-i-1, i+1)}}{Z^{\lambda }_{(-1, -1),v_N + (-i, i)}} \qquad \textup{ on the event $A$ } .
\end{align*}
By Theorem \ref{stat}, we can define $$S^\lambda_k = \sum_{i=1}^{k-1} \log \frac{Z^{\lambda }_{(-1, -1), v_N + (-i-1, i+1)}}{Z^{\lambda }_{(-1, -1),v_N + (-i, i)}}$$ 
which is an i.i.d.~random walk whose step has the same distribution as $\log  G_1-\log G_2$, where $G_1$ and $G_2$ are independent, respectively, $\textup{Ga}(\mu-\lambda)$ and $\textup{Ga}(\lambda)$ random variables. 
And we have
\begin{equation}\label{RW_est2}
\eqref{rw_est1} \leq \mathbb{P}\Big(\max_{k \in [\![0, 3\floor{N^{2/3}}]\!]}\bigl\{ S^\lambda_k -  [k\Psi_0(\mu-\rho) - k \Psi_0(\rho)] \bigr\}\geq \tfrac{C'}{8}rN^{1/3} \Big) + \mathbb{P}(A^c),
\end{equation}
where $\mathbb{P}(A^c) \leq e^{-Cr^3}$. Note $\mathbb{E}[S^\lambda_k] = k\Psi_0(\mu-\lambda) - k \Psi_0(\lambda)$, and using Taylor expansion and the fact that $k \leq 3N^{2/3}$, we have
$$ \Big|\mathbb{E}[S^\lambda_k] - [k\Psi_0(\mu-\rho) - k \Psi_0(\rho)]\Big| \leq  C\sqrt{r}N^{1/3}.$$
Finally, taking $C'\ge16C$, the probability in \eqref{RW_est2} is bounded as follows
\begin{align*}
&\mathbb{P}\Big(\max_{k \in [\![0, 3\floor{N^{2/3}}]\!]} \bigl\{S^\lambda_k -  [k\Psi_0(\mu-\rho) - k \Psi_0(\rho)]\bigr\} \geq \tfrac{C'}{8}rN^{1/3} \Big)\\
&\qquad\leq \mathbb{P}\Big(\max_{k \in [\![0, 3\floor{N^{2/3}}]\!]} \bigl\{S^\lambda_k -  \mathbb{E}[S^\lambda_k] \bigr\}\geq \tfrac{C'}{16}rN^{1/3} \Big) \leq e^{-C''r^{3/2}}
\end{align*}
where the last inequality follows from Theorem \ref{max_sub_exp}.
\end{proof}

With this result, we are ready to prove Lemma \ref{rtail_ub}. The proof uses arguments for a stationary polymer with an antidiagonal boundary instead of a southwest boundary, which we will now define. Let $\mathcal{S}_{(0,0)}$ be the bi-infinite staircase paths (with alternating $e_1$ and $-e_2$ steps) through $(0,0)$ 
\begin{equation}\label{stair_S}
\mathcal{S}_{(0,0)} = \{\dots, (-1, 1), (-1, 0), (0,0), (0,-1), (1,-1), \dots\}.
\end{equation}
Next, we attach boundary weights along $\mathcal{S}_{(0,0)}$, which are all independent. For each horizontal edge to the left and right of $(0,0)$, we attach $\text{Ga}(\mu-\rho)$ and $\text{Ga}^{-1}(\mu-\rho)$ weights. For each vertical edge to the left and  right of $(0,0)$, we attach $\text{Ga}^{-1}(\rho)$ and $\text{Ga}(\rho)$ weights. For $k \in \mathbb{Z}$, let $H_k$ denote the product of the edge weights from  $\mathcal{S}_{(0,0)}$ between $(0,0)$ and $(k, -k)$. 

The partition function for this polymer with antidiagonal boundary is defined by 
$$
Z^{\rho,\textup{dia}}_{0, x} = \sum_{k \in \mathbb{Z}} H_k \cdot \wt{Z}_{(k, -k), x} 
$$
where $\wt Z$ is the point-to-point partition but without using the weight at its starting point. The corresponding polymer measure $Q^{\rho, \textup{dia}}_{0,x}$ is a probability measure on paths that start at $0$, move along the antidiagonal, taking either only $e_1-e_2$ steps or only $e_2-e_1$ steps, and then enter the bulk by taking an $e_1$ or $e_2$ step, after which they only take steps in $\{e_i,i=1,2\}$.
For such a path $\gamma$, we define $\tau^{\textup{dia}}(\gamma)\in\mathbb{Z}\setminus \{0\}$ as the signed number of steps taken before entering the bulk, where the plus sign corresponds to $e_1-e_2$ steps and the minus sign to $e_2-e_1$ steps.  
For $k \in \mathbb{Z}$, let us define the partition function over paths with exit point $k$ as 
\begin{equation}\label{dia_exit}Z^{\rho, \textup{dia}}_{0, x}(\tau^{\textup{dia}} = k) = H_k \cdot \wt{Z}_{(k, -k), x}.
\end{equation}

\begin{proof}[Proof of Lemma \ref{rtail_ub}]
First, by Lemma \ref{dia_vs_sw}, it suffices to prove our estimate for the stationary polymer with the antidiagonal boundary defined above.
By a slight abuse of notation, let us denote $Z^\rho = Z^{\rho, \textup{dia}}$, and $Q^\rho = Q^{\rho, \textup{dia}}$. There is no confusion since we will only be working with the antidiagonal boundary in the remainder of this proof (instead of southwest boundary). 

By a union bound, it suffices to prove that there exist positive constants $C_1, C_2, s_0, c_0$ such that for each $N\geq N_0$ and $s_0 \leq s \leq c_0 N^{1/3}$, we have
$$\mathbb{P}\Big(\max_{kN^{-2/3} \in (s, s+1] }Q^\rho_{0, v_N}\{\tau^{\textup{dia}} = k\} \geq e^{-C_1 s^2 N^{1/3}}\Big) \leq e^{-C_2 s^3}.$$

To show this, we rewrite the quenched probability above in terms of the free energies, 
\begin{align}
&\mathbb{P}\Big(\log Z^\rho_{0, v_N} - \max_{kN^{-2/3} \in (s, s+1] } \log Z^\rho_{0, v_N}\{\tau^{\textup{dia}} = k\}  \leq C' s^2 N^{1/3}\Big)\nonumber\\
& \le \mathbb{P}\Big(\Big[\log Z^\rho_{0, v_N} - \Lambda(v_N)\Big] \label{negative_term}\\
& \qquad\quad - \max_{kN^{-2/3} \in (s, s+1] }\Big[\log Z^\rho_{0, v_N}\{\tau^{\textup{dia}} = k\} - \bigl(\Lambda(v_N + (-k, k)) - k \Psi_0(\mu-\rho) + k\Psi_0(\rho)\bigr)  \Big] \nonumber\\
& \qquad \qquad  \leq C' s^2N^{1/3} +  \max_{kN^{-2/3} \in (s, s+1] }\bigl(\Lambda(v_N + (-k, k)) - k \Psi_0(\mu-\rho) + k\Psi_0(\rho) -  \Lambda(v_N)\bigr) \Big).\nonumber
\end{align}
Applying Proposition \ref{far_s}, if we fix $C'$ in \eqref{negative_term} sufficiently small, then, we may replace the right side of the inequality in \eqref{negative_term} by $ - c's^2 N^{1/3}$ for some small positive constant $c'$. 

Let $\{Z_i\}_{i=1}^\infty$ denote a sequence of i.i.d.\ random variables with the same distribution given by $-\log  G_1 + \log G_2$, where $G_1$ and $G_2$ are independent, respectively, $\textup{Ga}(\mu-\rho)$ and $\textup{Ga}(\rho)$ random variables. The $Z_i$'s will play the role of the boundary weight at $(i,-i)$, $i\ge1$.
Now continuing with a union bond, we have
\begin{align}
\eqref{negative_term} 
&\leq \mathbb{P}\bigl( \log Z^\rho_{0, v_N} - \Lambda(v_N) \leq - \tfrac{1}{5}c's^2 N^{1/3}\bigr)\label{term_1}\\
 & \qquad + \mathbb{P}\Big(\max_{kN^{-2/3} \in (s, s+1]} \Big(\sum_{i=1}^{k} Z_i + k \Psi_0(\mu-\rho) - k\Psi_0(\rho)\Big)  \geq \tfrac{1}{5} c' s^2 N^{1/3} \Big)\label{term_2}\\
 & \qquad \qquad  + \mathbb{P}\Big(\max_{kN^{-2/3} \in (s, s+1]}  \bigl(\log \wt Z_{(-k, k) , v_N} - \Lambda(v_N + (-k, k)\bigr)\geq \tfrac{1}{5} c' s^2 N^{1/3} \Big)\label{term_3},
\end{align}
and  $\eqref{term_1} \leq e^{-Cs^3}$ by Proposition \ref{low_ub}, $\eqref{term_2} \leq e^{-Cs^3}$ by Proposition \ref{Ga_sub_exp} and Theorem \ref{max_sub_exp}, $ \eqref{term_3} \leq e^{-Cs^3}$ by Proposition \ref{interval_bd}.
Finally, we note that even the $\wt Z$ free energy does not use the first weight, but Proposition \ref{interval_bd} (which was originally stated for $Z$ instead of $\wt Z$) still applies since using a union bound we can get $P(\max_{0\leq k \leq 3N^{2/3}}\log Y_{(-k,k)}\ge\varepsilon s^2N^{1/3})\le N^{2/3}e^{-cs^2N^{1/3}}\le Ce^{-s^2N^{1/3}}\le C'e^{-c's^3}$. Also, we are  applying Proposition \ref{interval_bd} by first shifting the picture to move the $v_N$ in \eqref{rw_est1} to the origin, flipping it about the antidiagonal, and then using, in the proposition, a $v_N$ that is not exactly the $v_N$ in the lemma, but rather $v_N-sN^{2/3}(1,-1)$. This is allowed because the proposition is stated uniformly for a whole interval of characteristic directions.
\end{proof}

\subsubsection{Proof of Lemma \ref{rtail_lb}}
To prove Lemma \ref{rtail_lb},
we tilt the probability measure to make the event likely and pay for this with a bound on  the Radon-Nikodym derivative. This argument was introduced in \cite{Bal-Sep-10} in the context of the  asymmetric simple exclusion process and later adapted to lower bound proofs of the  longitudinal  fluctuation exponent \cite{Sep-18} and large exit time probability \cite{Sep-She-20} in the stationary last-passage percolation process. The key idea here is to perturb the parameter ${{\rho}}$ of the stationary polymer model to   ${{\rho}} \pm rN^{-1/3}$. This allows us to control the exit point on the scale $N^{2/3}$.  The general  idea of utilizing perturbations of order $N^{-1/3}$ goes back to the seminal paper \cite{Cat-Gro-06}. 
We now give the details.

For $v\in\mathbb Z^2_{\ge0}$ let $\partial^{\textup{NE}}[\![0,v]\!]$ denote the north-east boundary of the rectangle $[\![0,v]\!]$, i.e.\ the sites $u\in [\![0,v]\!]$ with $u\cdot e_1=v\cdot e_1$ or $u\cdot e_2=v\cdot e_2$.

Note that it is enough to prove the claimed bound with 
$\min_{x\not\in [\![0, v_N]\!]} Q^\rho_{0,x}\{|\tau|> rN^{1/3}\}$ replaced by 
\[\min_{x\in \partial^{\textup{NE}}[\![0, v_N]\!]} Q^\rho_{0, x}\{|\tau| > rN^{2/3}\},\]
since
\begin{align*}
&\min_{x\not\in [\![0, v_N]\!]} Q^\rho_{0,x}\{|\tau|> rN^{2/3}\} \\
&\qquad =\min_{x\not\in [\![0, v_N]\!]} \sum_{z\in \partial^{\textup{NE}} \lzb 0, v_N \rzb} Q^\rho_{0,x}\{|\tau|> rN^{2/3} \text{ and the path passes through z}\} \\
&\qquad =\min_{x\not\in  [\![0, v_N]\!]} \sum_{z\in \partial^{\textup{NE}} [\![0, v_N]\!]} Q^\rho_{0,z}\{|\tau|>rN^{2/3} \} Q_{0,x}^\rho\{\text{path passes through $z$}\}\\
&\qquad \geq\min_{x\not\in [\![0, v_N]\!]} \sum_{z\in \partial^{\textup{NE}} [\![0, v_N]\!]} \Big(\min_{z'\in \partial^{\textup{NE}} \lzb 0, v_N \rzb} Q^\rho_{0,z'}\{|\tau|>rN^{2/3} \}\Big) Q^\rho_{0,x}\{\text{passes through $z$}\} \\
&\qquad=\min_{z'\in \partial^{\textup{NE}} [\![0, v_N]\!]} Q^\rho_{0,z'}\{|\tau|>rN^{2/3}\}.
\end{align*}

Take $c\in(0,\frac{\varepsilon}{4\mu^2}\wedge\frac12]$, with $\varepsilon$ as in the statement of the theorem. Below, we will choose an exact value for $c$, which  will still only depend on $\varepsilon$ (and $\mu$).

Given positive $r$ and $N$, define the perturbed parameters  $\lambda = {{\rho}} + rN^{-1/3}$ and $\eta = {{\rho}} - rN^{-1/3}$.
The choice of $c$ guarantees that if
\begin{equation}\label{fix_r}r \leq c((\mu-{{\rho}})^2\wedge {{\rho}}^2) N^{1/3},\end{equation}
then  $\eta < {{\rho}} <\lambda$ are all contained in  $[\varepsilon/2, \mu-\varepsilon/2]$.


Given positive constants $a<b$, define a new environment $\wt{\mathbb{P}}$ by changing the original boundary weights (whose distribution we will denote by $\mathbb{P}^\rho$) on parts of the axes. Precisely, $\wt{\mathbb{P}}$ is the joint distribution, under $\mathbb{P}^\rho$ of
\begin{align*}
\wt\omega_{ke_1} &\sim \text{Ga}^{-1}(\mu-\lambda)  & &\text{for } k \in [\![\floor{arN^{2/3}}+ 1,  \floor{brN^{2/3}}]\!]\\
\wt\omega_{ke_2} & \sim \text{Ga}^{-1}({\eta}) & &\text{for } k \in [\![\floor{arN^{2/3}} + 1,  \floor{brN^{2/3}}]\!]\\
\wt\omega_{z} &\sim \omega_z & &\text{for all other $z\in \mathbb{Z}^2_{\geq 0}$.}
\end{align*}
The $\wt\w$ weights in the first two lines are all independent and independent of the $\w$ weights.
The exact values of $a$ and $b$ will be determined further down and will only depend on  $\varepsilon>0$ (and $c$). Essentially, they will be chosen so that, 
in the picture in the left panel of 
Figure \ref{fig21}, the two thick dotted lines passing through $v_N$ and having slopes $\xi[\lambda]$ and $\xi[\eta]$ rest inside the highlighted regions on the axes. 
Then, under the new random environment $\wt{\mathbb{P}}$, we will show that there exists some constant $C_1$ such that for $N$ and $r$ large,
\begin{equation}\label{modest}
\wt{\mathbb{P}}\Big(\min_{x\in \partial^{\textup{NE}}[\![0, v_N]\!]} Q_{0, x}\{|\tau| > arN^{2/3}\} \geq 1-e^{-C_1 r^2N^{1/3}}\Big) \geq 1/2. 
\end{equation}
We finish the proof of the theorem, assuming this inequality.
Denote the event inside \eqref{modest} by $S$ and let $f= \frac{d \wt{\mathbb{P}}}{d\mathbb{P}^{\rho}}$,
where $\mathbb P^\rho$ is the marginal of $\mathbb P$, i.e.\ a the probability measure with independent $\rho$-boundary weights and bulk weights. By the Cauchy-Schwarz inequality, we have 
$$1/2 \leq \wt{\mathbb{P}}(S) = \mathbb{E}^{\rho}[\mathbbm{1}_S f] \leq \mathbb{P}^{\rho}(S)^{1/2} \mathbb{E}^{\rho}[f^2]^{1/2} \leq \mathbb{P}^{\rho}(S)^{1/2} e^{Cr^3},$$
where the last bound for the second moment of $f$ follows from Proposition \ref{RNest}. This implies 
\begin{equation}\label{with_a}
\mathbb{P}\Big(\min_{x\in \partial^{\textup{NE}} [\![0, v_N]\!]} Q^{{{\rho}}}_{0, x}\{|\tau| > arN^{2/3}\} \geq 1-e^{-C_1r^2N^{1/3}}\Big) \geq e^{-C_2r^{3}}.
\end{equation}
To recover the statement of our theorem without the constant $a$ in \eqref{with_a}, just modify $C_1$ and $C_2$.

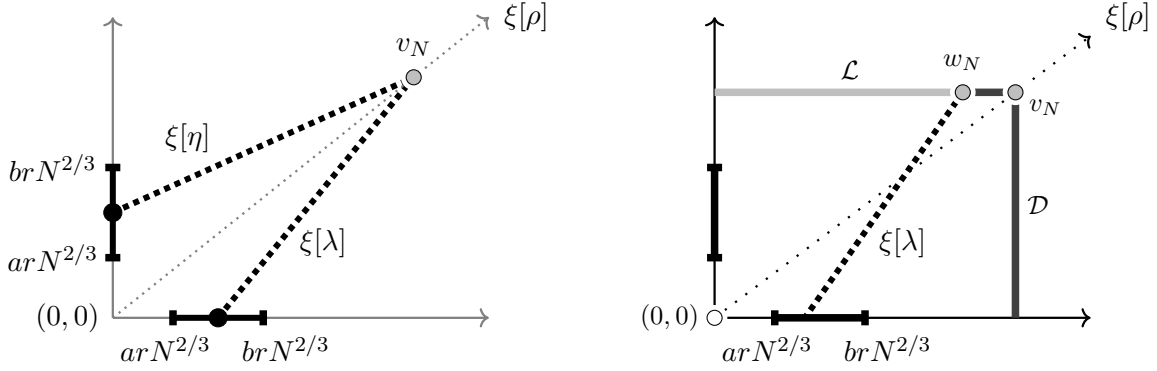
\begin{figure}[t]
\captionsetup{width=0.9\textwidth}
\begin{center}
\begin{tikzpicture}

\draw[gray, line width=0.3mm, ->] (0,0) -- (5,0);
\draw[gray, line width=0.3mm, ->] (0,0) -- (0,4);

\draw[ line width=0.8mm] (0.8,0) -- (2,0);
\draw[ line width=0.8mm] (0,0.8) -- (0,2);

\draw[ line width=1mm] (-0.1,0.8) -- (0.1,0.8);
\draw[ line width=1mm] (-0.1,2) -- (0.1,2);

\draw[ line width=1mm] (0.8,-0.1) -- (0.8,0.1);
\draw[ line width=1mm] (2,-0.1) -- (2,0.1);

\draw[dotted, line width=0.8mm] (1.4,0) -- (0.8*5, 0.8*4);
\draw[dotted, line width=0.8mm] (0,1.4) -- (0.8*5, 0.8*4);

\draw[gray ,dotted, line width=0.3mm, ->] (0,0) -- (5,4) ;
\fill[color=white] (0.8*5, 0.8*4)circle(1.7mm); 
\draw[ fill=lightgray](0.8*5, 0.8*4)circle(1mm);
\node at (0.8*5, 0.8*4+0.4) {$v_N$};

\node at (0.7, -0.4) {$arN^{2/3}$};
\node at (2.3, -0.4) {$brN^{2/3}$};

\node at (-0.8, 0.8) {$arN^{2/3}$};
\node at (-0.8, 2.0) {$brN^{2/3}$};

\node at (2.8,1) {$\xi[\lambda]$};
\node at (1,2.4) {$\xi[\eta]$};

\node at (-0.6, 0) {$(0,0)$};

\draw[ fill=black](1.4, 0)circle(1.2mm);
\draw[ fill=black](0, 1.4)circle(1.2mm);

\node at (5.5, 4) {$\xi[{{\rho}}]$};

----------------------------------------


\draw[line width = 0.3mm, ->] (0+8,0)--(0+8,4);
\draw[line width = 0.3mm, ->] (0+8,0)--(5+8,0);

\draw[line width = 1mm, ] (0.8+8,0)--(2+8,0);
\draw[line width = 1mm, ] (0.8+8,-0.1)--(0.8+8,0.1);
\draw[line width = 1mm, ] (2+8,-0.1)--(2+8,0.1);

\draw[line width = 1mm, ] (0+8,0.8)--(0+8,2);
\draw[line width = 1mm, ] (-0.1+8,0.8)--(0.1+8,0.8);
\draw[line width = 1mm, ] (-0.1+8,2)--(0.1+8,2);

\draw[dotted, line width = 0.8mm, ] (3.3+8,3)--(1.2+8,0);

\draw[color=lightgray, line width = 1mm] (0+8,3) -- (3.3+8, 3);
\draw[ color = darkgray, line width = 1mm] (3.3+8,3) -- (4+8, 3) -- (4+8,0);

\draw[loosely dotted, line width = 0.3mm, ->] (0+8,0) -- (5+8, 3+3/4);

  (2.5+8,3.1) -- (4+8,3.1);

\node at (0.7+8, -0.4) {$arN^{2/3}$};
\node at (2.3+8, -0.4) {$brN^{2/3}$};

\fill[color=white] (3.3+8,3)circle(1.7mm);
\draw[ fill=lightgray](3.3+8,3)circle(1mm);
\node at (3.3+8,3.4) {\small $w_N$};
\node at (1.8+8,3.3) {\small $\cL$};
\fill[color=white] (0+8,0)circle(1.7mm);
\draw[ fill=white](0+8,0)circle(1mm);
\node at (-0.6+8,0) {\small $(0,0)$};

\fill[color=white] (4+8,3)circle(1.7mm);
\draw[ fill=lightgray](4+8,3)circle(1mm);
\node at (4.4+8,2.8) {\small $v_N$};
\node at (4.3+8,1.5) {\small $\cD$};

\node at (5.5+8, 4) {$\xi[{{\rho}}]$};

\node at (2.5+8,1) {$\xi[\lambda]$};

\end{tikzpicture}
\end{center}
\caption{\small {\it Left:} Two dotted lines have slopes $\xi[\lambda]$ and $\xi[\eta]$. \textit{Right:}  
Decomposition of the north and east boundaries of $\lzb 0, v_N\rzb $ into   regions   $\cL$ (light gray)  and  $\cD$ (dark gray). A small perturbation of $v_N$ to $w_N$ keeps the endpoint of the $-\xi[\lambda]$ ray from $w_N$ in the interval $[arN^{2/3}, brN^{2/3}]$.}
\label{fig21}
\end{figure}

Next, we will show \eqref{modest} which will finish the proof of the theorem. 
To do this, we will show that for $r$ and $N$ large, we have
\begin{equation}\label{modestc}
\wt{\mathbb{P}}\Big(\max_{x\in \partial^{\textup{NE}}[\![0, v_N]\!]} Q_{0, x}\{1\leq \tau \leq arN^{2/3}\} < e^{-C'r^2N^{1/3}}\Big) \geq 1- Cr^{-3}.
\end{equation}
This and the similar estimate for the event $\{-1 \geq \tau_{0, x} \geq - arN^{2/3}\}$ imply \eqref{modest} when $r$ is taken large.
Note that here we will pick the values of $a<b$ and $c\in(0,\frac{\varepsilon}{4\mu^2}\wedge\frac12]$ only for the bound \eqref{modestc}. When applying the same argument to the other case, we obtain another set of constants $a'<b'$ and $c'\in(0,\frac{\varepsilon}{4\mu^2}\wedge\frac12]$ which are possibly different. Then, we replace $a$ and $a'$ by $a\wedge a'$, $b$ and $b'$ by $b\vee b'$, and $c$ and $c'$ by $c\wedge c'$.

Recall the perturbed parameter $\lambda = {{\rho}} + rN^{-1/3}$. If $c\in(0,\frac{\varepsilon}{4\mu^2}\wedge\frac12]$ and $r$ and $N$ satisfy condition \eqref{fix_r}, then $\lambda$ satisfies
\beq \label{lambda} \varepsilon/2<{{\rho}}< \lambda  \leq  {{\rho}} + c((\mu-{{\rho}})^2\wedge {{\rho}}^2)\le\mu-\varepsilon/2. \eeq
We estimate the difference of the reciprocal slopes (i.e. $\frac{\text{change of $x$}}{\text{change of $y$}}$) of the vectors $\xi[\lambda]$ and $\xi[{{\rho}}]$. By definition
$$\frac{\xi[\lambda]\cdot e_1}{\xi[\lambda]\cdot e_2} - \frac{\xi[{{\rho}}]\cdot e_1}{\xi[{{\rho}}]\cdot e_2}  =  \frac{\Psi_1({{\rho}} + rN^{-1/3})}{\Psi_1(\mu-{{\rho}} - rN^{-1/3})} - \frac{\Psi_1({{\rho}}) }{\Psi_1(\mu- {{\rho}} )}.
$$
Since $\Psi_1$ is smooth and takes positive values on compact intervals strictly contained inside $(0, \mu)$, we can Taylor expand the quotient $g(z) = \frac{\Psi_1({{\rho}} +z ) }{\Psi_1(\mu- {{\rho}} - z)}$ around $z=0$. 
This gives

\begin{equation}\label{taylor}\Big| \Big(\frac{\xi[\lambda]\cdot e_1}{\xi[\lambda]\cdot e_2} - \frac{\xi[{{\rho}}]\cdot e_1}{\xi[{{\rho}}]\cdot e_2}  \Big) - (-k_1 rN^{-1/3}) \Big| \leq k_2 r^2 N^{-2/3},
\end{equation}
for all $\rho$ and $\lambda$ such that $\varepsilon/2<\rho<\lambda<\mu-\varepsilon/2$.
Here, $k_1$ and $k_2$ are positives constant depending only on $\rho$, $\mu$, and $\varepsilon$. 
Take $c\in(0,\frac{\varepsilon}{4\mu^2}\wedge\frac12]$ to satisfy
\begin{equation}\label{fixc}
c  \leq \frac{1}{100} \frac{k_1}{k_2}\,.
\end{equation}
Then, for $r$ and $N$ satisfying \eqref{fix_r},
\begin{equation}\label{fixc_2}
k_2r^2 N^{-2/3} < \frac{1}{10} k_1rN^{-1/3}.\end{equation} 
And from \eqref{taylor} and \eqref{fixc_2} above, we obtain
\begin{equation}\label{slopeest} -2k_1 rN^{-1/3} \leq  \frac{\xi[\lambda]\cdot e_1}{\xi[\lambda]\cdot e_2} - \frac{\xi[{{\rho}}]\cdot e_1}{\xi[{{\rho}}]\cdot e_2}   \leq -\tfrac{1}{2}k_1 rN^{-1/3}.\end{equation}

Now, start two rays at $(0,0)$ in the directions $\xi[{{\rho}}]$ and $\xi[\lambda]$ and let $u_N$ be the lattice point closest to the $\xi[\lambda]$-directed ray such that $u_N \cdot e_2 = v_N \cdot e_2$. (See the right panel of Figure \ref{fig99}.) Then \eqref{slopeest}  implies that there exist two fixed positive constants $l_1, l_2$ depending only on $\rho$, $\mu$, and $\varepsilon$ such that 
\beq \label{cbound} l_1 rN^{2/3} \leq v_N\cdot e_1 - u_N \cdot e_1 \leq l_2 rN^{2/3}. \eeq
For now, we define 
$$a = \tfrac{1}{10} l_1 \quad\text{ and } \quad b = 10 l_2,$$
and note that the above value of $a$ will be lowered if necessary, later in the argument.

\begin{figure}[t]
\captionsetup{width=0.8\textwidth}
\begin{center}
\begin{tikzpicture}[scale = 1]

\draw[gray, line width=0.3mm, ->] (0+8,0) -- (5+8,0);
\draw[gray, line width=0.3mm, ->] (0+8,0) -- (0+8,4);

\node at (1.5, 2.7) {$\xi[\eta]$};

\node at (3, 1.2) {$\xi[\lambda]$};

\draw[gray ,loosely dotted, line width=0.5mm, ->] (0+8,0) -- (1*5+8, 1*4);

\draw[gray ,loosely dotted, line width=0.5mm, ->] (0+8,0) -- (5-1.4/0.8+8, 4) ;
\fill[color=white] (0.8*5+8, 0.8*4)circle(1.7mm); 
\draw[ fill=lightgray](0.8*5+8, 0.8*4)circle(1mm);
\node at (0.8*5+8, 0.8*4+0.4) {$v_N$};

\fill[color=white] (0.8*5- 1.4+8, 0.8*4)circle(1.7mm); 
\draw[ fill=lightgray](0.8*5- 1.4+8, 0.8*4)circle(1mm);
\node at (0.8*5- 1.5 +8, 0.8*4+0.4) {$u_N$};

\node at (5+0.4+8, 4-0.3) {$\xi[{{\rho}}]$};
\node at (5-1.4/0.8+0.1+8, 4+0.4){$\xi[\lambda]$};

\node at (0-0.6+8, 0) {$(0,0)$};

\draw[black ,dotted, line width=0.3mm] (1.4 +8,0) -- (0.8*5+8, 0.8*4);
\draw[ fill=black](1.4+8, 0)circle(1.3mm);

----------------------------------

\draw[gray, line width=0.3mm, ->] (-0,0) -- (5+0,0);
\draw[gray, line width=0.3mm, ->] (-0,0) -- (0,4);

\draw[gray ,dotted, line width= 1mm] (-0,0)--(1.1,0) --(1.1,0.6)--(1.7,0.6)--(1.7,0.8)--(1.7,1.1) --(2.1,1.1)--(2.1,1.5)--(2.9,1.5) -- (2.9,1.9)--(3.5,1.9)--(3.5,2.7)--(3.5,3.2);

\draw[ line width=1mm] (0.54+0,-0.1) -- (0.54+0,0.1);
\draw[ line width=1mm] (2+0,-0.1) -- (2+0,0.1);

\draw[black ,dotted, line width=0.3mm]  (1.4+0,0) -- (0.8*5+0, 0.8*4);
\draw[black ,dotted, line width=0.3mm] (1.4+0-0.5 ,0) -- (0.8*5+0-0.5, 0.8*4);


\draw[ fill=lightgray](0.8*5+0, 0.8*4)circle(1mm);
\node at (0.8*5.7, 0.8*4) {$v_N$};

\draw[ fill=lightgray](0.8*5+0-0.5, 0.8*4)circle(1mm);
\node at (0.8*5+0-0.6, 0.8*4+0.4) {$w_N$};

\node at (0.1+0, -0.4) {${arN^{2/3}}$};
\node at (2.5+0, -0.4) {${brN^{2/3}}$};


\node at (0-0.6, 0) {$(0,0)$};

\draw[ fill=black](1.4+0, 0)circle(1.3mm);

\draw[ fill=white](0.9+0, 0)circle(1.3mm);

\end{tikzpicture}
\end{center}
\caption{\small \textit{Left:} The dotted lines have characteristic slope $\xi[\lambda]$. Consequently, with high probability, the sampled $\lambda$ polymer from $0$ to $w_N$ exits through the interval $[\![ arN^{2/3}e_1, brN^{2/3}e_1]\!]$.  {\it Right:}  Illustration of estimate \eqref{cbound}.}
\label{fig99}
\end{figure}
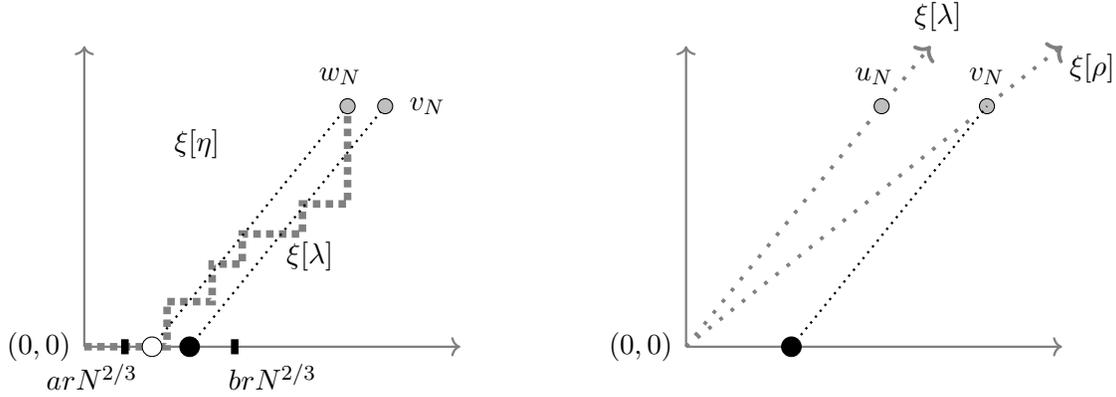

Fix a positive constant $q\leq \tfrac{1}{10}l_1$, let us define 
\begin{equation}\label{define_w_N}
w_N = v_N - \floor{q rN^{2/3}}e_1.
\end{equation}
As shown on the right of Figure \ref{fig21}, the point $w_N$ splits $\partial^{\textup{NE}}[\![0, v_N]\!]$ into  the dark region $\mathcal{D}$ and the light region $\mathcal{L}$.
We will first work with the dark region and show 
\begin{equation}\label{darkest}
\wt{\mathbb{P}}\Big(\max_{x\in \mathcal{D}} Q_{0, x}\{1\leq \tau \leq arN^{2/3}\} \leq e^{-C'r^2N^{1/3}}\Big) \geq 1- Cr^{-3}.
\end{equation}

Let us look at another polymer measure $R_{0,x}$ which is restricted to paths that start with an $e_1$ step from the origin, then
\begin{align*}
R_{0, x}\{1\leq\tau \leq arN^{2/3}\} &= \frac{Z_{0,x}(1\leq \tau\leq \floor{ arN^{2/3}})}{Z_{0,x}(1 \leq \tau)}.
\end{align*}
From the following three facts,
\begin{itemize}
\item $ Q_{0, x}\{1\leq\tau \leq arN^{2/3}\} \leq  R_{0, x}\{1\leq \tau \leq arN^{2/3}\},$
\item $R_{0, x}\{1\leq \tau \leq arN^{2/3}\} + R_{0, x}\{ \tau > arN^{2/3}\} = 1$,
and 
\item by Lemma \ref{2_exit_ineq}, $ R_{0, w_N}\{ \tau > arN^{2/3}\} \leq R_{0, x}\{ \tau > arN^{2/3}\} \text{ for each $x \in \mathcal{D}$}$,
\end{itemize}
we have 
$$Q_{0, x}\{1\leq \tau \leq arN^{2/3}\} \leq R_{0, w_N}\{1\leq \tau\leq arN^{2/3}\}\qquad \text{ for each $x \in \mathcal{D}$}.$$
Thus in order to show \eqref{darkest}, it suffices to show 
\begin{equation} \label{wnest}
\wt{\mathbb{P}}\Big( R_{0, w_N}\{1\leq \tau \leq arN^{2/3}\} \leq e^{-C'r^2 N^{1/3}}\Big) \geq 1- Cr^{-3} .
\end{equation}
To show \eqref{wnest}, we will find a high probability event 
$$A = A_1 \cap A_2 \cap A_3 \cap A_4$$ with $\wt{ \mathbb{P}}(A) \geq 1-Cr^{-3}$ such that on $A$,
\begin{equation}\label{weshow} Z_{0, w_N}(\floor{arN^{2/3}}+1\leq \tau\leq \floor{ brN^{2/3}})\geq e^{C'r^2N^{1/3}} Z_{0, w_N}(1 \leq \tau \leq \floor{arN^{2/3}}),
\end{equation}
as this implies
\begin{align*}
R_{0, w_N}\{\tau > arN^{2/3}\} &\geq \frac{Z_{0, w_N}(\floor{arN^{2/3}}+1\leq \tau \leq  \floor{ brN^{2/3}})}{Z_{0, w_N}(1\leq \tau)} \\
&\geq e^{C'r^2N^{1/3}} \frac{ Z_{0, w_N}(1 \leq \tau \leq \floor{arN^{2/3}})}{Z_{0, w_N}(1 \leq \tau)}\\
&= e^{C'r^2N^{1/3}} R_{0, w_N}\{1\leq \tau \leq arN^{2/3}\},
\end{align*}
which together with
$$ R_{0, w_N}\{1\leq \tau \leq arN^{2/3}\} + R_{0, w_N}\{\tau > arN^{2/3}\} =1 $$
gives
$$ R_{0, w_N}\{1\leq\tau\leq arN^{2/3}\} \leq\frac{1}{1+ e^{C'r^2N^{1/3}}} \leq e^{-C'r^2N^{1/3}}\quad \text{ on $A$}.$$

Next, we define $A_1, A_2, A_3$ and $A_4$ and their intersection gives $A$.
Let $Z^\lambda$ and $Z^{{\rho}}$ denote the partition functions with the $\lambda$- and ${{\rho}}$-boundary weights, and where all boundary weights are independent. Then, the $e_1$-boundary weights from $\wt{\mathbb{P}}$ can be seen as a mixture of these $\lambda$- and $\rho$- weights. The desired inequality \eqref{weshow} (under $\wt{\mathbb{P}}$) can be rewritten as 
\begin{align*}
&\Big(\frac{Z_{0, w_N}^\lambda}{Z^{{\rho}}_{0, w_N}} \prod_{i=1}^{\floor{arN^{2/3}}} \frac{I^{{\rho}}_{(i,0)}}{I^{\lambda}_{(i,0)}} \Big) \frac{Z_{0, w_N}^\lambda (\floor{arN^{2/3}}+1\leq \tau \leq  \floor{brN^{2/3}})}{Z^\lambda_{0, w_N}}\\
& \qquad \qquad \qquad \qquad \qquad \qquad \qquad \qquad \qquad\geq e^{C'r^2 N^{1/3}} \frac{Z^{{\rho}}_{0, w_N} (1 \leq \tau \leq \floor{arN^{2/3}})}{Z^{{\rho}}_{0, w_N}}, 
\end{align*}
which is implied by the inequality  
$$\Big(\frac{Z_{0, w_N}^\lambda}{Z_{0, w_N}^{{\rho}}} \prod_{i=1}^{arN^{2/3}} \frac{I^{{\rho}}_{(i,0)}}{I^{\lambda}_{(i,0)}} \Big) \frac{Z^\lambda_{0, w_N} (\floor{arN^{2/3}} + 1 \leq \tau \leq  \floor{brN^{2/3}})}{Z^\lambda_{0, w_N}} \geq e^{C'r^2 N^{1/3}}.$$

Because $w_N$ is a point of order $rN^{2/3}$ units away from $u_N$ (recall $u_N$ is along the $\xi[\lambda]$-characteristic ray defined above \eqref{cbound}), there is an event $A_1$ with $\mathbb{P}(A_1) \geq 1-e^{-Cr^3}$ such that  the $\lambda$ quenched probability appearing above (i.e.\ the last ratio of partition functions on the left-hand side) satisfies 
$$\frac{Z^\lambda_{0, w_N} (\floor{arN^{2/3}} + 1 \leq \tau \leq \floor{brN^{2/3}})}{Z^\lambda_{0, w_N}} \geq 1/2 \quad \text{ on the event $A_1$}.$$  This is proved as Lemma \ref{lemmaab} at the end of this section, and the idea is  illustrated on the right of  Figure \ref{fig99}.

Once on the event $A_1$, \eqref{weshow} would follow from having 
\begin{equation}\label{simple}
\frac{Z^\lambda_{0, w_N}}{Z^{{\rho}}_{0, w_N}} \prod_{i=1}^{arN^{2/3}} \frac{I^{{\rho}}_{(i,0)}}{I^{\lambda}_{(i,0)}} \geq e^{C'r^2N^{1/3}},
\end{equation}
with possibly a different $C'$.
This inequality should hold with a high probability if $a>0$ is taken sufficiently small. We will work with the logarithmic version of \eqref{simple}
$$\log Z^\lambda_{0, w_N}  - \log Z^{{\rho}}_{0, w_N}  - \Big( \sum_{i=1}^{arN^{2/3}} \log (I^\lambda_{(i,0)}) - \log (I^{{\rho}}_{(i,0)})\Big). 
$$

We start by showing that
\begin{equation}\label{exp_diff}
\mathbb{E}[\log Z^\lambda_{0, w_N} ] - \mathbb{E}[ \log Z^{{\rho}}_{0, w_N} ] \geq c_1 r^2 N^{1/3}
\end{equation}
for some $\varepsilon$-dependent constant $c_1$, and this constant $c_1$ will be used for the rest of the proof. First, note the exact values of the expectations are
\begin{align*}
\mathbb{E}[\log Z^\lambda_{0, w_N} ]  &= \Psi_0(\mu-\lambda)(\floor{\Psi_1({{\rho}}) N} - \floor{qrN^{2/3}}) + \Psi_0(\lambda)\floor{\Psi_1(\mu-{{\rho}}) N}\\
\mathbb{E}[\log Z^{{\rho}}_{0, w_N} ]  &= \Psi_0(\mu-{{\rho}})(\floor{\Psi_1({{\rho}}) N} - \floor{ qrN^{2/3}}) + \Psi_0({{\rho}})\floor{\Psi_1(\mu-{{\rho}}) N }.
\end{align*}
Using a Taylor expansion,
\begin{align*}
\Psi_0(\mu-\lambda) &=  \Psi_0(\mu-{{\rho}}) + \Psi_1(\mu-{{\rho}})(-rN^{-1/3}) + \tfrac{1}{2} \Psi_1'(\mu-{{\rho}})(-rN^{-1/3})^2 + R_1,\\
\Psi_0(\lambda) &=  \Psi_0({{\rho}}) + \Psi_1({{\rho}})(rN^{-1/3}) + \tfrac{1}{2} \Psi_1'({{\rho}})(rN^{-1/3})^2 + R_2.
\end{align*}
Due to condition \eqref{lambda} we have $|R_i|\le C(rN^{-1/3})^3$ for both $i\in\{1,2\}$ and with an $\varepsilon$-dependent constant $C>0$. Plugging these two formulas back into the right side of \eqref{exp_diff}, the linear terms from the expansions cancel out. By further lowering the value of $c$ from \eqref{fix_r} if necessary, $R_1$ and $R_2$ can be absorbed into the  $(rN^{-1/3})^2$ terms, and 
there exist two positive constants $D_1$ and $D_2$ depending only on $\varepsilon$, $\mu$ and $c$ such that 
$$\mathbb{E}[\log Z^\lambda_{0, w_N} ] - \mathbb{E}[ \log Z^{{\rho}}_{0, w_N} ]  \geq D_1 r^2 N^{1/3} - D_2 q r^2 N^{1/3},$$
where the parameter $q$ is from \eqref{define_w_N}. By fixing $q$ sufficiently small, we obtain the desired estimate \eqref{exp_diff}.

Next, with the constant $c_1$ from \eqref{exp_diff}, we define the two events
\begin{align*}
A_2 &= \Bigl\{\log Z^\lambda_{0, w_N}  \geq \mathbb{E}[\log Z^{{\rho}}_{0, w_N} ] + \frac{c_1}2 r^2N^{1/3}\Bigr\},\\
A_3 &= \Big\{\log Z^{{\rho}}_{0, w_N}  \leq \mathbb{E}[ \log Z^{{\rho}}_{0, w_N} ] + \frac{c_1}{10}{r^2N^{1/3}}\Big\},
\end{align*}
and we will show $\mathbb{P}(A_2) \wedge \mathbb{P}(A_3) \geq 1- Cr^{-3}$.

First, we work with $\mathbb{P}(A_2)$. For $\theta, x > 0$, let us define $L(\theta, x)$ as in (3.17) of \cite{Sep-12-corr} ,
$$L(\theta, x) = \int_0^x (\Psi_0(\theta) - \log y) x^{-\theta} y^{\theta-1}e^{x-y} dy.$$
In the next calculation, the first equality is the statement in Theorem 3.7 from \cite{Sep-12-corr},
\begin{align}
\Var[\log Z^{{\rho}}_{0, w_N} ] &= w_N\cdot e_2 \Psi_1({{\rho}}) - w_N\cdot e_1 \Psi_1 (\mu-{{\rho}}) + 2\mathbb{E}\Big[E^{Q^\rho_{0, w_N}} \Big[ \sum_{i=1}^{0 \vee \tau} L(\mu-{{\rho}}, I^\rho_{ie_1})\Big] \Big]\notag\\
&\leq C\Bigl(rN^{2/3} + \mathbb{E} \Big[ E^{Q^\rho_{0, w_N}}\Big[ \tau\mathbbm{1}_{\{\tau \geq 1\}}\Big] \Big] + 1\Bigr)\quad\text{(by Lemma 4.2 of \cite{Sep-12-corr})}\notag\\
&\leq C\Bigl(rN^{2/3} + \mathbb{E} \Big[ E^{Q^\rho_{0, v_N}}\Big[ \tau\mathbbm{1}_{\{\tau \geq 1\}}\Big] \Big] + 1\Bigr)\quad\text{ (by Lemma \ref{polymono})}\notag\\
&\leq CrN^{2/3} + C'N^{2/3}\qquad\qquad\qquad\qquad\qquad\qquad\quad\text{(by (4.32) of \cite{Sep-12-corr})}. \label{var_bound}
\end{align}
Now, we upper bound the compliment
\begin{align*}
\mathbb{P}(A_2^c) &= \mathbb{P}\Bigl\{\log Z^{\lambda}_{0, w_N} < \mathbb{E}[\log Z^{{\rho}}_{0, w_N} ] + \frac{c_1}2 r^2 N^{1/3}\Bigr\}\\
& \leq \mathbb{P}\Bigl\{\log Z^{\lambda}_{0, w_N} < \mathbb{E}[\log Z^\lambda_{0, w_N} ] - \frac{c_1}2 r^2 N^{1/3}\Bigr\}\quad\text{(by \eqref{exp_diff})}\\
& \leq \frac{4}{c_1^2r^4 N^{2/3}} \Var[\log Z^\lambda_{0, w_N} ]\\
& \leq \frac{4}{c_1^2r^4 N^{2/3}} (\Var[\log Z^{{\rho}}_{0, w_N} ] + c_3 rN^{2/3})\qquad\quad\text{(by Lemma 4.1 of \cite{Sep-12-corr})}\\
& \leq Cr^{-3}\qquad\qquad\qquad\qquad\qquad\qquad\qquad\qquad\ \text{(by \eqref{var_bound})}.
\end{align*}
The fact $\mathbb{P}(A_3) \geq 1- Cr^3$ comes from the Markov inequality 
$$\mathbb{P}(A_3^c) = \mathbb{P}\Big\{\log Z^{{\rho}}_{0, w_N}  > \mathbb{E}[ \log Z^{{\rho}}_{0, w_N} ] + \frac{c_1}{10}{r^2N^{1/3}}\Big\} \leq \frac{100}{c_1^2r^4 N^{2/3}} \Var[\log Z^{{\rho}}_{0, w_N} ] \leq Cr^{-3}.$$

Next, we define another high probability event $A_4$ by 
$$A_4 = \Big\{\sum_{i=1}^{arN^{2/3}} \bigl(\log I^\lambda_{(i,0)} - \log I^{{\rho}}_{(i,0)}\bigr) \leq \frac{c_1}{10} r^2N^{1/3}\Big\}.$$
If $a$ is chosen sufficiently small compared to $c_1$, then by Proposition \ref{Ga_sub_exp} and Theorem \ref{max_sub_exp},   
$$\mathbb{P}(A_4) \geq 1- e^{-C r^3}.$$

Finally, on the event 
$$A_1\cap A_2\cap A_3 \cap A_4,$$
our desired estimate \eqref{simple} (after taking logarithm) will hold
$$\log Z^\lambda_{0, w_N}  - \log Z^{{\rho}}_{0, w_N}  - \Big( \sum_{i=1}^{arN^{2/3}} \log I^\lambda_{(i,0)} - \log I^{{\rho}}_{(i,0)}\Big) \geq \frac{c_1}{10} r^2 N^{1/3} \geq C'r^2N^{1/3}.$$
This finishes the argument for the dark region.

For the light region, 
\begin{align*}
&\wt{\mathbb{P}}\Big(\max_{x\in \mathcal{L}} Q_{0, x}\{1\leq \tau \leq arN^{2/3}\} \leq e^{-C'r^2 N^{1/3}}\Big) \\
&\geq {\mathbb{P}}\Big(\max_{x\in \mathcal{L}} Q^\rho_{0, x}\{1\leq \tau \leq arN^{2/3}\} \leq e^{-C'r^2 N^{1/3}}\Big) \\
&\geq {\mathbb{P}}\Big(\max_{x\in \mathcal{L}} Q^\rho_{0, x}\{1\leq \tau \} \leq e^{-C'r^2 N^{1/3}}\Big) \\
 &= {\mathbb{P}}\Big( Q^\rho_{0, w_N}\{1\leq \tau \} \leq e^{-C'r^2 N^{1/3}}\Big) \qquad (\text{by Lemma \ref{polymono}} ) \\
& = 1- \mathbb{P}\Big( Q^\rho_{0, w_N}\{1\leq \tau \} > e^{-C'r^2 N^{1/3}}\Big)\\
\qquad &\geq 1- e^{-Cr^3} \qquad \qquad \quad 
 \; \qquad \qquad \qquad  (\text{by Corollary \ref{sec3cor}}).
\end{align*}
The proof of Lemma \ref{rtail_lb} is complete.\hfill\qed

\medskip

Lemma \ref{lemmaab} below is  an auxiliary estimate for the proof of Lemma \ref{rtail_lb}. 
Recall that $\lambda = {{\rho}} + rN^{-1/3}$ and satisfies the condition \eqref{lambda}. As shown on the right of Figure \ref{fig21},  $u_N$ and $v_N$ on the north boundary satisfies \eqref{cbound}. Using the parameters $l_1$ and $l_2$ in \eqref{cbound}, we fix
\begin{equation} \label{fixabq}
a \leq \tfrac{1}{10} l_1, \quad b \geq 10 l_2, \quad q\leq \tfrac{1}{10} l_1.
\end{equation}
Recall $w_N = v_N - qrN^{2/3}e_1$ is a point on the north boundary of $[\![ 0, v_N]\!]$.
Lemma \ref{lemmaab} shows that for  small enough $a>0$ and large enough $b>0$, the sampled polymer path between the origin and $w_N$ exits the $e_1$-axis   through the interval $[\![ arN^{2/3}e_1, brN^{2/3}e_1]\!]$ with high probability under $\mathbb{P}^\lambda$.  This is illustrated  on the left of Figure \ref{fig99}.

\begin{lemma} \label{lemmaab}  
Fix ${{\rho}}\in (0, \mu)$ and $a, b, q$ as in \eqref{fixabq}. There exist positive constants $C_1, C_2, C_3, r_0$, and $N_0$ that  depend only on ${{\rho}}$ such that, for any $r>r_0$, $N \geq N_0$ with $\lambda = {{\rho}} + rN^{-1/3}$ satisfying \eqref{lambda}, we have
$$\mathbb{P} \Big(Q^\lambda_{0, w_N} \big\{arN^{2/3} \leq \tau\leq brN^{2/3}\big\} \leq 1- e^{-C_1 r^2N^{1/3}} \Big) \leq e^{-C_2r^3}$$
and
$$\mathbb{E} \Big[Q^\lambda_{0, w_N} \big\{arN^{2/3} \leq \tau \leq brN^{2/3}\big\} \Big] \geq 1-e^{-C_3 r^3}.$$ 
\end{lemma}
\begin{proof}
First, note we have the following horizontal distance bound between $w_N$ and $u_N$, where $u_N$ is defined previously above \eqref{cbound}
$$\tfrac{1}{2} l_1 rN^{2/3} \leq w_N\cdot e_1 - u_N \cdot e_1 \leq l_2 rN^{2/3}. $$
Let $z_N$ be the integer point closest to where the $-\xi[\lambda]$-directed ray starting at $w_N$ crosses the $e_1$-axis (illustrated as the white dot in Figure \ref{fig99}), then the distance between the origin and $z_N$ satisfies the same bound
\begin{equation}\label{hbound}
\tfrac{1}{2} l_1 rN^{2/3} \leq z_N \cdot e_1 \leq l_2 rN^{2/3}. 
\end{equation}

In the next part, we will show that sampled polymer path between the origin and $w_N$ will exist on the $e_1$-axis near $z_N$.
More precisely, we show for $r>r_0$ and  $N \geq N_0$ such that \eqref{lambda} holds, then
\begin{align}
&\mathbb{P}\big(Q^\lambda _{0, w_N} \{\tau <a rN^{2/3} \big) \geq e^{-Cr^2N^{1/3}}\big) \leq e^{-C'r^3}, \label{noteasyone}\\
&\mathbb{P} \big(Q^\lambda_{0, w_N} \{\tau >b rN^{2/3} \big) \geq e^{-Cr^2N^{1/3}}\big) \leq e^{-C'r^3}.\label{easyone}
\end{align}

First, we show \eqref{easyone}. In the estimate below, the first inequality follows from Lemma \ref{polymono}; the next equality comes from
Moving the base from the origin to $z_N$ as a nested polymer; the final inequality comes from applying  Theorem \ref{r_up_low} to the nested polymer where the starting and end points are in the $\xi[\lambda]$ direction),
\begin{align*}
& \mathbb{P} \big(Q^\lambda_{0, w_N}\{\tau > brN^{2/3}\} \geq e^{-C_1r^2N^{1/3}} \big)\\
& \leq \mathbb{P} \big(Q^\lambda_{0, v_N}\{\tau > brN^{2/3}\} \geq e^{-C_1r^2N^{1/3}} \big)\\
& = \mathbb{P}\big(Q^\lambda _{z_N, v_N}\{\tau > brN^{2/3} - z_N\cdot e_1\}\geq e^{-C_1r^2N^{1/3}} \big) \\
& = \mathbb{P} \big(Q^\lambda_{z_N, v_N}\{\tau > \tfrac{b}{2}rN^{2/3} \}\geq e^{-C_1r^2N^{1/3}} \big) \\
& \leq e^{-C_2r^3}.
\end{align*}
This proves \eqref{easyone}.

To prove \eqref{noteasyone}   choose $h$  so  that   $(\floor{a rN^{2/3}}, -h)$ is the  closest integer point to the $(-\xi[\lambda])$-directed ray starting at $w_N$  (see the right of Figure \ref{figlemmaab}). Lemma \ref{relatetau} gives 
\begin{align*}
&\mathbb{P}\big(Q^\lambda _{0,w_N}\{\tau \leq  a rN^{2/3}\} \geq e^{-C_1r^2N^{1/3}}\big)  \\
&\qquad \qquad = \mathbb{P} \big(Q^\lambda_{\floor{arN^{2/3}}, -h),w_N}\{\tau < -h \} \geq e^{-C_1r^2N^{1/3}} \big).
\end{align*}
\begin{figure}[t]
\captionsetup{width=0.8\textwidth}
\begin{center}
\begin{tikzpicture}[scale = 0.8]

--------------------------------------------------------------

\draw[gray ,dotted, line width=0.3mm](1,-1)--(4,3);

\draw[gray, line width=0.3mm, ->] (0,0) -- (6,0);
\draw[gray, line width=0.3mm, ->] (0,0) -- (0,4);

\draw[ line width=0.3mm, ->] (1,-1) -- (5,-1);
\draw[ line width=0.3mm, ->] (1,-1) -- (1,2);

\draw[ fill=white](1.75, 0)circle(1.3mm);

\draw[gray ,dotted, line width=1.1mm] (0,0) -- (0.7,0) -- (0.7,0.5) --(1,0.5) ;

\draw[lightgray ,dotted, line width=1.4mm] (1,0.5) --(2,0.5)-- (2,1) -- (4, 1) -- (4,3) ;
\draw[black ,dotted, line width=0.6mm] (1,0)  --(1,0.5) --(2,0.5)-- (2,1) -- (4, 1) -- (4,3) ;

\draw[black ,dotted, line width=1.1mm] (1,-1) -- (1,0.5);

\draw[ fill=lightgray](0,0)circle(1mm);
\node at (-0.3,-0.4) {$(0,0)$};

\draw[ fill=lightgray](1,-1)circle(1mm);
\node at (0.7,-1.4) {$({arN^{2/3}},-h)$};

\fill[color=white] (4,3)circle(1.7mm); 
\draw[ fill=lightgray](4,3)circle(1mm);
\node at (4.3,3.3) {$w_N$};

\draw[ line width=1mm] (3,-0.1) -- (3,0.1);
\node at (3.3, -0.4) {${brN^{2/3}}$};

\node at (2, -0.35) {$z_N$};

\end{tikzpicture}
\end{center}
\caption{The vertex $z_N$ is shown as the white dot. Applying Lemma \ref{relatetau} in the proof of Lemma \ref{lemmaab} to assert that  $Q^\lambda_{0,w_N}\{\tau \leq a rN^{2/3}\} = Q^\lambda_{(\floor{arN^{2/3}}, -h),w_N}\{\tau< -h \},$ which is small. }
\label{figlemmaab}
\end{figure}
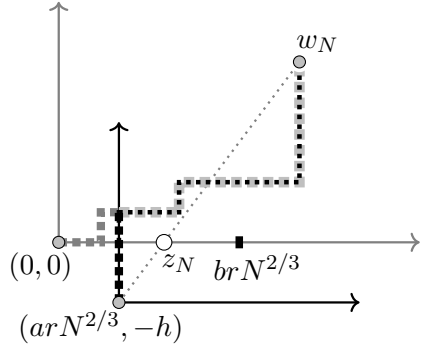
Theorem \ref{r_up_low} states that it is unlikely for the sampled polymer paths from $Q_{(\floor{arN^{2/3}}, -h),w_N}$ to exit  late in the scale $N^{2/3}$  from the $y$-axis because  the direction is the  characteristic one $\xi[\lambda]$. Thus it suffices to show $h$ is bounded below by some $k({{\rho}})rN^{2/3}$. 

Using the lower bound from $\eqref{hbound}$, the distance between $z_N$ and $\floor{arN^{2/3}}e_1$ is bounded below by $4arN^{2/3}$. The slope of the line going through $w_N$ and $z_N$ is roughly $\xi[\lambda]$, because recall $z_N$ is defined to be the closes integer point to the crossing point between the $-\xi[\lambda]$-directed ray from $w_N$ and the $e_1$-axis.
Thus its slope is contained inside a compact interval strictly inside $(0, \mu)$. Thus, we have 
\begin{equation} \label{bound_h}
h \geq k({{\rho}}) rN^{2/3}
\end{equation}
which finishes the proof.
\end{proof}

\subsection{Proof of Theorem \ref{dupper}}

First, note that instead of 
$$\max_{x\not\in [\![0, v_N]\!]} Q^\rho_{0,x}\{|\tau|\leq \delta N^{2/3}\},$$
 it suffices to work with 
 $$\max_{x\in \partial^{\textup{NE}} [\![0, v_N]\!]} Q^\rho_{0,x}\{|\tau|\leq \delta N^{2/3}\}$$
since 
\begin{align}
\begin{split}\label{out_vs_bd}
&\max_{x\not\in [\![0, v_N]\!]} Q^\rho_{0,x}\{|\tau|\leq \delta N^{2/3}\} \\
&\qquad =\max_{x\not\in [\![0, v_N]\!]} \sum_{z\in \partial^{\textup{NE}} \lzb 0, v_N \rzb} Q^\rho_{0,x}\{|\tau|\leq \delta N^{2/3} \text{ and passes through {z}}\}  \\
&\qquad =\max_{x\not\in  [\![0, v_N]\!]} \sum_{z\in \partial^{\textup{NE}} [\![0, v_N]\!]} Q^\rho_{0,z}\{|\tau|\leq \delta N^{2/3} \} Q_{0,x}\{\text{passes through $z$}\} \\
&\qquad \leq \max_{x\not\in [\![0, v_N]\!]} \sum_{z\in \partial^{\textup{NE}} [\![0, v_N]\!]} \Big(\max_{z'\in \partial^{\textup{NE}} \lzb 0, v_N \rzb} Q^\rho_{0,z'}\{|\tau|\leq \delta N^{2/3} \}\Big) Q^\rho_{0,x}\{\text{passes through $z$}\} \\
&\qquad=\max_{z'\in \partial^{\textup{NE}} [\![0, v_N]\!]} Q^\rho_{0,z'}\{|\tau|\leq \delta N^{2/3}\} .
\end{split}
\end{align}

\begin{figure}[t]
\captionsetup{width=0.8\textwidth}
 \begin{center}
\begin{tikzpicture}[>=latex, scale=1]
\draw[line width = 0.3mm, ->] (0,0)--(0,4);
\draw[line width = 0.3mm, ->] (0,0)--(5,0);

\draw[color=lightgray, line width = 1mm] (0,3) -- (2.5, 3);
\draw[ color = darkgray, line width = 1mm] (2.5,3) -- (4, 3) -- (4,1.5);
\draw[color=lightgray, line width = 1mm] (4,1.5)--(4,0);

\draw[loosely dotted, line width = 0.3mm, ->] (0,0) -- (5, 3+3/4);

\draw [decorate,decoration={brace,amplitude=6pt,raise=0ex}]
  (2.5,3.1) -- (4,3.1);

\node at (3.4,3.6) {\small ${qrN^{2/3}}$};

\fill[color=white] (4,1.5)circle(1.7mm);
\draw[ fill=lightgray](4,1.5)circle(1mm);
\node at (4.4,1.5) {\small $w_N^-$};

\fill[color=white] (2.5,3)circle(1.7mm);
\draw[ fill=lightgray](2.5,3)circle(1mm);
\node at (2.4,2.6) {\small $w_N^+$};
\fill[color=white] (0,0)circle(1.7mm);
\draw[ fill=white](0,0)circle(1mm);
\node at (-0.6,0) {\small $(0,0)$};

\fill[color=white] (4,3)circle(1.7mm);
\draw[ fill=lightgray](4,3)circle(1mm);
\node at (4.4,2.8) {\small $v_N$};

\node at (5.5, 4) {$\xi[{{\rho}}]$};

\node at (4.4, 2.2){$\cD$};
\node at (1.2, 3.3){$\mathcal{L}^+$};

\node at (4.4, 0.7){$\mathcal{L}^-$};

\end{tikzpicture}

\end{center}
\caption{\small The north and east boundaries of $\lzb 0, v_N\rzb $ are decomposed into  $\cL^\pm$ (light gray) and $\cD$ (dark gray). The parameter $q$ is less than some small constant that depends only on ${{\rho}}$.}
\label{fig9}
\end{figure}
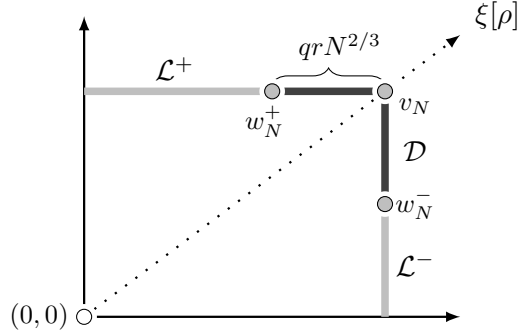

Decompose the northeast boundary $\partial^{\textup{NE}}\lzb 0, v_N \rzb$ into three parts $\cD$ and $\cL^\pm$  as in Figure \ref{fig9}, with
 $$ w_N^+ = v_N - \floor{qrN^{2/3}}e_1 \quad\text{and}\quad 
   w_N^- = v_N - \floor{qrN^{2/3}}e_2$$
where $q \leq 1 $ is a small positive constant to be chosen later above \eqref{lbuw}, and 
$$r = |\log \delta|.$$ The dark gray set $\cD$ comprises the vertices between  $w_N^+$ and $w_N^-$ on the northeast corner of the boundary of the  rectangle $\lzb0, v_N \rzb$. 
Recall that we assume in the theorem that 
\be\label{Nass4}  N>\delta^{-3/2}.\ee
This is natural since otherwise the probability in the statement 
of the theorem would be zero.
Introduce the perturbed parameters
  \be\label{la8} 
  \lambda = {{\rho}} + rN^{-1/3} \quad \text{ and } \quad\eta= {{\rho}} - rN^{-1/3}.
  \ee
We require the following bounds to hold for these two parameters
\be\label{la9} {{\rho}} < \lambda \leq {{\rho}} + \frac{{{\rho}}\wedge(\mu-{{\rho}})}{2} < \mu \quad \text{ and } \quad 0< {{\rho}} - \frac{{{\rho}}\wedge(\mu-{{\rho}})}{2} \leq \eta < {{\rho}}. 
       \ee
The point of  the choice ${{\rho}} \pm \frac{{{\rho}}\wedge(\mu-{{\rho}})}{2}$ is only to bound $\lambda$ and $\eta$ from above and below by two constants strictly inside $(0, \mu)$ and only depending on $\varepsilon$. The above two requirements can be rewritten as
$$
N\geq \Big( \frac{2r}{{{\rho}}\wedge(\mu-{{\rho}})}\Big)^3.
$$
With \eqref{Nass4}, this bound on $N$  is automatically satisfied  as soon as  $  \delta^{-3/2} \geq  \big( \frac{2r}{{{\rho}}\wedge(\mu-{{\rho}})}\big)^3$.
Since $r = |\log \delta |$,  we can ensure this by lowering the value of $\delta_0$.

Now we show that if one takes $q$ and $\alpha$ small enough, 
then the $\xi[\eta]$- and $\xi[\lambda]$-directed rays started at the points $\pm \floor{\alpha rN^{2/3}} e_1$ will avoid $\cD$ as shown in Figure \ref{nest1}. To this end, 
recall $\xi[{{\rho}}]$ defined in \ref{char_dir}. Let $u_N$ be the point where the $\xi[\lambda]$-ray starting from $\floor{\alpha r N^{2/3}}e_1$ crosses the north boundary of $[\![0, v_N ]\!]$. Then the $e_1$-coordinates of $w_N^+$ and $u_N$ can be lower bounded by 
\begin{align}
&\Bigl(\frac{\Psi_1(\lambda)}{\Psi_1(\mu-\lambda)}\cdot\frac{\Psi_1(\mu-\rho)}{\Psi_1(\rho)+\Psi_1(\mu-\rho)}-\frac{\Psi_1(\rho)}{\Psi_1(\rho)+\Psi_1(\mu-\rho)}\Bigr)N- \alpha rN^{2/3} - qrN^{2/3} - 5\notag\\
&\quad=\frac{\Psi_1(\mu-\rho)}{\Psi_1(\rho)+\Psi_1(\mu-\rho)}\cdot\Bigl(\frac{\Psi_1(\lambda)}{\Psi_1(\mu-\lambda)}-\frac{\Psi_1(\rho)}{\Psi_1(\mu-\rho)}\Bigr)N- \alpha rN^{2/3} - qrN^{2/3} - 5\notag\\
&\quad\ge C_1(\varepsilon) rN^{2/3}-\alpha rN^{2/3} - qrN^{2/3} - 5,
\label{lbuw}
\end{align}
where the inequality comes from Taylor's theorem since 
$\Psi_1$ is a smooth function
on a compact interval inside $(0, \mu)$ depending on $\varepsilon$. 
Here, $C_1(\varepsilon)$ is a finite positive constant that only depends on $\varepsilon$. 
The inequality holds provided $rN^{-1/3}\le c(\varepsilon)$ for some positive $c(\varepsilon)$ that only depends on $\varepsilon$ and this can be guaranteed to hold by lowering the threshold $\delta_0$ since 
$$rN^{-1/3} \leq |\log \delta|\delta^{1/2} \leq \delta_0^{1/3}.$$ 
Now choosing 
\begin{equation}\label{q=a}
q = \alpha = C_1({{\varepsilon}})/10,
\end{equation}
we obtain 
\begin{equation}\eqref{lbuw} \geq C_2({{\varepsilon}}) rN^{2/3},\end{equation}
and this  gives us the desired picture for $\xi[\lambda]$ shown in Figure \ref{nest1}. The argument for the $\xi[\eta]$-directed ray is similar. 
For what follows we also want to guarantee that $\delta < \alpha r  = \alpha |\log \delta|$. This can be done by  decreasing the value of $\delta_0$ after having fixed $\alpha$.  This completes the setup described in Figure \ref{nest1}.

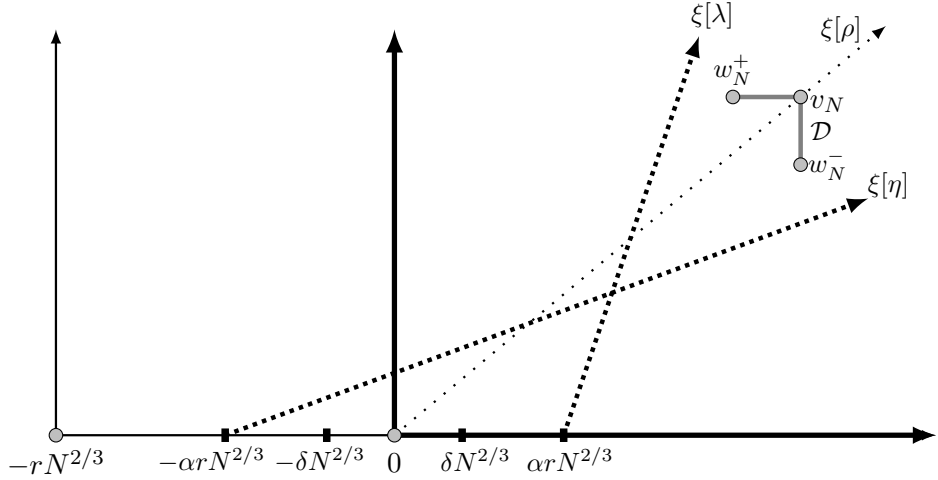
\begin{figure}[t]
\captionsetup{width=0.8\textwidth}
\begin{center}
 
\begin{tikzpicture}[>=latex, scale=0.9]


\draw[ line width = 0.6mm, ->, dotted] (-2.5,0)--(7,3.5);
\draw[ line width = 0.6mm, ->,  dotted] (2.5,0)--(4.5,5.9);
\node at (7.3,3.7) {\small $\xi[\eta]$};
\node at (4.7,6.2) {\small $\xi[\lambda]$};

\draw[ line width = 0.3mm, ->, loosely dotted] (0,0)--(6.6*1.1,5.5*1.1);
\node at (6.6,6) {\small $\xi[{{\rho}}]$};

\node at (6.3, 4.5) {\small $\cD$};
\draw[ line width = 0.6mm, color=gray] (5,5)--(6,5)--(6,4);

\draw[ fill=lightgray](5,5)circle(1mm);
\node at (5,5.4) {\small $w_N^+$};

\draw[ fill=lightgray](6,5)circle(1mm);
\node at (6.4,4.9) {$v_N$};

\draw[ fill=lightgray](6,4)circle(1mm);
\node at (6.4,4) {\small $w_N^-$};

\draw[line width = 0.3mm, ->] (-5,0)--(8,0);
\draw[line width = 0.3mm, ->] (-5,0)--(-5,6);

\draw[line width = 0.7mm, ->] (0,0)--(0,6);
\draw[line width = 0.7mm, ->] (0,0)--(8,0);

\draw[line width = 1mm] (1,-0.1) -- (1,0.1);
\node at (1.2, -0.35) {\small $\delta N^{2/3}$};

\draw[line width = 1mm] (2.5,-0.1) -- (2.5,0.1);
\node at (2.6, -0.35) {\small $\alpha rN^{2/3}$};

\draw[line width = 1mm] (-1,-0.1) -- (-1,0.1);
\node at (-1.1, -0.35) {\small $-\delta N^{2/3}$};

\draw[line width = 1mm] (-2.5,-0.1) -- (-2.5,0.1);
\node at (-2.7, -0.35) {\small $-\alpha rN^{2/3}$};

\draw[ fill=lightgray](-5,0)circle(1mm);
\node at (-5,-0.4) {$-rN^{2/3}$};

\draw[ fill=lightgray](0,0)circle(1mm);
\node at (-0,-0.4) {$0$};

\end{tikzpicture}
 
\end{center}
\caption{\small Illustration of the set $\mathcal {D}$, the nested polymer, and three characteristic directions. The parameters $q=\alpha$ are less than some small constant that depends only on ${{\rho}}$, $\delta$ is a small positive constant in $(0, \delta_0)$, and $r$ is a large constant with $r = |\log \delta|$.} \label{nest1}
\end{figure}

Consider the set $\cD$
shown in Figure \ref{fig9} in dark gray and also in Figure \ref{nest1}. Place the stationary polymer model on $0 + \mathbb{Z}^2_{\geq 0}$ as a nested polymer inside a larger stationary polymer model on the quadrant  $-\floor{rN^{2/3}}e_1 + \mathbb{Z}^2_{\geq 0}$.
From the relation between two nested polymers given by Lemma \ref{nestedpoly}, we have
\begin{align}
\mathbb{P}&\Big(\max_{ z \in \cD} Q^{{\rho}}_{0,z}\{1\leq \tau \leq \delta N^{2/3}\} \geq e^{-|\log \delta|^2 \sqrt{\delta} N^{1/3}} \Big) \label{mainest2} \\
  \leq \mathbb{P} &\Big(\max_{ z \in \cD}Q^{{\rho}}_{-\floor{rN^{2/3}}e_1, z}\{\floor{rN^{2/3}}+1 \leq \tau \leq \floor{rN^{2/3}} + \delta N^{2/3}\} \geq e^{-|\log \delta|^2 \sqrt{\delta} N^{1/3}} \Big) \nonumber\\
 \leq \mathbb{P}& \Big(\max_{ z \in \cD} \frac{Z^{{\rho}}_{-\floor{rN^{2/3}}e_1, z}(\floor{rN^{2/3}}+1  \leq \tau \leq \floor{rN^{2/3}} + \delta N^{2/3})}{Z^{{\rho}}_{-\floor{rN^{2/3}}e_1, z}(\floor{rN^{2/3}} - \alpha rN^{2/3}+1 \leq \tau \leq \floor{rN^{2/3}} + \alpha rN^{2/3})} \geq e^{-|\log \delta|^2 \sqrt{\delta} N^{1/3}}\Big)\nonumber\\
 = \mathbb{P}& \Big(\min_{ z \in \cD} \Big\{\log Z^{{\rho}}_{-\floor{rN^{2/3}}e_1, z}(\floor{rN^{2/3}} - \alpha N^{2/3} +1\leq \tau \leq \floor{rN^{2/3}} + \alpha N^{2/3}) \nonumber\\
& \qquad   - \log {Z^{{\rho}}_{-\floor{rN^{2/3}}e_1, z}(\floor{rN^{2/3}} +1\leq \tau \leq \floor{rN^{2/3}} + \delta N^{2/3})} \Big\}\leq{|\log \delta|^2 \sqrt{\delta} N^{1/3}}\Big)\nonumber\\
\leq \mathbb{P}& \Big(\min_{ z \in \cD}  \Big\{\max_{i \in [\![-\alpha N^{2/3} + 1, \alpha N^{2/3}]\!] }\log Z^{{\rho}}_{-\floor{rN^{2/3}}e_1, z}( \tau = \floor{rN^{2/3}} + i) \nonumber \\
& \qquad   - \log Z^{{\rho}}_{-\floor{rN^{2/3}}e_1, z}( \tau = \floor{rN^{2/3}} )  + \log Z^{{\rho}}_{-\floor{rN^{2/3}}e_1, z}( \tau = \floor{rN^{2/3}}  )\nonumber\\
& \qquad \qquad  - \max_{k   \in [\![1, \delta  N^{2/3}]\!]}\log {Z^{{\rho}}_{-\floor{rN^{2/3}}e_1, z}(\tau = \floor{rN^{2/3}}+k)}\Big\} \leq  2{|\log \delta|^2 \sqrt{\delta}N^{1/3}}\Big).\nonumber\\
\leq \mathbb{P} &\Big(\min_{ z \in \cD}  \Big\{\max_{i \in [\![-\alpha N^{2/3} + 1, \alpha N^{2/3}]\!] }\log Z^{{\rho}}_{-\floor{rN^{2/3}}e_1, z}( \tau = \floor{rN^{2/3}} + i) \nonumber \\
&\qquad   - \log Z^{{\rho}}_{-\floor{rN^{2/3}}e_1, z}( \tau = \floor{rN^{2/3}} ) \Big\}  \leq 3 {|\log \delta|^2 \sqrt{\delta} N^{1/3}}\Big)\label{cont_bd} \\
& \qquad \qquad + \mathbb{P}\Big(\max_{ z \in \cD} \Big\{\max_{k   \in [\![1, \delta N^{2/3}]\!]} \log Z^{{\rho}}_{-\floor{rN^{2/3}}e_1, z}( \tau = \floor{rN^{2/3}} + k  )\nonumber\\
& \qquad \qquad \qquad - \log {Z^{{\rho}}_{-\floor{rN^{2/3}}e_1, z}(\tau  = \floor{rN^{2/3}} )}\Big\} \geq  {|\log \delta|^2 \sqrt{\delta} N^{1/3}}\Big).\nonumber
\end{align}
Before we continue our bound, let us simplify our notation. 
For $z\in \cD$ and $i\in [\![-\floor{\alpha rN^{2/3}}+1, \floor{\alpha rN^{2/3}}]\!]$, define horizontal increments  
$$
\widetilde{I}{}^z_{(i,1)} = \frac{Z_{(i-1,1), z}}{ Z_{(i,1), z} }
$$
which live on the horizontal line $y=1$.  With these increments, define a two-sided multiplicative walk $\{M_n^{z}\}_{n\in[\![-\floor{\alpha rN^{2/3}}+1, \floor{\alpha rN^{2/3}}]\!]}$ by setting  $M_{0}^{z} = 1$ and 
\begin{equation}\label{M_walk}
M_n^{z}/M_{n-1}^{z} = I^\rho_{(n,0)}/\widetilde{I}{}^{z}_{(n, 1)} 
\end{equation}
where $I_{(n,0)}^\rho$ are the boundary weights from the stationary polymer in the quadrant $-\floor{rN^{2/3}}e_1 + \mathbb{Z}^2_{\geq 0}$. 
Note that $n=0$ corresponds to $\tau=\floor{r N^{2/3}}$, which is exit at the origin.
Then, $\eqref{cont_bd}$ can be upper bounded as
\begin{align}
\eqref{cont_bd} &= \mathbb{P} \Big( \min_{z\in\cD}  \max_{n \in [\![-\alpha rN^{2/3}+1, \alpha rN^{2/3} ]\!]} \log M_n^{z} \leq 3|\log \delta|^2 \sqrt{\delta} N^{1/3}\Big)\label{asdf1}\\
& \qquad \qquad + \mathbb{P} \Big( \max_{z\in\cD}  \max_{n \in [\![1, \delta N^{2/3} ]\!]} \log M_n^{z} \geq |\log \delta|^2 \sqrt{\delta} N^{1/3}\Big) \label{asdf2}\\
\begin{split}
& \leq \mathbb{P} \Big(  \Big\{ \min_{z\in \cD} \max_{n\in  [\![1, \floor{\frac{1}{2}\alpha rN^{2/3}} ]\!]} \log M_n^{z }\leq 3|\log \delta|^2 \sqrt{\delta} N^{1/3}\Big\}\\
&\quad \quad \quad\quad \quad \quad \quad \quad \quad  \bigcap  \Big\{  \min_{z\in \cD}  \max_{n\in  [\![- \floor{\frac{1}{2}\alpha rN^{2/3}}, 0]\!] } \log M_n^{z }\leq 3|\log \delta|^2 \sqrt{\delta} N^{1/3}\Big\} \Big)
\end{split}\label{withunion}\\
&\qquad \qquad \qquad  + \mathbb{P} \Big( \max_{z\in\cD}  \max_{n \in [\![1, \floor{\delta N^{2/3}} ]\!]} \log M_n^{z} \geq |\log \delta|^2 \sqrt{\delta} N^{1/3}\Big). \label{run_up}
\end{align}
 
For any $z\in \cD$, Lemma \ref{mono_ratio} gives
\begin{align*}
M_n^{z} \geq M_n^{w_N^+} \quad \text{ for $n \geq 1$}
\quad\text{and}\quad 
M_n^{z} \geq M_n^{w_N^-} \quad \text{ for $n \leq 0$}.
\end{align*}
Therefore, we may bound \eqref{withunion} and \eqref{run_up} by
\begin{align}
\eqref{withunion} +  \eqref{run_up} & \leq \mathbb{P} \Big(  \Big\{\max_{n\in  [\![ 1,  \floor{\frac{1}{2}\alpha rN^{2/3}} ]\!]} \log M_n^{w_N^+}\leq 3|\log \delta|^2 \sqrt{\delta} N^{1/3} \Big\} \label{boundpm}\\
& \qquad \qquad \qquad \qquad \qquad \bigcap  \Big\{  \max_{n\in  [\![- \floor{\frac{1}{2}\alpha rN^{2/3}}, 0 ]\!]} \log M_n^{w_N^-}\leq 3 |\log \delta|^2 \sqrt{\delta} N^{1/3} \Big\} \Big) \nonumber\\
& \qquad \qquad \qquad + \mathbb{P} \Big(  \max_{n \in [\![1, \floor{\delta N^{2/3}} ]\!]} \log M_n^{w_N^-} \geq |\log \delta|^2 \sqrt{\delta} N^{1/3}\Big)\label{boundpm1} 
\end{align}

\begin{figure}[t]
\captionsetup{width=0.8\textwidth}
 \begin{center}
 
\begin{tikzpicture}[>=latex, scale=0.7]

\draw[line width = 0.3mm, ->] (-5,0)--(8,0);
\draw[line width = 0.3mm, ->] (-5,0)--(-5,6);


\draw[line width = 0.3mm, ->] (6.7,5.7)--(-0,5.7);
\draw[line width = 0.3mm, ->] (6.7,5.7)--(6.7,2);



\node at (-1, 2.2) { $\widetilde{I}^z, I^{\lambda, \textup{NE}}$ and $I^{\eta, \textup{NE}}$ };

\draw[line width = 0.3mm, ->] (-1, 1.8) -- (-0.1, 0.9);

\draw[line width = 1mm, dotted] (-2.5,0.7)--(2.5,0.7);


\draw[line width = 1mm] (2.5,-0.1) -- (2.5,0.1);
\node at (2.6, -0.35) {\small $\alpha rN^{2/3}$};

\draw[line width = 0.3mm, ->] (1.9, -1.1) -- (1.1, -0.1);
\node at (2.1, -1.3) { $I^\rho$ };

\draw[line width = 1mm] (-2.5,-0.1) -- (-2.5,0.1);
\node at (-2.7, -0.35) {\small $-\alpha rN^{2/3}$};

\node at (3.5,6.1) {\small $I^{\lambda, \textup{NE}}$ and $I^{\eta, \textup{NE}}$};
\node at (8.7,3.5) {\small $J^{\lambda, \textup{NE}}$ and $J^{\eta, \textup{NE}}$ };


\draw[ fill=lightgray](-5,0)circle(1mm);
\node at (-5,-0.5) {$-rN^{2/3}$};

\draw[ fill=lightgray](0,0)circle(1mm);
\node at (-0,-0.4) {$0$};

\draw[ fill=white](6.7,5.7)circle(1mm);
\node at (8.4,5.9) {$v_N + e_1+e_2$};

\draw[ line width = 0.6mm, color=gray] (5,5)--(6,5)--(6,4);

\draw[ fill=lightgray](5,5)circle(1mm);
\node at (4.5,5) {\small $w_N^+$};

\draw[ fill=lightgray](6,5)circle(1mm);

\draw[ fill=lightgray](6,4)circle(1mm);
\node at (6,3.6) {\small $w_N^-$};

\end{tikzpicture}
 
\end{center}
\caption{\small Setup for the stationary polymer with ratios of partition functions.}
\label{fig11aa}
\end{figure}
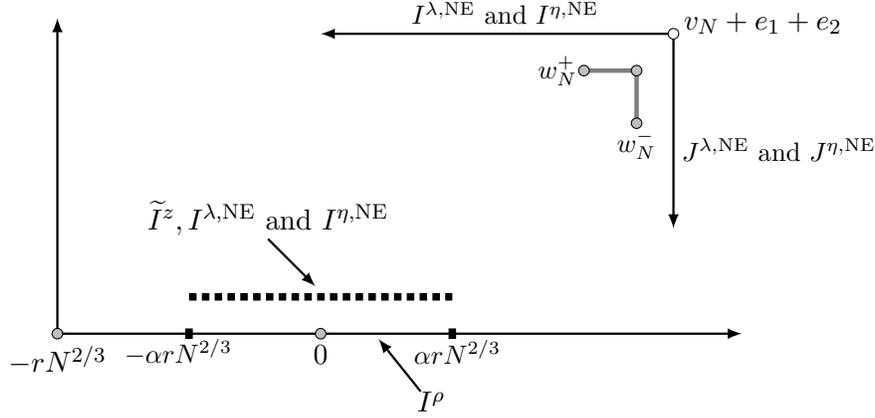

Next, to each edge on the north and east sides of the rectangle $[\![ -\floor{rN^{2/3}}e_1, v_N +e_1+e_2\rzb$, we attach both  $\lambda$- and $\eta$-edge weights, coupled as in Theorem B.4 from \cite{Bus-Sep-22-ejp}. We denote these weights by $I^{\lambda,\textup{NE}}_{v_n+ke_1+e_2}$, $J^{\lambda,\textup{NE}}_{v_n+e_1+ke_2}$, $I^{\eta,\textup{NE}}_{v_n+ke_1+e_2}$, and $J^{\eta,\textup{NE}}_{v_n+e_1+ke_2}$, $k\le1$.   Together with the bulk weights in $[\![ -\floor{rN^{2/3}}e_1+e_2, v_N\rzb$, these define  stationary polymers  with northeast boundary. Let us denote their partition functions by $Z^{\lambda, \textup{NE}}_{x,v_N +e_1+e_2} $ and $Z^{\eta, \textup{NE}}_{x,v_N +e_1+e_2}$ for $x \in \lzb (-\floor{rN^{2/3}}, 1), v_N \rzb$.   The corresponding polymer measures are denoted by $Q^{\lambda, \textup{NE} }_{ x, v_N+e_1+e_2 }$ and $Q^{\eta, \textup{NE} }_{ x, v_N+e_1+e_2 }$, respectively.  This is depicted in Figure \ref{fig11aa}.

On the horizontal line $y=1$, let us also define for $i\in [\![-\floor{\alpha rN^{2/3}}+1, \floor{\alpha rN^{2/3}}]\!]$ 
\be\label{buseincrement}\begin{aligned}
I^{\lambda, \textup{NE}}_{(i,1)} = \frac{Z^{\lambda, \textup{NE}}_{(i-1,1), v_N+ e_1 + e_2}}{ Z^{\lambda, \textup{NE}}_{(i,1), v_N+ e_1 + e_2}}\qquad 
\text{and} \qquad I^{\eta, \textup{NE}}_{(i,1)} = \frac{Z^{\eta, \textup{NE}}_{(i-1,1), v_N+ e_1 + e_2}}{Z^{\eta, \textup{NE}}_{(i,1), v_N+ e_1 + e_2} }.
\end{aligned}\ee

\begin{lemma}\label{lm:A78} 
There exists a positive constant $C$, depending only on $\varepsilon$,
such that for $\alpha$, $r$, $N$ as chosen above, 
and for any integers $a, b\in  [\![-\floor{\alpha rN^{2/3}}+1, \floor{\alpha rN^{2/3}}]\!]$, the event 
\be\label{A78} 
A =\Big\{\tfrac{1}{2} \prod_{i=a}^b I^{\eta, \textup{NE}}_{(i,1)}\leq \prod_{i=a}^b\widetilde{I}{}^{\hspace{1.5pt}w_N^-}_{(i,1)}  \leq \prod_{i=a}^b\widetilde{I}{}^{\hspace{1.5pt}w_N^+}_{(i, 1)} \leq 2\prod_{i=a}^bI^{\lambda, \textup{NE}}_{(i, 1)}\Big\} 
\ee
satisfies $\mathbb{P}(A^c) \leq e^{-Cr^3}$.
\end{lemma}

\begin{proof}
Due to  the relative positions of $w_N^\pm$ and $z$,  Lemma \ref{mono_ratio} implies the middle inequality in the definition of $A$.
We will prove the desired bound for the inequality on the right, i.e.
\beq\mathbb{P}\Big( \prod_{i=a}^b\widetilde{I}{}^{\hspace{1.5pt}w_N^+}_i \leq 2 \prod_{i=a}^bI^{\lambda, \textup{NE}}_i \Big)
\geq 1-e^{-Cr^3}.\label{weshow1}\eeq
The argument for the inequality on the left is similar and will be omitted.

Let $\tau^{\textup{NE}}$ be defined similarly to $\tau$, but acting on down-left paths. Namely, it gives the number of steps the path takes before making its first corner. We will again use the convention that $\tau^{\textup{NE}}>0$ if the first step of the path is $-e_1$ and $\tau^{\textup{NE}}<0$ if the first step is $-e_2$.

Our estimate essentially follows from the following two facts. The first fact is that the random variable
$$Q^{\lambda, \textup{NE} }_{ \floor{\alpha rN^{2/3}}e_1 + e_2 , v_N+e_1+e_2 } \{\tau^{\textup{NE}} \geq qrN^{2/3}\}$$
is, almost surely, less than or equal to
$$Q^{\lambda, \textup{NE} }_{ (a,1), v_N+e_1+e_2 } \{\tau^{\textup{NE}}  \geq qrN^{2/3}\} \qquad \forall a \in [\![-\floor{\alpha rN^{2/3}}+1, \floor{\alpha rN^{2/3}}]\!].$$
This follows directly from Lemma \ref{polymono},
although note that here we exit from the \textup{NE} boundary instead of the SW boundary. 
The second fact is that there exist positive constants $C_1$ and $C_2$ such that 
\begin{equation}\label{lambda_ray}
\mathbb{P}\Big(Q^{\lambda, \textup{NE} }_{ \floor{\alpha rN^{2/3}}e_1 + e_2, v_N+e_1+e_2 }\{\tau^{\textup{NE}}  \geq qrN^{2/3}\} \geq 1-e^{-C_1r^2N^{1/3}}\Big) \geq 1- e^{-C_2r^3}.
\end{equation}
To see this, observe that 
$$\mathbb{P}\Big(Q^{\lambda, \textup{NE} }_{ \floor{\alpha rN^{2/3}}e_1 + e_2, v_N+e_1+e_2 }\{\tau^{\textup{NE}} \leq qrN^{2/3}\} \geq e^{-C_1r^2N^{1/3}}\Big) \leq e^{-C_2r^3}$$
is the same as \eqref{noteasyone}, except here we rotate the picture by $180^\circ$. The key idea is illustrated in Figure \ref{fig111}. Note the similarities between   Figures \ref{figlemmaab} and \ref{fig111}. From Figure \ref{fig111}, the calculation $z_N\cdot e_2 - v_N \cdot e_2 - 1 \geq CrN^{2/3}$ is omitted since it is similar to \eqref{bound_h}. 

Let $Z^{\lambda,\textup{N}}_{(b,1),w_N^++e_2}$ denote the partition function for up-right paths from $(b,1)$ to $w_N^++e_2$, which uses the same weights as $Z^{\lambda,\textup{NE}}_{(b,1),w_N^++e_2}$ does on the north boundary but uses the original (bulk) weights on $w_N^++\Z_{\le0}^2$.

On the high probability event 
\begin{equation}\label{high_p_event}
\Big\{Q^{\lambda, \textup{NE} }_{ \floor{\alpha rN^{2/3}}e_1 + e_2, v_N+e_1+e_2 }\{\tau^{\textup{NE}} \geq qrN^{2/3}\} \geq 1-e^{-C_1r^2N^{1/3}}\Big\},
\end{equation}
we have 
\begin{align*} 
\prod_{i=a+1}^b\widetilde{I}{}^{\hspace{1.5pt}w_N^+}_{(i,1)} &= \frac{Z_{(a, 1), w_N^+}}{Z_{(b, 1), w_N^+}}\\
 & \leq \frac{Z^N_{(a, 1), w_N^+ +e_2}}{Z^N_{(b, 1), w_N^+ +e_2}} \qquad \qquad \qquad  (\text{By Lemma \ref{mono_ratio}})\\
 &= \frac{Z^N_{(a, 1), w_N^+ +e_2} \prod_{i=1}^{\floor{qrN^{2/3}}+1}I^{\lambda, \textup{NE}}_{ v_N+e_1+e_2-ie_1} }{Z^N_{(b, 1), w_N^+ +e_2}\prod_{i=1}^{\floor{qrN^{2/3}}+1}I^{\lambda, \textup{NE}}_{ v_N+e_1+e_2-ie_1}}\\
& = \frac{Z^{\textup{NE}}_{(a, 1), v_N+e_1 +e_2}(\tau^{\textup{NE}} \geq \floor{qrN^{2/3}})} {Z^{\textup{NE}}_{(b, 1), v_N+e_1 +e_2}(\tau^{\textup{NE}} \geq \floor{qrN^{2/3}})}\\
& = \frac{Q^{\textup{NE}}_{(a, 1), v_N+e_1 +e_2}(\tau^{\textup{NE}} \geq \floor{qrN^{2/3}})} {Q^{\textup{NE}}_{(b, 1), v_N+e_1 +e_2}(\tau^{\textup{NE}} \geq \floor{qrN^{2/3}})} \prod_{i=a+1}^bI^{\lambda, \textup{NE}}_{(i,1)} \\
\qquad & \leq \frac{1}{1-e^{-C_1r^2N^{1/3}}}\prod_{i=a+1}^bI^{\lambda, \textup{NE}}_{(i,1)} \qquad (\text{on the event \eqref{high_p_event}}).
\qedhere
\end{align*}
\end{proof}

\begin{figure}[t]
\captionsetup{width=0.8\textwidth}
 \begin{center}
 
\begin{tikzpicture}[>=latex, scale=1.2]

\draw[line width = 0.3mm, ->, gray] (7+4.2,3)--(7+4.2-4,3);
\draw[line width = 0.3mm, ->, gray] (7+4.2,3)--(7+4.2,3-3);

\draw[line width = 0.3mm, loosely dotted, ->] (7+1,0.5) -- (7+1+3,3.5+1.5);
\draw[dotted, color=lightgray, line width = 1.3mm] (7+1,0.5) -- (7+1,1) --(7+1.5, 1) -- (7+1.5,1.5) --(7+2,1.5) -- (7+2,2)--(7+3,2)--(7+3,2.5) -- (7+3.76,2.5)-- (7+3.76,3) -- (7+4.2,3)  ;
\draw[dotted, color=darkgray, line width = 1mm](7+3.5,4.26)--(10.5, 2.5);
\draw[dotted, color=black, line width = 0.5mm] (7+1,0.5) -- (7+1,1) --(7+1.5, 1) -- (7+1.5,1.5) --(7+2,1.5) -- (7+2,2)--(7+3,2)--(7+3,2.5) -- (7+3.5,2.5);

\draw[line width = 0.3mm, ->] (7+3.5,4.26)--(7,4.26);
\draw[line width = 0.3mm, ->] (7+3.5,4.26)--(7+3.5,1.3);

\draw[ fill=black](7+3.5,4.26)circle(1mm);

\node at (11.4,4.9) {\small $\xi[\lambda]$};

\draw[ fill=lightgray](7+4.2,3)circle(1mm);
\node at (7+5.7,3) {\small $b_N = v_N^++e_1+e_2$};

\draw[ fill=white](7+2.65,3)circle(1mm);

\draw[ fill=lightgray](7+3.5,3)circle(1mm);
\node at (7+4.8,3.8) {\small $w_N^++e_2$};

\draw[line width = 0.3mm, ->]  (7+4.5,3.55) --(7+3.6,3.1);

\draw[ fill=lightgray](7+1,0.5)circle(1mm);
\node at (7+2.4 , 0.5) {\small $a_N = (\alpha rN^{2/3},1)$};

\node at (10.4, 4.6) {$z_N$};

\end{tikzpicture}
 
\end{center}
\caption{\small  Illustration of \eqref{lambda_ray}. By Lemma \ref{relatetau}, $Q^{\lambda, \textup{NE}}_{a_N, b_N }(\tau^{\textup{NE}}  \leq qrN^{2/3}) = Q^{\lambda, \textup{NE}}_{ a_N, z_N }(\tau^{\textup{NE}}  <-(z_N \cdot e_2 - v_N \cdot e_2 - 1))$, and this is unlikely because $z_N\cdot e_2 - v_N \cdot e_2 - 1 \geq CrN^{2/3}$.}
\label{fig111}
\end{figure}
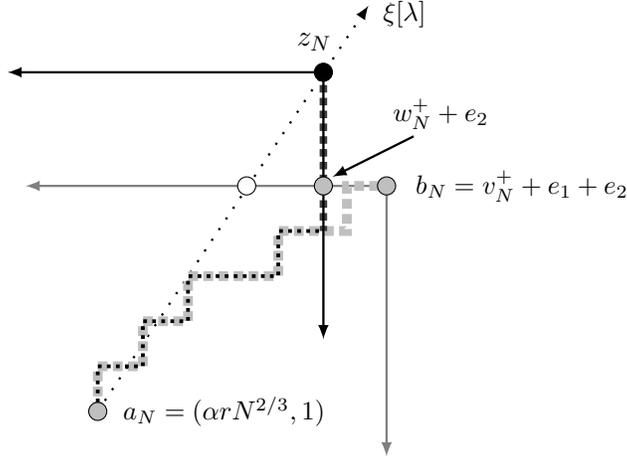

With the new horizontal increments $I^{\lambda,\textup{NE}}$ and $I^{\eta,\textup{NE}}$, define two more two-sided multiplicative random walks $M^{\lambda}_n$ and $M^{\eta}_n$  with  $M^{\lambda}_{0}= M^{\eta}_{0} = 1$,
\begin{align*}
M_n^{\lambda}/M_{n-1}^{\lambda} = I^\rho_{(n,0)}/{I}{}^{\lambda, \textup{NE}}_{(n,1)},\quad\text{and}\quad
M_n^{\eta} /M_{n-1}^{\eta} = I^\rho_{(n, 0)}/{I}{}^{\eta, \textup{NE}}_{(n,1)}.
\end{align*}
On the event $A$ from \eqref{A78}, we get
\begin{align}\label{the_bounds}
\frac{1}{2}M_n^{\lambda} \leq M_n^{w_N^+}  \leq 2M_n^{\eta}  \text{ for $n \geq 1$} 
\quad\text{and}\quad 
\frac{1}{2}M_n^{\eta} \leq M_n^{w_N^-} \leq 2 M_n^{\lambda} \text{ for $n \leq 0$} .
\end{align}
Now we can bound
\begin{align}
\begin{split}
\mathbb{P}(\text{event in }\eqref{boundpm} \cap A)   &\leq  \mathbb{P} \Big(  \Big\{\max_{n\in  [\![ 1,  \floor{\tfrac{1}{2}\alpha rN^{2/3}} ]\!]} \log M_n^{\lambda}\leq 6|\log \delta|^2 \sqrt{\delta} N^{1/3} \Big\}\\
& \qquad \qquad \qquad \bigcap  \Big\{  \max_{n\in  [\![- \floor{\tfrac{1}{2}\alpha rN^{2/3}}, 0 ]\!]} \log M_n^{\eta}\leq 6 |\log \delta|^2 \sqrt{\delta} N^{1/3} \Big\} \Big),
\end{split}\label{boundwithind}\\
\mathbb{P}(\text{event in }\eqref{boundpm1} \cap A) & \leq \mathbb{P} \Big(  \max_{n \in [\![1, \floor{\delta N^{2/3}} ]\!]} \log M_n^{\eta} \geq \tfrac{1}{2}|\log \delta|^2 \sqrt{\delta} N^{1/3}\Big).\label{boundwithind1}
\end{align}
Theorem B.4 from \cite{Bus-Sep-22-ejp} states that the increment  variables $\{I^{\lambda, \textup{NE}}_{(i,1)}\}_{i\geq 1} \cup \{I^{\eta, \textup{NE}}_{(i,1)}\}_{i \leq 0}$ are independent, and these are independent of the boundary weights $\{I^\rho_{(i,0)}\}$ by construction. Thus, we get 
\begin{align}\label{ind_prod}
\begin{split}
\eqref{boundwithind} &\leq \mathbb{P} \Big(\max_{n\in  [\![ 1,  \floor{\tfrac{1}{2}\alpha rN^{2/3}} ]\!]} \log M_n^{\lambda}\leq 6|\log \delta|^2 \sqrt{\delta} N^{1/3} \Big)\\
& \qquad \qquad \qquad \qquad \qquad \times\mathbb{P}\Big(   \max_{n\in  [\![- \floor{\tfrac{1}{2}\alpha rN^{2/3}}, 0 ]\!]} \log M_n^{\eta}\leq 6 |\log \delta|^2 \sqrt{\delta} N^{1/3}  \Big).
\end{split}
\end{align}
The next step is a random walk estimate because the steps of the walks $\log M^\lambda_n$ and $\log M^\eta_n$ are given by the difference of two independent log-gamma random variables, which are sub-exponential random variables.  Using Proposition \ref{rwest}, we see that 
$\text{\eqref{ind_prod}}\leq C |\log \delta|^6 \delta$.
Using Theorem \ref{max_sub_exp}, we also have
$\text{\eqref{boundwithind1}}\leq C \delta.$
 
 To summarize, we have shown 
\begin{align}\label{D_region}
\begin{aligned}
\mathbb{P} (\text{event in \eqref{mainest2}})
&\leq 2\mathbb{P} (A^c) + \mathbb{P}(\text{event in }\eqref{boundpm} \cap A)+ \mathbb{P}(\text{event in }\eqref{boundpm1} \cap A)\\
& \leq 2e^{-C|\log \delta|^3} + C|\log \delta|^6 \delta \\
& \leq  |\log \delta|^{10} \delta.
\end{aligned}
\end{align}
This completes the proof of the desired bound \eqref{dupper:claim1} with the maximum taken over the dark region $\cD\subset\partial^{\textup{NE}}\lzb 0, v_N \rzb$ in Figure \ref{fig9}. 

For the endpoints in $\mathcal{L}^+$, we have the following estimate,  
\begin{align*}
&\mathbb{P} \Big(\max_{ z \in \mathcal{L}^+} Q^\rho_{0,z}\{1\leq \tau \leq \delta N^{2/3}\} \geq e^{-|\log \delta|^2 \sqrt{\delta} N^{1/3}} \Big) \\
 & \leq \mathbb{P} \Big(\max_{ z \in \mathcal{L}^+} Q^\rho_{0,z}\{1\leq \tau \} \geq e^{-|\log \delta|^2 \sqrt{\delta} N^{1/3}} \Big)\\
& \leq \mathbb{P} \Big( Q^\rho_{0,w_N^+}\{1\leq \tau \} \geq e^{-|\log \delta|^2 \sqrt{\delta} N^{1/3}} \Big) \qquad (\text{by Lemma \ref{polymono}})\\
 &\leq e^{-C|\log \delta|^3} \qquad \qquad \qquad \qquad \qquad \qquad \quad \ (\text{by Corollary \ref{sec3cor}}).
\end{align*}

\begin{figure}[t]
\captionsetup{width=0.8\textwidth}
\begin{center}
\begin{tikzpicture}[scale = 1]

\draw[gray ,dotted, line width=0.3mm](1,-1)--(4,3);
\draw[gray ,dotted, line width=0.3mm](0,0)--(4,16/3);

\draw[gray, line width=0.3mm, ->] (0,0) -- (6,0);
\draw[gray, line width=0.3mm, ->] (0,0) -- (0,4);

\draw[ line width=0.3mm, ->] (1,-1) -- (5,-1);
\draw[ line width=0.3mm, ->] (1,-1) -- (1,2);

\draw[ fill=white](1.75, 0)circle(1.3mm);
\draw[ fill=black ](1,0)circle(1.3mm);

\draw[gray ,dotted, line width=1.1mm] (0,0) -- (0.7,0) -- (0.7,0.5) --(1,0.5) ;

\draw[lightgray ,dotted, line width=1.4mm] (1,0.5) --(2,0.5)-- (2,1) -- (4, 1) -- (4,3) ;
\draw[black ,dotted, line width=0.6mm] (1,0)  --(1,0.5) --(2,0.5)-- (2,1) -- (4, 1) -- (4,3) ;

\draw[black ,dotted, line width=1.1mm] (1,-1) -- (1,0.5);

\draw[ fill=lightgray](0,0)circle(1mm);
\node at (-0.3,-0.4) {$(0,0)$};

\draw[ fill=lightgray](1,-1)circle(1mm);
\node at (0.7,-1.4) {$(\delta  N^{2/3},-h)$};

\fill[color=white] (4,3)circle(1.7mm); 
\draw[ fill=lightgray](4,3)circle(1mm);
\node at (4.5,2.8) {$w_N^-$};

\fill[color=white] (4,16/3)circle(1.7mm); 
\draw[ fill=lightgray](4,16/3)circle(1mm);
\node at (4.4,16/3+0.3) {$v_N$};

\draw [decorate,decoration={brace,amplitude=10pt, mirror},  xshift=0pt,yshift=0pt]
(4.2,3) -- (4.2,16/3) ;

\node at (6.1,4.2) {$qrN^{2/3}$};

\end{tikzpicture}
\end{center}
\caption{\small We have $Q_{0,w_N^-}\{\tau\leq \delta  N^{2/3}\}$ $= Q_{(\floor{\delta  N^{2/3}}, -h), w_N^-}\{\tau < -h\}$ which is rare because $h$ is lower bounded by $CrN^{2/3}$. The lower bound on $h$  follows from the fact the vertical distance between $v_N$ and $w_N^-$ is of order $rN^{2/3}$.}
\label{figlminus}
\end{figure}
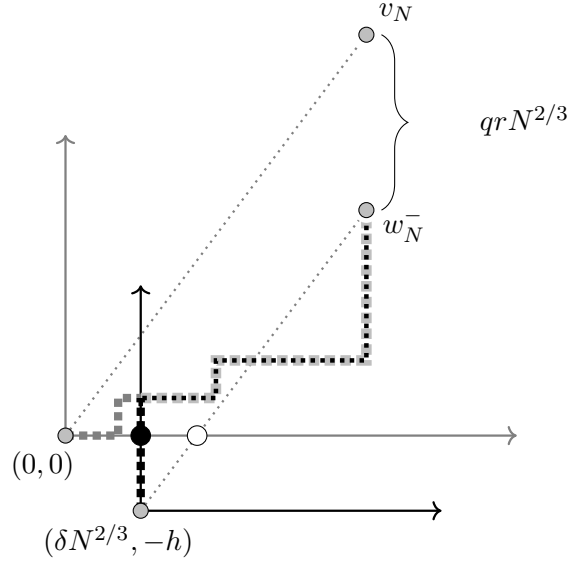

Similarly, for the $\mathcal{L}^-$ region,  we have
\begin{align*}
\mathbb{P} \Big(\max_{ z \in \mathcal{L}^-} Q^\rho_{0,z}\{1\leq \tau \leq \delta N^{2/3}\} \geq  e^{-|\log \delta|^2 \sqrt{\delta} N^{1/3}} \Big)
&\leq \mathbb{P} \Big(\max_{ z \in \mathcal{L}^-} Q^\rho_{0,z}\{\tau \leq \delta N^{2/3}\} \geq  e^{-|\log \delta|^2 \sqrt{\delta} N^{1/3}} \Big)\\
&\leq \mathbb{P} \Big( Q^\rho_{0,w_N^-}\{\tau \leq \delta N^{2/3}\}\geq  e^{-|\log \delta|^2 \sqrt{\delta} N^{1/3}} \Big)\\ &\leq e^{-C|\log \delta|^3}.
\end{align*}
The idea for the last inequality is illustrated in Figure \ref{figlminus}, essentially again following from Lemma \ref{polymono} and Corollary \ref{sec3cor}. This finishes the argument for the $\mathcal{L}^-$ region. The bound \eqref{dupper:claim1} is thus proved.

The probability bound implies the upper bound in \eqref{dupper:claim2}:
$$\mathbb{E}\Big[\max_{z\in \partial^{\textup{NE}} [\![0, v_N]\!]} Q^\rho_{0,z}\{|\tau|\leq \delta N^{2/3}\} \Big] \leq \delta + \mathbb{P}\Big(\max_{z\in \partial^{\textup{NE}}[\![0, v_N]\!]} Q^\rho_{0,z}\{|\tau|\leq \delta N^{2/3}\} \geq \delta \Big) \leq C|\log \delta|^{10}\delta.$$

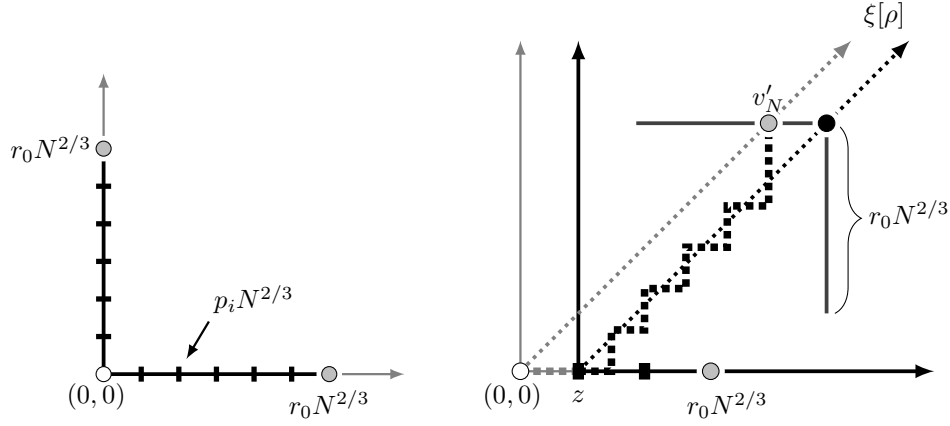
\begin{figure}[t]
\captionsetup{width=0.8\textwidth}
\begin{center}
 
\begin{tikzpicture}[>=latex, scale=1]

\draw[color = gray, line width = 0.3mm, ->] (0,0)--(0,4);
\draw[color = gray, line width = 0.3mm, ->] (0,0)--(4,0);

\draw[color = black, line width = 0.5mm] (0,0)--(0,3);
\draw[color = black, line width = 0.5mm] (0,0)--(3,0);

\draw[color = black, line width = 0.7mm] (0.5,-0.1)--(0.5,0.1);
\draw[color = black, line width = 0.7mm] (1,-0.1)--(1,0.1);
\draw[color = black, line width = 0.7mm] (1.5,-0.1)--(1.5,0.1);
\draw[color = black, line width = 0.7mm] (2,-0.1)--(2,0.1);
\draw[color = black, line width = 0.7mm] (2.5,-0.1)--(2.5,0.1);

\draw[color = black, line width = 0.7mm] (-0.1,0.5)--(0.1,0.5);
\draw[color = black, line width = 0.7mm] (-0.1,1)--(0.1,1);
\draw[color = black, line width = 0.7mm] (-0.1,1.5)--(0.1,1.5);
\draw[color = black, line width = 0.7mm] (-0.1,2)--(0.1,2);
\draw[color = black, line width = 0.7mm] (-0.1,2.5)--(0.1,2.5);

\node at (2,1) {\small $p_iN^{2/3}$};
\draw[color = black, line width = 0.3mm, ->] (1.4,0.7)--(1.1,0.2);

\fill[color=white] (3,0)circle(1.7mm); 
\draw[ fill=lightgray](3,0)circle(1mm);
\node at (3,-0.4) {\small $r_0N^{2/3}$};

\fill[color=white] (0,3)circle(1.7mm); 
\draw[ fill=lightgray](0,3)circle(1mm);
\node at (-0.7,3) {\small $r_0N^{2/3}$};

\draw[ fill=white ](0,0)circle(1mm);
\node at (-0.1,-0.3) {\small $(0,0)$};

\end{tikzpicture}
\qquad
\begin{tikzpicture}[>=latex, scale=1.1]

\draw[color = gray, line width = 0.3mm, ->] (0,0)--(0,4);
\draw[color = gray, line width = 0.3mm] (0,0)-- (0.7,0);
\draw[color = black, line width = 0.5mm, ->] (0.7,0)-- (5,0);

\draw[line width = 0.5mm,  ->] (0.7,0)--(0.7,4);

\draw[color=gray, dotted, line width = 1mm] (0,0) -- (0.7,0);

\draw[color = darkgray,  line width = 0.5mm,] (3.7-2.3,3)--(3.7,3) -- (3.7,3 -2.3);

\draw[color=black, dotted, line width = 1mm]  (0.7,0)--(1.1, 0)--(1.1,0.5) -- (1.5,0.5) -- (1.5,1)--(2,1)--(2,1.5)--(2.5,1.5) -- (2.5,2) -- (3,2) -- (3,3) ;
\draw [decorate,decoration={brace,amplitude=10pt}, xshift=0pt,yshift=0pt]
(3.8,3) -- (3.8,0.7) ;
\node at (4.7,1.9) {\small $r_0 N^{2/3} $};

\draw[color=gray, dotted, line width = 0.5mm, ->]  (0,0) -- (4,4);

\draw[color=black , dotted, line width = 0.5mm, ->]  (0.7,0) -- (4.7,4);

\node at (4.4,4.3 ) {\small $\xi[\rho]$};

\draw[line width = 1.5mm] (0.7,-0.1)--(0.7,0.1);
\draw[line width = 1.5mm] (1.5,-0.1)--(1.5,0.1);

\fill[color=white] (2.3,0)circle(1.7mm); 
\draw[ fill=lightgray](2.3,0)circle(1mm);
\node at (2.5,-0.4) {\small $r_0N^{2/3}$};

\fill[color=white] (3,3)circle(1.7mm); 
\draw[ fill=lightgray](3,3)circle(1mm);

\fill[color=white] (3.7,3)circle(1.7mm); 
\draw[ fill=black](3.7,3)circle(1mm);

\draw[ fill=white ](0,0)circle(1mm);
\node at (-0.1,-0.3) {\small $(0,0)$};

\node at (0.7,-0.3) {\small $z$};

\node at (3,3.3) {\small $v_N'$};

\end{tikzpicture}
 
\end{center}
\caption{\small  Left: Partition for the collection of paths in \eqref{exit47}.   The origin is not necessarily a partition point. Right: An illustration for \eqref{justify}. The nested polymer with its quenched measure $Q^{(0)}_{z, v_N'}$ is shown in black.}
\label{fig13}
\end{figure}

We turn to the lower bound in \eqref{dupper:claim2}. Utilizing  the 
 proof of Lemma \ref{rtail_lb}, we will show that we can fix two constants $r_0$ and  $N_0$ (depending on $\varepsilon$)  such that, for $N\geq N_0$,  
\beq \label{exit47} \mathbb{E} \bigl[ Q^{{\rho}}_{0, v_N+e_1+e_2}\{|\tau | \leq r_0 N^{2/3}\}\bigr]\geq \frac{1}{2}.\eeq
Abbreviate   $v_N'=v_N+e_1+e_2 $.  
Given $\delta\geq N^{-2/3}$,  partition $[-r_0,r_0]$ as 
\[    -r_0=p_0   < p_1 < \dotsm  <  p_{\floor{\frac{2r_0}{\delta}}}< p_{\floor{\frac{2r_0}{\delta}}+ 1} = r_0  \]
with mesh size $p_{i+1}-p_i\le \delta$. 
See the left side of Figure \ref{fig13}.
By \eqref{exit47} 
 there exists an integer $i^\star\in[ 0, \floor{\frac{2r_0}{\delta}}]$ such that 
\beq \label{istar}\mathbb{E} \big[Q^\rho_{0, v_N'}\{   p_{i^\star} N^{2/3} \le \tau \le  p_{i^\star+1}N^{2/3}\} \big] \geq \frac{\tfrac{1}{2} \delta}{2r_0} =  C(\varepsilon) \delta.\eeq

Since we cannot control the exact location of $i^\star$, we compensate by varying the endpoint around $v_N'$. 
Let \[  A_N=\lzb v_N' -r_0 N^{2/3}e_1, v_N'\rzb \cup \lzb v_N'-r_0 N^{2/3}e_2, v_N'\rzb \]
denote   the set of lattice points on the boundary  of the rectangle $\lzb  0, v_N'\rzb $ within distance  $r_0 N^{2/3}$ of the upper right corner $v_N'$.   We claim  that for any integer $i\in [0,  \floor{\frac{2r_0}{\delta}}]$, 
\begin{align}  
\label{alli}
\mathbb{E} \Big[\max_{z \in A_N} Q^\rho_{0, z}\{ |\tau|\leq \delta N^{2/3}\}\Big] 
 \geq \mathbb{E} \big[Q^\rho_{0, v_N'}\{   p_{i^\star} N^{2/3} \le \tau \le  p_{i^\star+1}N^{2/3}\} \big].
\end{align}
Then bounds  \eqref{istar} and  \eqref{alli} imply
\beq\label{newest}\mathbb{E} \Big[\max_{z \in A_N} Q^\rho_{0, z}\{ |\tau|\leq \delta N^{2/3}\}\Big]  \geq C(\rho)\delta,\eeq
and the lower bound in \eqref{dupper:claim2} follows directly from \eqref{newest}.

 
It remains to prove claim \eqref{alli}. If $p_{i^\star}\le 0 \le  p_{i^\star+1}$,  \eqref{alli} is immediate.   We argue the case $p_{i^\star+1} > p_{i^\star} > 0$, the other one being analogous.  
 Set  $z=(\floor{p_{i^\star} N^{2/3}} - 1)e_1$ and 
 apply Lemma \ref{nestedpoly} to the polymer with the nested quenched measure $Q^{(0)}_{z, \,\bbullet}$. See the right side of Figure \ref{fig13}. Then 
 \begin{align}
 &\mathbb{E} \bigl[Q^{{\rho}}_{0, v_N'}\{p_{i^\star} N^{2/3} \le \tau \le  p_{i^\star+1}N^{2/3} \} \bigr]\notag\\
&\le  \mathbb{E} \bigl[Q^{(0)}_{z, v_N'}\{1 \leq \tau \leq \delta N^{2/3}\}\bigr] \label{justify}\\
&= \mathbb{E} \bigl[Q^{{\rho}}_{0, v_N'-(\floor{p_{i^\star} N^{2/3}} - 1)e_1}\{1 \leq \tau\leq \delta N^{2/3}\}\bigr] ]\qquad\text{(by shift-invariance)}\notag\\
 & \leq \mathbb{E}\Big[\max_{z \in A_N}   Q^{{\rho}}_{0, z}\{|\tau| \leq \delta N^{2/3}\}\Big].\notag
\end{align}
Theorem \ref{dupper} is proved.\hfill\qed

\subsection{Coupled polymer measures}
\begin{proof}[Proof of Theorem \ref{joint_r}]
From Theorem \ref{r_up_low}, there exists an  event $A$ with probability at least $e^{-C_1r^3}$ such that on $A$, we have 
\begin{align*}
\min_{x\in \partial^{\textup{NE}}[\![0, v_N]\!]}  Q^{{\rho}}_{0,x} \{|\tau|> rN^{2/3}\} &\geq 1-e^{-C_2r^2 N^{1/3}}.
\end{align*}
By a union bound, on the event $A$ we have 
$$ \wt{Q}^{{\rho}}_{0,\partial^{\textup{NE}}[\![0, v_N]\!]} \Big(\bigcap_{x\in  \partial^{\textup{NE}}[\![0, v_N]\!]}\{|\wt\tau_{0,x}|> rN^{2/3} \}\Big) \geq 1-N e^{-C_2 r^2N^{1/3}} \geq 1-e^{-C_3 r^2N^{1/3}}$$
provided that $r_0, N_0$ are sufficiently large. With this, we have finished the proof of this theorem.

\end{proof}

\begin{proof}[Proof of Theorem \ref{joint_d}]
By Theorem \ref{dupper}, on the high probability event $B$ with probability at least $ 1-C_1\delta |\log \delta|^{10}$, we have 
\begin{align*}
\max_{x\in \partial^{\textup{NE}} [\![0, v_N]\!]} Q^\rho_{0,x}\{|\tau|\leq \delta N^{2/3}\} & \leq e^{-|\log \delta|^2\sqrt{\delta}N^{1/3}}.
\end{align*}
With the assumption that $\sqrt{\delta} N^{1/3} \geq 1$, a union bound implies that on $B$, 
$$\wt{Q}_{0,\partial^{\textup{NE}} [\![0, v_N]\!]}\Big(\bigcup_{x\in \partial^{\textup{NE}}[\![0, v_N]\!]}\{\wt \tau_{0,x} \leq \delta N^{2/3}\}\Big) \leq Ne^{-|\log \delta|^2 \sqrt{\delta}N^{1/3}} \leq \delta.$$
The claim of the theorem follows.
\end{proof}

\section{Coalescence of semi-infinite polymers}\label{dual_coal}

In this section, we will define the semi-infinite polymer measures and prove Theorems \ref{d_coal} and \ref{r_coal} about their coalescence. The proof will use a duality between forward and backward polymer measures, which we describe in Section \ref{intro_dual}.


\subsection{Busemann functions and semi-infinite polymers}\label{Bus+Poly}
Following Theorem 4.1 from \cite{Geo-etal-15}, for any fixed $\rho\in(0,\mu)$, $\mathbb P$-almost surely, the limits 
\begin{align}
&B^\rho(x,y) = \lim_{N\to \infty} \bigl(\log Z_{x, v_N}-\log Z_{y, v_N}\bigr),\label{Bdef}
\end{align}
exist for any $x,y\in\mathbb Z^2$ and satisfy
\[Y_z^{-1}=e^{-B^\rho(z,z+e_1)}+e^{-B^\rho(z,z+e_2)}\]
and 
\[B^\rho(x,y)+B^\rho(y,z)=B^\rho(x,z),\]
for all $x,y,z\in\mathbb Z^2$.
Furthermore, for any $z\in\mathbb Z^2$, $I^{\rho}_z=e^{B^\rho(z-e_1,z)}\sim\text{Ga}^{-1}(\mu-{{\rho}})$, $J^{\rho}_z=e^{B^\rho(z-e_2,z)}\sim\text{Ga}^{-1}(\rho)$, and 
if we fix any vertex $v\in \mathbb{Z}^2$, then the weights 
$Y_z, I^{{\rho}}_{v-ke_1}, J^{{\rho}}_{v-ke_2}$, $z\in v-\mathbb{Z}^2_{>0}$, $k\ge0$, are mutually independent and thus define a stationary polymer with northeast boundary on $v - \mathbb{Z}^2_{\geq 0}$.  The partition function and quenched polymer measure will be denoted by $Z^{{{\rho}}, \textup{NE}}_{\bbullet, v}, Q^{{{\rho}}, \textup{NE}}_{\bbullet, v}$.
Similarly, if we define 
\[\wc Y^\rho_z=\frac1{e^{-B^\rho(z-e_1,z)}+e^{-B^\rho(z-e_2,z)}},\quad z\in\mathbb Z^2\,,\]
then $\wc Y^\rho_z\sim\text{Ga}^{-1}(\mu)$ for all $z\in\mathbb Z^2$, and for any vertex $v\in \mathbb{Z}^2$ the weights 
$\wc Y^{{\rho}}_z, I^{{\rho}}_{v+ke_1}, J^{{\rho}}_{v+ke_2}$, $z\in v+\mathbb Z^2_{>0}$, $k\ge1$, are mutually independent and defined a stationary polymer with southwest boundary on $v + \mathbb{Z}^2_{\geq 0}$.  The partition function and quenched polymer measure will be denoted by $\wc Z^{{{\rho}}, SW}_{v, \bbullet}, \wc Q^{{{\rho}}, SW}_{v, \bbullet}$. 
Thus, for any $v\in\mathbb Z^2$, $\wc Q^\rho_{v,\bbullet}$ has the same distribution as the generic $Q^\rho_{v,\bbullet}$ we introduced in Section \ref{stat_poly} and used in Section \ref{exit_sec}. (This distributional equality is a special feature of the inverse-gamma polymer.)

The $\xi[\rho]$-directed (forward) semi-infinite polymer measure starting at $z$, denoted by $\Pi^{{{\rho}}}_{z}$, is a Markov chain on $\mathbb Z^2$ with transition probabilities 
\begin{align}
\begin{split}
\pi^{{\rho}}(x, x+e_1) 
&= \frac{J^{{\rho}}_{x+e_2}}{I^{{\rho}}_{x+e_1}+ J^{{\rho}}_{x+e_2}}=Y_x\, e^{-B^\rho(x,x+e_1)},\\ 
\pi^{{\rho}}(x, x+e_2) 
&= \frac{I^{{\rho}}_{x+e_1}}{I^{{\rho}}_{x+e_1}+ J^{{\rho}}_{x+e_2}}=Y_x\, e^{-B^\rho(x,x+e_2)}\,. 
\end{split}
\label{polymerRWRE}
\end{align}

The $\xi[\rho]$-directed backward semi-infinite polymer measure starting at $z$, denoted by $\wc\Pi^{{{\rho}}}_{z}$, is a Markov chain on $\mathbb Z^2$ with transition probabilities 
\beq \wc\pi^{{\rho}}(x, x-e_1) = \frac{J^{{\rho}}_{x}}{I^{{\rho}}_{x}+ J^{{\rho}}_{x}}=\wc Y^\rho_x\, e^{-B^\rho(x-e_1,x)}\quad\text{and}\quad \wc\pi^{{\rho}}(x, x-e_2) = \frac{I^{{\rho}}_{x}}{I^{{\rho}}_{x}+ J^{{\rho}}_x}=\wc Y^\rho_x\,e^{-B^\rho(x-e_2,x)}\,. \label{backpolymerRWRE}\eeq

The next proposition relates the semi-infinite polymers to the stationary ones.  
For $u\in\mathbb Z^2$ and $v\in u+\mathbb Z^2_{>0}$ let $\Pi^\rho_{u,v}$ be the distribution of the Markov chain that starts at $u$,  has transition probabilities $\pi^\rho(x,x+e_i)$, $i\in\{1,2\}$, if $x\in [\![u, v-e_1-e_2]\!]$, and when it gets to $v-\Z_{>0}e_i$, $i\in\{1,2\}$, it takes $e_i$ steps to get to $v$ and end there. Similarly, let $\wc\Pi^\rho_{v,u}$ be the distribution of the Markov chain that starts at $v$,  has transition probabilities $\wc\pi^\rho(x,x-e_i)$, $i\in\{1,2\}$, if $x\in [\![u+e_1+e_2, v]\!]$, and when it gets to $u+\Z_{>0}e_i$, $i\in\{1,2\}$, it takes $-e_i$ steps to get to $u$ and end there. 

Define, similarly to $\mathbb{X}_{u,v}$, the set $\mathbb{X}_{v,u}$ of down-left paths starting at $v$ and ending at $u$. For $x_\bbullet\in\mathbb{X}_{u,v}$, respectively $\in\mathbb{X}_{v,u}$, let $\bar x_\bbullet\in\mathbb{X}_{v,u}$, respectively $\in\mathbb{X}_{u,v}$, be the path that traverses $x_\bbullet$ in the reverse direction.


\begin{proposition} \label{stat_iid}
We have $\mathbb P$-almost surely, for any $u\in\mathbb Z^2$ and $v\in u+\mathbb Z^2_{>0}$, for any $x_\bbullet\in\mathbb{X}_{u,v}$, 
$$\Pi^{{\rho}}_{u,v}(x_\bbullet) =  Q^{{{\rho}}, \textup{NE}}_{u, v}(x_\bbullet)
\quad\text{and}\quad\wc\Pi^{{\rho}}_{v,u}(\bar x_\bbullet) =  \wc Q^{{{\rho}}, \textup{SW}}_{u, v}(x_\bbullet).$$
\end{proposition}

\begin{proof}
We prove the second claim, the first one being similar. Let $\ell=|v-u|_1$ and index the path $x_\bbullet$ so that $x_0=u$ and $x_\ell=v$. We will consider the case where $x_1=e_1$ and the proof in the other case is identical. Let $k\ge1$ be such that $x_k=u+ke_1$ and $x_{k+1}=u+ke_1+e_2$. Then
\begin{align*}
\wc\Pi^{{\rho}}_{v,u}(\bar x_\bbullet)
&=\prod_{i=k}^{\ell-1}\wc\pi^\rho(x_{i+1},x_i)
=\prod_{i=k}^{\ell-1}\wc Y^\rho_{x_{i+1}}e^{-B^\rho(x_i,x_{i+1})}\\
&=e^{-B^\rho(x_k,v)}\prod_{i=k}^{\ell-1}\wc Y^\rho_{x_{i+1}}
=e^{-B^\rho(u,v)}\prod_{i=1}^{k}I^\rho_{u+ie_1}\prod_{i=k}^{\ell-1}\wc Y^\rho_{x_{i+1}}\,.
\end{align*}
Adding the above over all paths $x_\bbullet\in\mathbb X_{u,v}$ gives 
    \[1=e^{-B^\rho(u,v)}\wc Z^{\rho,\textup{SW}}_{u,v}.\]
Consequently, 
\[\wc\Pi^{{\rho}}_{v,u}(\bar x_\bbullet)
=\frac{\prod_{i=1}^{k}I^\rho_{u+ie_1}\prod_{i=k}^{\ell-1}\wc Y^\rho_{x_{i+1}}}{\wc Z^{\rho,\textup{SW}}_{u,v}}=\wc Q^{\rho,\textup{SW}}_{u,v}(x_\bbullet).\qedhere\]
\end{proof}

\subsection{Coupling the forward and backward semi-infinite polymers}\label{intro_dual}  

We now couple the polymer measures $\{\Pi^\rho_z:z\in\mathbb{Z}^2\}$ following the construction in Appendix A of \cite{Jan-Ras-20-aop}.
To this end, introduce a collection of i.i.d.\ $\text{Uniform}[0,1]$ random variables $\{\theta_z\}_{z\in \mathbb{Z}^2}$ which are also independent of the random environment $\{Y_z:z\in\mathbb Z^2\}$. Let $\mathbf{P}$ denote the distribution of $\theta$.

Define a directed random graph $g^{{\rho}}$ on $\mathbb{Z}^2$, according to the following rule 
\begin{align*}
g^{{\rho}}(x) = \begin{cases} 
      e_1 & \text{ if } \theta_x\leq   \frac{J^{{\rho}}_{x+e_2}}{I^{{\rho}}_{x+e_1}+ J^{{\rho}}_{x+e_2}},\\[10pt]
      e_2 & \text{ if } \theta_x >\frac{I^{{\rho}}_{x+e_1}}{I^{{\rho}}_{x+e_1}+ J^{{\rho}}_{x+e_2}} \,.
   \end{cases}
\end{align*}
From $g^{{\rho}}$, we can construct a semi-infinite path $X^{{{\rho}}, z}_\bbullet$ defined by 
\beq X^{{{\rho}}, z}_0 = z \quad \text{ and }\quad X^{{{\rho}}, z}_{n+1} = X^{{{\rho}}, z}_n + g^{{\rho}}(X^{{{\rho}}, z}_n).\eeq
It is clear from the construction that for $\mathbb P$-almost every $Y_\bbullet$, the distribution of $X^{{{\rho}}, z}_\bbullet$ under $\mathbf P$ is exactly $\Pi^\rho_z$.  Namely, we have $\mathbb P$-almost surely, for any $z\in\mathbb Z^2$ and any finite up-right path $x_\bbullet$ starting at $z$,
\begin{equation}\label{unif_poly}
\mathbf{P}\{X^{{{\rho}}, z}_\bbullet = x_\bbullet\} = \Pi^{{{\rho}}}_{ z}\{x_\bbullet\}.
\end{equation}

We next couple the backward semi-infinite polymer measures together with the forward ones.
To this end, define another (dual) random graph $\wc g^{{\rho}}$ by
$$\wc g^{{\rho}}(x) = \begin{cases} 
      -e_1 & \text{ if }  g^{{\rho}}(x-e_1-e_2) = e_1,\\
      -e_2 & \text{ if } g^{{\rho}}(x-e_1-e_2) = e_2. 
   \end{cases}$$
Define the down-left semi-infinite paths $\wc X^{{{\rho}}, z}$ according to 
\beq\wc X^{{{\rho}}, z}_0 = z \quad \text{ and } \quad \wc{X}^{{{\rho}}, z}_{n+1} = \wc{X}^{{{\rho}}, z}_{n} + \wc g^{{\rho}}(\wc X^{{{\rho}}, z}_{n}).\eeq
By construction, for $\mathbb P$-almost every $Y_\bbullet$, the distribution of $\wc X^{{{\rho}}, z}_\bbullet$ under $\mathbf P$ is that of a Markov chain on $\mathbb Z^2$ with steps in $\{-e_1,-e_2\}$ and  transition probabilities 
    \begin{align*}
        \frac{J^\rho_{x-e_1}}{I^\rho_{x-e_2}+J^\rho_{x-e_2}}
        &=\frac{e^{B^\rho(x-e_1-e_2,x-e_1)}}{e^{B^\rho(x-e_1-e_2,x-e_2)}+e^{B^\rho(x-e_1-e_2,x-e_1)}}
        =\frac{e^{-B^\rho(x-e_1,x)}}{e^{-B^\rho(x-e_2,x)}+e^{-B^\rho(x-e_1,x)}}\\
        &=\frac{e^{B^\rho(x-e_2,x)}}{e^{B^\rho(x-e_1,x)}+e^{B^\rho(x-e_2,x)}}=\wc\pi^\rho(x,x-e_1)
        \end{align*}
to go from $x$ to $x-e_1$ and, similarly,
    \[\frac{I^\rho_{x-e_2}}{I^\rho_{x-e_2}+J^\rho_{x-e_2}}=\wc\pi^\rho(x,x-e_2)\]
to go from $x$ to $x-e_2$.

\begin{remark}\label{duality}
Note that the graph $g^\rho$ and its coupled paths $\{X^{\rho,z}_\bbullet:z\in\mathbb Z^2\}$ are constructed to form a forest that covers all of $\mathbb Z^2$. By Theorem A.2 in \cite{Jan-Ras-20-aop}, this forest is in fact a spanning tree, with probability 1 under $\mathbb P$.
The paths $\{\wc X^{\rho,z}_\bbullet-(e_1+e_2)/2:z\in\mathbb Z^2\}$ form the dual forest that spans the dual lattice $\mathbb Z^2-(e_1+e_2)/2$. Again, by Theorem A.2 in \cite{Jan-Ras-20-aop}, this dual forest is also a spanning forest $\mathbb P$-almost surely.
\end{remark}


For $z\in\mathbb Z^2_{>0}$ let $\wt X^{{{\rho}}, z}_\bbullet \in \mathbb{X}_{z, 0}$ be the random path that follows $\wc X^{{{\rho}}, z}_\bbullet$ from $z$ until the first time it hits the axes $\Z_{>0}e_i$, $i\in\{1,2\}$, and then goes to $0$ taking only $-e_1$ or only $-e_2$ steps. 
For $A\subset \mathbb Z^2_{>0}$ let $\wt Q^\rho_{0,A}$ be the distribution under $\mathbf P$ of the paths $\{\wt X^{{{\rho}}, z}_\bbullet: z\in A\}$. By Proposition \ref{stat_iid}, this is a coupling of the measures $\{\wc Q^{\rho,\textup{SW}}_{0,v}:v\in A\}$ and by their construction, the paths $\{\wt X^{{{\rho}}, z}_\bbullet: z\in A\}$ are $\wt Q^\rho_{0,A}$-almost surely ordered.

\subsection{Proofs of Theorems \ref{d_coal}, \ref{r_coal}, and \ref{r_coal_tail}, and Corollary \ref{h_start}}\label{coal_proof}

We note that the probability $\ncouple^\rho_{a,b}\bigl(\Gamma^{A}\bigr)$ 
is the same as the probability under $\mathbf{P}$  that the coalescence point of the coupled paths $X^{{{\rho}}, a}_\bbullet$ and 
$X^{{{\rho}}, b}_\bbullet$ belongs to $A$. 

\begin{figure}[t]
\captionsetup{width=0.8\textwidth}
\begin{center}
 
\begin{tikzpicture}[>=latex, scale=1.3]

\draw[line width = 0.3mm, ->] (0,0)--(0,4);
\draw[line width = 0.3mm, ->] (0,0)--(5,0);

\draw[color=gray, dotted, line width = 1mm, ->] (1,0)--(1,0.5) -- (1.5,0.5) -- (1.5,1)--(2,1)--(2,1.5)--(2.5,1.5) -- (2.5,3) ;

\draw[color=gray, dotted, line width = 1mm, ->] (0,1)--(1,1) -- (1,1.8) -- (1.7,1.8)--(1.7,3.5) -- (2.6,3.5);

\draw[color=gray, dotted, line width = 0.5mm]  (3,0) -- (3,2)--(0,2);

\draw[ dotted, line width = 1mm] (2.1,2)--(2.1, 1.65) -- (1.6, 1.65) -- (1.6, 1.2) --(1.2,1.2) -- (1.2,0.7) -- (0.4,0.7) -- (0.4, 0.5) -- (-0.1, 0.5)--(-0.1,-0.1);

\draw[ fill=black](2.1,2)circle(1mm);
\node at (2.1,2.25) {\small $x^*$};

\fill[color=white] (0,1)circle(1.7mm); 
\draw[ fill=lightgray](0,1)circle(1mm);
\node at (-0.9,1) {\small $(0, \delta N^{2/3})$};

\fill[color=white] (1,0)circle(1.7mm); 
\draw[ fill=lightgray](1,0)circle(1mm);
\node at (1,-0.4) {\small $(\delta N^{2/3},0)$};

\fill[color=white] (3,2)circle(1.7mm); 
\draw[ fill=lightgray](3,2)circle(1mm);
\node at (3.9,2) {\small $v_N-(\frac12,\frac12)$};

\draw[ fill=black](-0.1,-0.1)circle(1mm);
\node at (-0.4,-0.4) {\small $(-\frac12,-\frac12)$};

\end{tikzpicture}
 
\end{center}
\caption{\small The sampled polymers starting $(\floor{\delta N^{2/3}},0)$ and  $(0,\floor{\delta N^{2/3}})$   (gray dotted lines)  coalesce outside $\lzb0,v_N\rzb$.  Equivalently, some dual point ${x^*}=x-(1/2,1/2)$ outside of $\lzb0,v_N\rzb - (1/2,1/2)$ sends a  dual polymer $\wc X^{\rho,x}_\bullet-(1/2,1/2)$ (black dotted line)  into the rectangle $\lzb (0,0), (\floor{\delta N^{2/3}},\,\floor{\delta N^{2/3}})\rzb $.}
\label{fig12}
\end{figure}

\begin{proof}[Proof of Theorem \ref{d_coal}]
As shown in Figure \ref{fig12}, the duality mentioned in Remark \ref{duality} implies that the sampled polymer paths coalesce outside of the rectangle $[\![0, v_N]\!]$ if and only if there exists some $x$ on  the northeast boundary of $[\![0, v_N]\!]$ 
such that the polymer $\wt X^{\rho,x}_\bbullet$ satisfies $|\tau_{0, x}| \leq \delta N^{2/3}$. 

By this equivalence, the expectation in  Theorem \ref{d_coal} is equal to the expectation in Theorem \ref{joint_d},
\begin{align*}
&\mathbb{E} \Big[ \ncouple^\rho_{\floor{\delta N^{2/3}}e_1,\floor{\delta N^{2/3}}e_2}\Big(\Gamma^{\mathbb{Z}^2 \setminus [\![0, v_N]\!]}\Big) \Big]
 = \mathbb{E}\Big[\wt{Q}^\rho_{0,\partial^{\textup{NE}}[\![0, v_N]\!]}\Big( \bigcup_{x\in\partial^{\textup{NE}}[\![0, v_N]\!]}\{ |\wt \tau_{0,x}| \leq \delta N^{2/3}\}\Big)\Big].
\end{align*}
Finally, for the exit time expectation on the right-hand side, the upper bound follows from Theorem \ref{joint_d}. The lower bound follows from \eqref{out_vs_bd} and \eqref{dupper:claim2} in Theorem \ref{dupper} since the probability of a union of events is bounded below by the maximum of the probabilities of the individual events.
\end{proof}


\begin{figure}[t]
\captionsetup{width=0.8\textwidth}
\begin{center}
\begin{tikzpicture}[>=latex, scale=1.3]
\draw[fill=lightgray ] (1,0) -- (1,1) --(0,1) -- (0,0)-- (1,0);

\draw[gray, line width=0.3mm, ->] (0,0) -- (5,0);
\draw[gray, line width=0.3mm, ->] (0,0) -- (0,4);

\draw[gray, dotted, line width=1mm, ->] (1,0) -- (1.5,0) -- (1.5,1) -- (2,1)-- (2,2) -- (3,2) -- (3,4);
\draw[gray, dotted, line width=1mm] (0,1) -- (0.4,1) -- (0.4,2)--  (1,2) -- (2,2);

\draw[black, dotted, line width=1mm] (2.8,3.2) -- (2.8,2.5) -- (1, 2.5) -- (1,2.3) -- (0.2, 2.3) -- (0.2,1.2) -- (0, 1.2);

\draw[black, dotted, line width=1mm] (3.2,3.2) -- (3.2,0.8)  -- (1.9,0.8) -- (1.9,0);

\draw[black, dotted, line width=1mm,] (1.9,-0.1) -- (-0.1,0-0.1)  -- (-0.1,1.2)  ;

\fill[color=white] (2.8,3.2)circle(1.7mm); 
\draw[ fill=lightgray](2.8,3.2)circle(1mm);

\fill[color=white] (3.2,3.2)circle(1.7mm); 
\draw[ fill=lightgray](3.2,3.2)circle(1mm);

\draw[gray ,dotted, line width=0.3mm, ] (0.8*5,0) -- (0.8*5,0.8*4) --(0,0.8*4) ;

\fill[color=white] (0.8*5, 0.8*4)circle(1.7mm); 
\draw[ fill=lightgray](0.8*5, 0.8*4)circle(1mm);
\node at (0.8*6+0.04, 0.8*4) {\small{$v_N-(\tfrac12,\tfrac12)$}};

\draw[black, line width=1mm] (1, -0.2) -- (1, 0.1);
\draw[black, line width=1mm] (-0.2, 1) -- (0.1, 1);
\node at (1.2, -0.4) {$rN^{2/3}$};

\node at (-0.7, 1) {$rN^{2/3}$};
\node at (-0.8, -0.2) {\footnotesize{$(-\tfrac12,-\tfrac12)$}};

\fill[color=white] (2.8,3.2)circle(1.7mm); 
\draw[ fill=lightgray](2.8,3.2)circle(1mm);
\node at (2.7,3.5) {\small $x^*$};

\fill[color=white] (3.2,3.2)circle(1.7mm); 
\draw[ fill=lightgray](3.2,3.2)circle(1mm);
\node at (3.6,3.5) {\small $x^*+e_1$};

\draw[ fill=black](-0.1,-0.1)circle(1mm);

\end{tikzpicture}
\end{center}
\caption{\small None of the backward polymers (black dotted lines) will enter the gray square because they are shielded away from it by the coalescing forward polymers (gray dotted lines).  }
\label{fig20}
\end{figure}
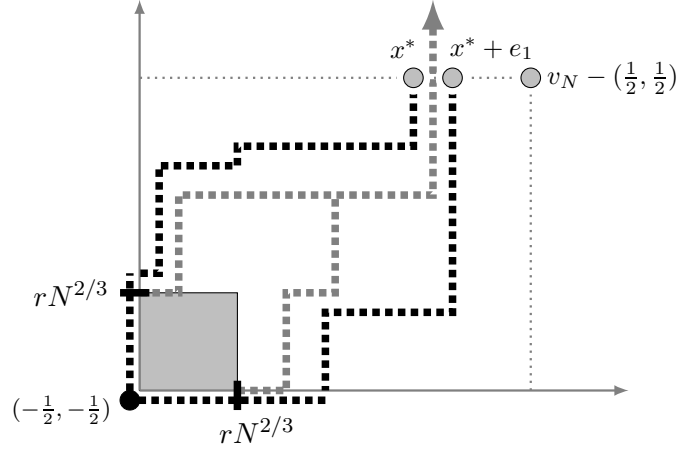

\begin{proof}[Proof of Theorems \ref{r_coal} and \ref{r_coal_tail}]
As shown in Figure \ref{fig20}, if the two sampled forward polymers starting at $(\floor{rN^{2/3}}, 0)$ and $(0,\floor{rN^{2/3}})$ coalesce inside $[\![0, v_N]\!]$, then by duality, this happens if and only for each $x\in\partial^{\textup{NE}}[\![0, v_N]\!]$ the polymer $\wt X^{\rho,x}_\bbullet$ satisfies $|\tau_{0,x}| \geq r N^{2/3}$.
Then, we have 
\begin{align}\label{r_dual_r}
& \ncouple^\rho_{\floor{r N^{2/3}}e_1,\floor{r N^{2/3}}e_2}\Big(\Gamma^{[\![0, v_N]\!]}\Big) 
 \stackrel{d}{=} \wt{Q}^\rho_{0,\partial^{\textup{NE}}[\![0, v_N]\!]}\Big( \bigcap_{x\in\partial^{\textup{NE}}[\![0, v_N]\!]}\{ |\wt \tau_{0,x}| \geq r N^{2/3}\}\Big).
\end{align}
The expectation and the tail probabilities of the right-hand side can be lower  bounded using Theorem \ref{joint_r}. And they are  upper bounded by Theorem \ref{r_up_low} since the probability of an intersection of events is bounded above by the minimum of the probabilities of the individual events.
\end{proof}


\begin{proof}[Proof of Corollary \ref{h_start}]
To prove the first inequality we will lower bound its complement. By duality, it suffices to show that for some small $q$ depending only on $\varepsilon$,
\begin{align}
\mathbb{E} \Big[ \ncouple^\rho_{0,\floor{r N^{2/3}}e_1}\Big(\Gamma^{\mathbb{Z}^2 \setminus [\![0, v_N]\!]}\Big) \Big]
&=\mathbb{E}\Big[\wt{Q}^\rho_{0,\partial^{\textup{NE}}[\![0, v_N]\!]}\Big( \bigcup_{x\in\partial^{\textup{NE}}[\![0, v_N]\!]}\{ 1\leq \wt \tau_{0,x} \leq r N^{2/3}\}\Big)\Big] \nonumber\\
&\geq \mathbb{E}\Big[{Q}^\rho_{0,v_N - q rN^{2/3}e_2}\{ 1\leq \tau \leq r N^{2/3}\}\Big] \nonumber\\
&\geq 1-e^{-Cr^3}. \label{lb_1r}
\end{align}
The last inequality \eqref{lb_1r} follows from an argument similar to the proof of Lemma \ref{lemmaab}. Here, instead of perturbing the directional parameter, we simply perturb our end point from $v_N$ to $v_N - qrN^{2/3}e_2$. Then, as shown in Figure \ref{fig_ab}, if we fix $q$ sufficiently small, then the $-\xi[\rho]$ directed ray starting at $v_N-qrN^{2/3}e_2$ will hit the $e_1$-axis within $[\![arN^{2/3}, brN^{2/3}  ]\!]$, for some $0<a<b<1$. This again just follows from Taylor's theorem and we omit the details. Then the rest of the argument is exactly the same as in Lemma \ref{lemmaab}. 

To prove the second inequality in the claim of the corollary we start with the following calculation, where the first equality comes from duality and the same calculation from \eqref{out_vs_bd} gives us the inequality when we switch from $``\max_{x\in\partial^{\textup{NE}}[\![0, v_N]\!]} \dots"$ to $``\max_{x\not\in[\![0, v_N]\!]}  \dots"$
\begin{align*}
 \mathbb{E} \Big[ \ncouple^\rho_{0,\floor{\delta N^{2/3}}e_1}\Big(\Gamma^{\mathbb{Z}^2\setminus [\![0, v_N]\!]}\Big) \Big] 
 & = \mathbb{E}\Big[\wt{Q}^\rho_{0,\partial^{\textup{NE}}[\![0, v_N]\!]}\Big( \bigcup_{x\in\partial^{\textup{NE}}[\![0, v_N]\!]}\{ 1 \leq \tau_{0,x}\leq \delta N^{2/3}\}\Big)\Big]\\
 & \geq \mathbb{E}\Big[ \max_{x\in\partial^{\textup{NE}}[\![0, v_N]\!]} {Q}^\rho_{0,x}\{1 \leq  \tau_{0,x}\leq \delta N^{2/3}\}\Big]\\
 & \geq \mathbb{E}\Big[ \max_{x\not\in[\![0, v_N]\!]} {Q}^\rho_{0,x}\{1 \leq  \tau_{0,x}\leq \delta N^{2/3}\}\Big].
\end{align*}
The last expectation can be lower bounded by $C\delta$. The proof is very similar to that of the lower bound in \eqref{dupper:claim2}. More precisely, by \eqref{lb_1r}, we can fix two constants $r_0$ and  $N_0$ (depending on $\varepsilon$)  such that, for $N\geq N_0$,  
\beq \label{exit47'} \mathbb{E} \bigl[ Q^{{\rho}}_{0, v_N-qr_0N^{2/3}e_2 +e_1}\{1\leq \tau  \leq r_0 N^{2/3}\}\bigr]\geq \frac{1}{2}.\eeq
Note that using the endpoint $v_N-qr_0N^{2/3}e_2+e_1$ instead of $v_N-qr_0N^{2/3}e_2$ does not change the proof of this lower bound. 

Now, \eqref{exit47'}
replaces the input \eqref{exit47}, and we form our partition $\{p_i\}$ in the range $[1, r_0]$ instead of $[-r_0, r_0]$. Then, the rest of the proof is the same as the lower bound proof in \eqref{dupper:claim2}.
\end{proof}

\begin{figure}[t]
\captionsetup{width=0.8\textwidth}
\begin{center}
\begin{tikzpicture}[scale = 0.8]

--------------------------------------------------------------

\draw[gray ,dotted, line width=0.3mm, <-](1,-1)--(4,3);

\draw[gray, line width=0.3mm, ->] (0,0) -- (6,0);
\draw[gray, line width=0.3mm, ->] (0,0) -- (0,4);

\draw[ fill=white](1.75, 0)circle(1.3mm);

\draw[gray ,dotted, line width=0.3mm, <-](0,0)--(4,16/3);

\draw[ fill=lightgray](0,0)circle(1mm);
\node at (-0.9,-0) {$(0,0)$};

\fill[color=white] (4,16/3)circle(1.7mm); 
\draw[ fill=lightgray](4,16/3)circle(1mm);
\node at (4.5,16/3) {$v_N$};

\fill[color=white] (4,3)circle(1.7mm); 
\draw[ fill=lightgray](4,3)circle(1mm);
\node at (5.8,3) {$v_N-qrN^{2/3}e_2$};

\draw[ line width=1mm] (2.5,-0.1) -- (2.5,0.1);
\node at (2.5, -0.4) {${brN^{2/3}}$};

\draw[ line width=1mm] (1,-0.1) -- (1,0.1);
\node at (0.5, -0.4) {${arN^{2/3}}$};

\node at (3.5, 1.5) {$\xi[\rho]$};
\end{tikzpicture}
\end{center}
\caption{An illustration for the inequality \eqref{lb_1r}. Starting from the point  $v_N-qrN^{2/3}e_2$, the $-\xi[\rho]$-directed ray will hit the $e_1$-axis between $[\![arN^{2/3}, brN^{2/3}  ]\!]$ for some $0<a<b<1$, provided that $q$ is fixed sufficiently small.}
\label{fig_ab}
\end{figure}
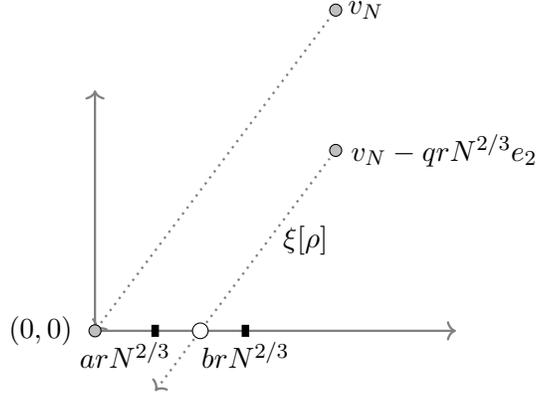

\section{Total variation distance bounds.}\label{TV_dist}


\begin{proof}[Proof of Theorem \ref{tv_d_upper}]
The claim follows from the fact that if $U$ and $V$ are two random variables with distributions $\mu$ and $\nu$, respectively, and if $\mathbf P$ is any coupling of the two random variables, then
\be\label{coupineq} d_{\textup{TV}}(\mu, \nu) \leq \mathbf{P}(U\neq V).\ee 

Consider the paths $X^{\rho,\delta N^{2/3}e_i}_\bbullet$, $i\in\{1,2\}$, defined in Section \ref{intro_dual}. 
Then, $\chiup_N(X^{\rho,\delta N^{2/3}e_1}_\bbullet)\ne \chiup_N(X^{\rho,\delta N^{2/3}e_2}_\bbullet)$ implies the two paths did not coalesce inside $[\![0, v_N]\!]$. Hence, if $\mathbf P$ is the probability measure from Section \ref{intro_dual}, then
$$\mathbf{P}\bigl\{\chiup_N(X^{\rho,\delta N^{2/3}e_1}_\bbullet)\ne \chiup_N(X^{\rho,\delta N^{2/3}e_2}_\bbullet)\bigr\} \leq \ncouple^\rho_{\floor{\delta N^{2/3}}e_1,\floor{\delta N^{2/3}}e_2}\Big(\Gamma^{\mathbb{Z}^2 \setminus [\![0, v_N]\!]}\Big).$$
Now the upper bound claimed in the theorem follows directly from Theorem \ref{d_coal}.
\end{proof}

\begin{proof}[Proof of Theorem \ref{tv_r_upper}]

We will first look at $u$ only in the north boundary of $[\![0, v_N]\!]$, which we denote as $\partial^{\textup{N}}[\![0, v_N]\!]$, and we will show that 
$$\sum_{u\in \partial^{\textup{N}}[\![0, v_N]\!]}
\abs{\Pi_{\floor{r N^{2/3}}e_1}^{{\rho}}(\chiup_N=u)-\Pi_{\floor{r N^{2/3}}e_2}^{{\rho}} (\chiup_N=u)}\qquad
\text{ is close to $1$}.$$

A similar argument can be applied to the east boundary to show that sum is also close to $1$. And combining the two calculations for the north and east boundaries would finish the proof.

From Proposition \ref{stat_iid} and Theorem \ref{r_up_low},
\begin{align*}
&\mathbb{P}\Big(\Pi_{\floor{r N^{2/3}}e_2}^{{\rho}}(\chiup_N\in \partial^N [\![0, v_N]\!]) \geq 1-e^{-cr^{2}N^{1/3}}\Big) \geq 1- e^{-Cr^3},\\
&\mathbb{P}\Big(\Pi_{\floor{r N^{2/3}}e_1}^{{\rho}}(\chiup_N\in \partial^N [\![0, v_N]\!]) \leq e^{-cr^{2}N^{1/3}}\Big) \geq 1-e^{-Cr^3}.
\end{align*}
To finish the proof, on the intersection of the two events above, we have
\begin{align*}
&\sum_{u\in \partial^{\textup{N}}[\![0, v_N]\!]}\abs{\Pi_{\floor{r N^{2/3}}e_1}^{{\rho}}(\chiup_N=u)-\Pi_{\floor{r N^{2/3}}e_2}^{{\rho}} (\chiup_N=u)}\\
&\qquad\qquad \geq \sum_{u\in \partial^{\textup{N}}[\![0, v_N]\!]}\Bigl(\Pi_{\floor{r N^{2/3}}e_1}^{{\rho}}(\chiup_N=u)-\Pi_{\floor{r N^{2/3}}e_2}^{{\rho}} (\chiup_N=u)\Bigr)\\
&\qquad\qquad = \Pi_{\floor{r N^{2/3}}e_2}^{{\rho}}(\chiup_N \in \partial^N [\![0, v_N]\!]) - \Pi_{\floor{r N^{2/3}}e_1}^{{\rho}}(\chiup_N \in \partial^N [\![0, v_N]\!])\\
&\qquad\qquad\geq 1-2e^{-cr^{2}N^{1/3}}.\qedhere
\end{align*}
\end{proof}

\section{Transversal fluctuation lower bound}\label{sec:trans}

In this section, we prove  Theorem \ref{fluc_lb}, but omit some of the details since the whole proof is similar to the proof of the upper bound in Theorem \ref{dupper}. 

First, for $i\in\{1,2\}$, let us define $\{\textsf{mid}_i \leq \delta N^{2/3}\}$ to be  the collection of paths between $-v_N$ and $v_N$ which crosses the segment between $-\delta N^{2/3} e_i$ and $\delta N^{2/3} e_i$. Since $$\{\textsf{mid} \leq \delta N^{2/3}\} \subset \{\textsf{mid}_1 \leq \delta N^{2/3}\}\cup \{\textsf{mid}_2 \leq \delta N^{2/3}\},$$
by a union bound and the symmetry between $i = 1$ and $2$ it suffices to prove that 
$$\mathbb{E}\Big[Q_{-v_N, v_N}\{\textsf{mid}_1\leq \delta N^{2/3}\} \Big] \leq C|\log \delta|^{10} \delta.$$
We prove this by showing that 
\begin{equation}\label{tr_show}
\mathbb{P}\Big(Q_{-v_N, v_N}\{\textsf{mid}_1\leq \delta N^{2/3}\}  \geq e^{-|\log \delta|^2 \sqrt{\delta}N^{1/3}}\Big) \leq C|\log \delta|^{10} \delta.
\end{equation}
Let $r = |\log \delta|$ and fix $\alpha$ sufficiently small (now depending only on $\mu$) as in the proof of Theorem \ref{dupper}. The next calculation follows the same steps as \eqref{mainest2}, except that we now set $\rho = \mu/2$ and consider the dark region $\mathcal{D}$ as a single point $v_N$. 
\begin{align}
&\text{left side of \eqref{tr_show}}\nonumber \\
& =  \mathbb{P}\Big(\log Z_{-v_N, v_N}- \log Z_{-v_N, v_N}\{\textsf{mid}_1\leq \delta N^{2/3}\}  \leq |\log \delta|^2 \sqrt{\delta}N^{1/3}\Big)\nonumber\\
& \leq  \mathbb{P}\Big(\log Z_{-v_N, v_N}\{\textsf{mid}_1\leq r N^{2/3}\} - \log Z_{-v_N, v_N}\{\textsf{mid}_1\leq \delta N^{2/3}\}  \leq |\log \delta|^2 \sqrt{\delta}N^{1/3}\Big)\nonumber\\
& \leq  \mathbb{P}\Big(\max_{|k| \leq \floor{rN^{2/3}}}\Big[\log Z_{-v_N, ke_1} + \log Z_{(k,1), v_N} \Big] \nonumber\\
& \qquad \qquad \qquad  - \max_{|j| \leq \floor{\delta N^{2/3}}} \Big[\log Z_{-v_N, ke_1} + \log Z_{(k,1), v_N} \Big]  \leq 2|\log \delta|^2 \sqrt{\delta}N^{1/3}\Big)\nonumber\\
& =  \mathbb{P}\Big(\max_{|k| \leq \floor{rN^{2/3}}}\Big[\log \frac{Z_{-v_N, ke_1}}{Z_{-v_N, (0,0)}}  + \log \frac{Z_{(k,1), v_N}}{ Z_{e_2, v_N}} \Big] \nonumber\\
& \qquad \qquad   - \max_{1\leq j \leq \floor{\delta N^{2/3}}} \Big[\log \frac{Z_{-v_N, je_1}}{Z_{-v_N, (0,0)}}  + \log \frac{Z_{(j,1), v_N}}{ Z_{e_2, v_N}} \Big] \leq 2|\log \delta|^2 \sqrt{\delta}N^{1/3}\Big)\nonumber\\
& \leq  \mathbb{P}\Big(\max_{|k| \leq \floor{rN^{2/3}}}\Big[\log \frac{Z_{-v_N, ke_1}}{Z_{-v_N, (0,0)}}  + \log \frac{Z_{(k,1), v_N}}{ Z_{e_2, v_N}} \Big] \leq 3|\log \delta|^2 \sqrt{\delta}N^{1/3}\Big)\label{trest1}\\
& \qquad    + \mathbb{P}\Big(\max_{1\leq j \leq \floor{\delta N^{2/3}}} \Big[\log \frac{Z_{-v_N, je_1}}{Z_{-v_N, (0,0)}}  + \log \frac{Z_{(j,1), v_N}}{ Z_{e_2, v_N}} \Big]  \geq |\log \delta|^2 \sqrt{\delta}N^{1/3}\Big).\label{trest2}
\end{align}
Next, let us define 
$$
\widetilde{I}^{v_N}_{(i,1)} = \frac{Z_{(i-1,1), v_N}}{ Z_{(i,1), v_N} }, \qquad \widetilde{I}^{-v_N}_{(i,0)} = \frac{Z_{-v_N, (i,1)}}{Z_{-v_N, (i-1,1)}},
$$
and a two-sided multiplicative walk $\{M_n'\}_{n\in[\![-\floor{\alpha rN^{2/3}}+1, \floor{\alpha rN^{2/3}}]\!]}$ by setting  $M_{0}' = 1$ and 
$$
M_n'/M_{n-1}' = \widetilde{I}^{-v_N}_{(n,0)}/\widetilde{I}{}^{v_N}_{(n, 1)}.
$$
Then, the two probabilities can be rewritten as 
\begin{equation}\label{same_blw}
\begin{aligned}
\eqref{trest1} + \eqref{trest2} &= \mathbb{P} \Big(  \max_{n \in [\![-\alpha rN^{2/3}+1, \alpha rN^{2/3} ]\!]} \log M_n' \leq 3|\log \delta|^2 \sqrt{\delta} N^{1/3}\Big)\\
& \qquad \qquad + \mathbb{P} \Big( \max_{n \in [\![1, \delta N^{2/3} ]\!]} \log M_n'\geq |\log \delta|^2 \sqrt{\delta} N^{1/3}\Big).
\end{aligned}
\end{equation}
Note how the right-hand side is similar to \eqref{asdf1} + \eqref{asdf2}, except for having $M'_n$ instead of $M_n$, and the region $\mathcal{D}$ is reduced to the single vertex $v_N$. 
Next, we give a sketch of how to carry over the estimate from the proof of Theorem \ref{dupper} to the random walk in this proof. The essential step is to upper and lower bound the walk $M_n'$ by two other walks with i.i.d. steps. This was done for $M_n$ previously in  \eqref{the_bounds}. After that, the bound on the two probabilities above comes from the same estimates as in the proof of Theorem \ref{dupper}. 

First, let us summarize how the desired random walk bound was obtained in the proof of Theorem \ref{dupper}.  Recall $\lambda$ and $\mu$, defined in \eqref{la8}. Lemma \ref{lm:A78} showed that with probability at least $1-e^{-Cr^3}$, for each  $a, b\in  [\![-\floor{\alpha rN^{2/3}}+1, \floor{\alpha rN^{2/3}}]\!]$,
$$\tfrac{1}{2} \prod_{i=a}^b I^{\eta, \textup{NE}}_{(i,1)}\leq \prod_{i=a}^b\widetilde{I}^{v_N}_{(i,1)}  \leq 2\prod_{i=a}^b I^{\lambda, \textup{NE}}_{(i, 1)},$$
where $I^{\,\abullet, \textup{NE}}_{(i, 1)} \sim \textup{Ga}^{-1}(\abullet)$.
Furthermore, as stated below \eqref{boundwithind1}, there is a coupling such that the random variables
\begin{equation}\label{abv_ind}
\big\{ I^{\eta, \textup{NE}}_{(i,1)},I^{\lambda, \textup{NE}}_{(j,1)}:i \leq 0,j\ge1\big\}\text{ are independent.}
\end{equation}

By symmetry (or rotating the picture $180^\circ$), the exact same argument can be applied to $\widetilde{I}^{-v_N}_{(i,0)} $, where now these edge weights are calculated to the point $-v_N-(e_1+e_2)$ instead of to $v_N+(e_1+e_2)$. We get that with probability at least $1-e^{-Cr^3}$, for each  $a, b\in  [\![-\floor{\alpha rN^{2/3}}+1, \floor{\alpha rN^{2/3}}]\!]$,
$$\tfrac{1}{2} \prod_{i=a}^b I^{\eta, \textup{SW}}_{(i,0)}\leq \prod_{i=a}^b\widetilde{I}^{-v_N}_{(i,0)}  \leq 2\prod_{i=a}^b I^{\lambda, \textup{SW}}_{(i, 0)},$$
where $I^{\,\abullet, \textup{SW}}_{(i, 0)} \sim \textup{Ga}^{-1}(\abullet)$ are edge weights that are calculated to $-v_{N}-(e_1+e_2)$ and with a boundary placed on the south-west edges of the quadrant $-v_N-(e_1+e_2)+\Z^2_{\ge0}$.
As above, the random variables
\begin{equation}\label{blw_ind}
\big\{ I^{\lambda, \textup{SW}}_{(i,0)},I^{\eta, \textup{SW}}_{(j,0)}:i \leq 0,j\ge0\big\}\text{ are independent.}
\end{equation}
Note how the parameters switched sides, as compared to \eqref{abv_ind}.

Next, define two two-sided multiplicative random walks $M^{+}_n$, $M^{-}_n$  with  $M^{\pm}_{0}= 1$ and 
\begin{align*}
M_n^{+}/M_{n-1}^{ +} &= {I}{}^{\lambda, \textup{SW}}_{(n,0)}/{I}{}^{\eta, \textup{NE}}_{(n,1)}\\
M_n^{ -}/M_{n-1}^{-} &= {I}{}^{\eta, \textup{SW}}_{(n,0)}/{I}{}^{\lambda, \textup{NE}}_{(n,1)}
\end{align*}
We get
\begin{align*}
\frac{1}{2}M_n^{-} \leq M_n'  \leq 2M_n^{+}  \text{ for $n \geq 1$} 
\quad\text{and}\quad 
\frac{1}{2}M_n^{+} \leq M_n' \leq 2 M_n^{-} \text{ for $n \leq 0$} .
\end{align*}
These bounds play the role of \eqref{the_bounds}. With this, go back to \eqref{same_blw} and follow the same argument as the one we used to bound \eqref{asdf1} + \eqref{asdf2}, but with  $M_n$, $M^\lambda_n$, and $M^\mu_n$  replaced by $M_n'$, $M^-_n$, and $M^+_n$, respectively. We should point out that an essential fact that is used in the step analogous to \eqref{ind_prod} is the independence of the walks $\{M^-_n: n \geq 1\}$ and $\{M^+_n: n \leq 0\}$, which follows from \eqref{abv_ind} and \eqref{blw_ind}. We omit the rest of the details.

\appendix
\section{Appendix}

\subsection{Moderate deviation of the bulk free energy}

We present here two estimates that we use in the proof of \eqref{improved}. The first tail bound can be derived for the inverse-gamma polymer by combining Theorem 1.7 of \cite{Bar-COr-Dim-21}, which utilizes integrable probability methods, with Theorem 2.2 of \cite{Geo-Sep-13}. For the O'Connell-Yor polymer, the bound was established in \cite{Lan-Sos-22-b-} as Proposition 2.1 without the use of integrable probability. A proof of the bound for the inverse-gamma polymer, without the use of integrable probability, will appear in \cite{Emr-Jan-Xie-23-}. This result can be found in Theorem 4.3.1 of the Ph.D.\ thesis \cite{Xie-22}.

\begin{proposition}\label{up_ub}
Fix $\varepsilon \in (0, \mu/2)$. There exist positive constants $C, N_0$ depending on $\varepsilon$ such that for each $N\geq N_0$, $t \geq 1$, and each $\rho \in [\varepsilon, \mu-\varepsilon]$, we have
$$\mathbb{P}\bigl(\log Z_{0,  v_N} - \Lambda(v_N) \geq tN^{1/3}\bigr) \leq e^{-C \min\{t^{3/2}\!,\, tN^{1/3}\}}.$$
\end{proposition}

The next tail bound is Proposition 3.8 in \cite{Bas-Sep-She-23-}.  The analogous bound for the O'Connell-Yor polymer appears as Proposition 3.4 in \cite{Lan-Sos-22-b-}.

\begin{proposition}\label{low_ub}
Let $\varepsilon \in (0, \mu/2)$. There exist positive constants $C, N_0 $ depending on $\varepsilon$ such that for each $N\geq N_0$, $t\geq 1$ and and each $\rho\in[\varepsilon, \mu-\varepsilon]$, we have 
$$\mathbb{P}\bigl(\log Z_{0, v_N} - \Lambda(v_N) \leq -tN^{1/3}\bigr) \leq e^{-C \min\{t^{3/2}, \, tN^{1/3}\}}.$$
\end{proposition}

\subsection{Proof of Propositions \ref{far_s} and \ref{close_s}}\label{exp_calc}

Let $\varepsilon \in (0, \mu/2)$ and fix $\rho \in [\varepsilon, \mu-\varepsilon]$.
We start with a few derivative calculations.
\begin{align}
&\frac{d}{dz} \frac{\Psi_1(\rho+z)}{\Psi_1(\rho+z) + \Psi_1(\mu - \rho-z)}\bigg|_{z = 0} = \frac{\Psi_2(\rho) \Psi_1(\mu-\rho) +\Psi_1(\rho) \Psi_2(\mu-\rho) }{(\Psi_1(\rho) + \Psi_1(\mu - \rho))^2},\label{long1}\\
&\frac{d}{dz} \frac{\Psi_1(\mu-\rho-z)}{\Psi_1(\rho+z) + \Psi_1(\mu - \rho-z)}\bigg|_{z = 0} = -\frac{\Psi_2(\rho) \Psi_1(\mu-\rho) +\Psi_1(\rho) \Psi_2(\mu-\rho) }{(\Psi_1(\rho) + \Psi_1(\mu - \rho))^2},\nonumber\\
&\frac{d^2}{dz^2} \frac{\Psi_1(\rho+z)}{\Psi_1(\rho+z) + \Psi_1(\mu - \rho-z)}\bigg|_{z = 0} = -\frac{2\Psi_2(\rho)(\Psi_2(\rho)- \Psi_2(\mu-\rho))}{{(\Psi_1(\rho) + \Psi_1(\mu - \rho))^2}} + \frac{\Psi_3(\rho)}{\Psi_1(\rho) + \Psi_1(\mu - \rho)} \nonumber\\
& \qquad \qquad \qquad + \Psi_1(\rho)\Big(\frac{2(\Psi_2(\rho) - \Psi_2(\mu-\rho))^2}{(\Psi_1(\rho) + \Psi_1(\mu - \rho))^3} - \frac{\Psi_3(\mu-\rho) + \Psi_3(\rho)}{(\Psi_1(\rho) + \Psi_1(\mu - \rho))^2} \Big),\nonumber\\
&\frac{d^2}{dz^2} \frac{\Psi_1(\mu-\rho-z)}{\Psi_1(\rho+z) + \Psi_1(\mu - \rho-z)}\bigg|_{z = 0} = \frac{2\Psi_2(\mu-\rho)(\Psi_2(\rho)- \Psi_2(\mu-\rho))}{{(\Psi_1(\rho) + \Psi_1(\mu - \rho))^2}} + \frac{\Psi_3(\mu-\rho)}{\Psi_1(\rho) + \Psi_1(\mu - \rho)} \nonumber\\
& \qquad \qquad \qquad+ \Psi_1(\mu-\rho)\Big(\frac{2(\Psi_2(\rho) - \Psi_2(\mu-\rho))^2}{(\Psi_1(\rho) + \Psi_1(\mu - \rho))^3} - \frac{\Psi_3(\mu-\rho) + \Psi_3(\rho)}{(\Psi_1(\rho) + \Psi_1(\mu - \rho))^2} \Big),\nonumber\\
&\frac{d}{dz}\Big(\frac{\Psi_1(\rho+z)}{\Psi_1(\rho+z) + \Psi_1(\mu - \rho-z)}\Psi_0(\mu-\rho-z) +\frac{\Psi_1(\mu-\rho-z)}{\Psi_1(\rho+z) + \Psi_1(\mu - \rho-z)} \Psi_0(\rho+z) \Big)\bigg|_{z = 0}\nonumber\\
& = \frac{(\Psi_0(\mu-\rho) - \Psi_0(\rho))(\Psi_2(\rho) \Psi_1(\mu-\rho) +\Psi_1(\rho) \Psi_2(\mu-\rho)) }{(\Psi_1(\rho) + \Psi_1(\mu - \rho))^2},\\
&\frac{d^2}{dz^2}\Big(\frac{\Psi_1(\rho+z)}{\Psi_1(\rho+z) + \Psi_1(\mu - \rho-z)}\Psi_0(\mu-\rho-z) +\frac{\Psi_1(\mu-\rho-z)}{\Psi_1(\rho+z) + \Psi_1(\mu - \rho-z)} \Psi_0(\rho+z) \Big)\bigg|_{z = 0}\nonumber\\
& =  \frac{2(\Psi_0(\rho) \Psi_2(\mu-\rho)  - \Psi_2(\rho) \Psi_0(\mu-\rho)) (\Psi_2(\rho) - \Psi_2(\mu-\rho))}{(\Psi_1(\rho) + \Psi_1(\mu - \rho))^2} \nonumber\\
&\quad + \frac{\Psi_3(\rho) \Psi_0(\mu-\rho) + \Psi_0(\rho) \Psi_3(\mu-\rho) - \Psi_2(\rho) \Psi_1(\mu-\rho) - \Psi_1(\rho)\Psi_2(\mu-\rho)}{\Psi_1(\rho) + \Psi_1(\mu - \rho)} \nonumber\\
& \quad + (\Psi_1(\rho) \Psi_0(\mu-\rho) + \Psi_0(\rho) \Psi_1( \mu-\rho))\Big(\frac{2(\Psi_2(\rho) - \Psi_2(\mu-\rho))^2}{(\Psi_1(\rho) + \Psi_1(\mu - \rho))^3} - \frac{\Psi_3(\mu-\rho) + \Psi_3(\rho)}{(\Psi_1(\rho) + \Psi_1(\mu - \rho))^2} \Big).\nonumber
\end{align}

Because of the bijection in \eqref{char_dir}, there exists a $z$ such that 
\begin{align}\label{z-choice}
N\xi[\rho + z] =  v_N - \floor{sN^{2/3}}e_1 + \floor{sN^{2/3}} e_2.
\end{align}
From \eqref{long1} we see that the derivative of $\frac{\Psi_1(\rho+z)}{\Psi_1(\rho+z)+\Psi_1(\mu-\rho-z)}$ at $z=0$ is strictly negative. By continuity, it is also strictly negative on a neighborhood of $0$. This and the mean value theorem imply that 
\begin{align}\label{z_range}
z\in [c_1 sN^{-1/3}, c_2 sN^{-1/3}]
\end{align}
for some positive constant $c_1, c_2$ depending on $\varepsilon$.

The quantity appearing on the left side of  Proposition \ref{far_s} and Proposition \ref{close_s} is essentially the following (we ignore the  integer floor function),
\begin{align*}
&-N\Big[\frac{\Psi_1(\rho +z)}{\Psi_1(\rho +z) + \Psi_1(\mu - \rho -z)} \Psi_0(\mu-\rho-z) + \frac{\Psi_1(\mu- \rho -z)}{\Psi_1(\rho +z) + \Psi_1(\mu - \rho -z)} \Psi_0(\rho+z)\Big] \\
& + N\Big[\frac{\Psi_1(\rho)}{\Psi_1(\rho ) + \Psi_1(\mu - \rho )} \Psi_0(\mu-\rho) + \frac{\Psi_1(\mu- \rho )}{\Psi_1(\rho ) + \Psi_1(\mu - \rho )} \Psi_0(\rho)\Big]\\
& + N\Psi_0(\mu-\rho)\Big[(\frac{\Psi_1(\rho +z)}{\Psi_1(\rho +z) + \Psi_1(\mu - \rho -z)} - \frac{\Psi_1(\rho )}{\Psi_1(\rho) + \Psi_1(\mu - \rho)}\Big] \\
& + N\Psi_0(\rho)\Big[\frac{\Psi_1(\mu-\rho -z)}{\Psi_1(\rho +z) + \Psi_1(\mu - \rho -z)} - \frac{\Psi_1(\mu-\rho )}{\Psi_1(\rho) + \Psi_1(\mu - \rho)}\Big].
\end{align*}
In the above, we used \eqref{z-choice} to write $\floor{sN^{2/3}}=(v_N-N\xi[\rho+z])\cdot e_1=(N\xi[\rho+z]-v_N)\cdot e_2$.

By performing Taylor expansions in $z$ and using the computations presented earlier in this section, we observe a number of cancellations, ultimately turning the above expression into
$$\frac{N}{2} \cdot  \frac{\Psi_1(\rho) \Psi_2(\mu-\rho) + \Psi_2(\rho) \Psi_1(\mu-\rho)}{\Psi_1(\rho) + \Psi_1(\mu-\rho)}z^2 + N\cdot \mathcal{O}(z^3).$$
This and \eqref{z_range} imply the claimed bounds in Propositions \ref{far_s} and \ref{close_s}, provided that a sufficiently small value of $c_0$ is chosen.\hfill\qed

\subsection{Non-random properties}  \label{nonrandom}

The following monotonicity property of the ratios of partition functions is in  \cite[Lemma A.2]{Bus-Sep-22-ejp}.

\begin{lemma} \label{mono_ratio}
Let $x, y, z\in \mathbb{Z}^2$ be such that $x\cdot e_1 \leq y \cdot e_1$, $x\cdot e_2 \geq y \cdot e_2$, and $x, y \leq z$, then \begin{equation}
\frac{Z_{x, z}}{Z_{x, z-e_1}} \leq \frac{Z_{y,z}}{Z_{y, z-e_1}} \qquad \text{ and } \qquad \frac{Z_{x, z}}{Z_{x, z-e_2}} \geq \frac{Z_{y,z}}{Z_{y, z-e_2}}.
\end{equation}
\end{lemma}
The above lemma implies the following results about the monotonicity between the ratio of partition functions and exit times. 
\begin{lemma}\label{2_exit_ineq}
Let $z\in \mathbb{Z}^2_{\geq 0}$ and let $k,l\in \mathbb{Z}_{\geq 0}$ be such that $l \leq k$. Then
$$\frac{Z_{0, z}(\tau \geq l)}{Z_{0, z-e_1}(\tau\geq l)} \leq \frac{Z_{0, z}(\tau\geq k)}{Z_{0, z-e_1}(\tau\geq k)} \qquad \text{ and } \qquad \frac{Z_{0, z}(\tau\geq l)}{Z_{0, z-e_2}(\tau\geq l)} \geq \frac{Z_{0, z}(\tau\geq k)}{Z_{0, z-e_2}(\tau\geq k)}\,.$$
\end{lemma}
\begin{proof}
Note that
$\frac{Z_{0, z}(\tau\geq l)}{Z_{0, z-e_1}(\tau\geq l)}  = \frac{Z_{le_1, z}}{Z_{le_1, z-e_1}}$ and $\frac{Z_{0, z}(\tau\geq k)}{Z_{0, z-e_1}(\tau\geq k)} = \frac{Z_{ke_1, z}}{Z_{ke_1, z-e_1}}$. Then Lemma \ref{mono_ratio} gives us the inequality 
$$ \frac{Z_{le_1, z}}{Z_{le_1, z-e_1}}  \leq  \frac{Z_{ke_1, z}}{Z_{ke_1, z-e_1}}.$$
The other inequality with $e_2$ follows from a similar argument.
\end{proof}

The next lemma is an immediate consequence of Lemma \ref{2_exit_ineq}. It suggests that shifting the endpoint to the right or down increases the likelihood of the polymer taking more $e_1$ steps at the beginning.
\begin{lemma} \label{polymono}
For any $k, l, m\in \mathbb{Z}_{\geq 0}$ and $x\in \mathbb{Z}^2_{\geq 0}$ such that $x+le_1-me_2\in\mathbb{Z}^2_{\ge0}$,
$$Q_{0, x} \{\tau \geq k \} \leq Q_{0, x+le_1-me_2} \{\tau \geq k \}.$$
\end{lemma}

\begin{proof}
Note that the proof of Lemma \ref{2_exit_ineq} also gives 
$$\frac{Z_{0, x}}{Z_{0, x-e_1}} \leq \frac{Z_{0, z}(\tau\geq k)}{Z_{0, x-e_1}(\tau\geq k)} \qquad \text{ and } \qquad \frac{Z_{0, x}}{Z_{0, x-e_2}} \geq \frac{Z_{0, x}(\tau\geq k)}{Z_{0, x-e_2}(\tau_{0,x}\geq k)}.$$
Rearrange to get
\begin{equation}\label{ineq_Q_1}Q_{0, x} \{\tau \geq k \}  = \frac{Z_{0,x}(\tau \geq k)}{Z_{0,x}}\leq \frac{Z_{0,x+e_1}(\tau \geq k)}{Z_{0,x+e_1}} = Q_{0, x+e_1}  \{\tau \geq k\}
\end{equation}
and 
\begin{equation}\label{ineq_Q_2}Q_{0, x} \{\tau \geq k \}  \geq Q_{0, x+e_2}  \{\tau \geq k\}.\end{equation}
Applying the two inequalities \eqref{ineq_Q_1} and \eqref{ineq_Q_2} repeatedly gives us the statement of our lemma.
\end{proof}

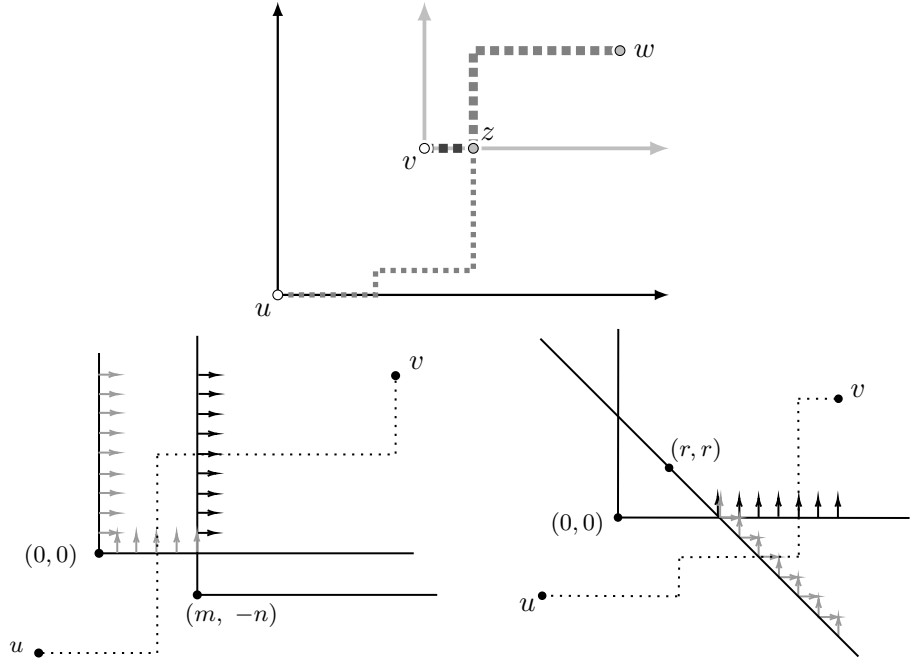
\begin{figure}[t]
\captionsetup{width=0.8\textwidth}
\begin{center}
\begin{tikzpicture}[>=latex, scale=0.65]

\draw[line width=0.3mm, ->] (0,0) -- (8,0);
\draw[line width=0.3mm, ->] (0,0) -- (0,6);
\draw[gray ,dotted, line width=0.7mm] (0,0) -- (2,0) -- (2,0.5) -- (4, 0.5) -- (4,3) ;

\draw[lightgray, line width=0.5mm, ->] (3,3) -- (8,3);

(3,2.8) -- (7.8,2.8) ;

\draw[lightgray, line width=0.5mm, ->] (3,3) -- (3,6);

\draw[darkgray, dotted, line width=1.2mm] (3,3) -- (4,3); 
\draw[gray, dotted, line width=1.2mm]  (4,3) -- (4,5) -- (7,5);
\fill[color=white] (3,3)circle(1.7mm);
\draw[ fill=white](3,3)circle(1mm);
\node at (2.7,2.7) {$v$};

\draw[ fill=white](0,0)circle(1mm);
\node at (-0.3,-0.3) {$u$};

\fill[color=white] (4,3)circle(1.7mm); 
\draw[ fill=lightgray](4,3)circle(1mm);
\node at (4.3,3.3) {$z$};

\fill[color=white] (7,5)circle(1.7mm);
\draw[ fill=lightgray](7,5)circle(1mm);
\node at (7.5,5) {$w$};

\end{tikzpicture}

\tikzset{every picture/.style={line width=0.75pt}} 

\begin{tikzpicture}[x=0.75pt,y=0.75pt,yscale=-1,xscale=1]

\draw   (409.53,160.3) -- (250.8,160.3) -- (250.8,59.2) ;
\draw [color={rgb, 255:red, 155; green, 155; blue, 155 }  ,draw opacity=1 ]   (260,160.17) -- (260,152.57) ;
\draw [shift={(260,150.57)}, rotate = 90] [color={rgb, 255:red, 155; green, 155; blue, 155 }  ,draw opacity=1 ][line width=0.75]    (4.37,-1.32) .. controls (2.78,-0.56) and (1.32,-0.12) .. (0,0) .. controls (1.32,0.12) and (2.78,0.56) .. (4.37,1.32)   ;
\draw [color={rgb, 255:red, 155; green, 155; blue, 155 }  ,draw opacity=1 ]   (279.74,160.17) -- (279.74,152.57) ;
\draw [shift={(279.74,150.57)}, rotate = 90] [color={rgb, 255:red, 155; green, 155; blue, 155 }  ,draw opacity=1 ][line width=0.75]    (4.37,-1.32) .. controls (2.78,-0.56) and (1.32,-0.12) .. (0,0) .. controls (1.32,0.12) and (2.78,0.56) .. (4.37,1.32)   ;
\draw [color={rgb, 255:red, 155; green, 155; blue, 155 }  ,draw opacity=1 ]   (289.94,159.97) -- (289.94,152.37) ;
\draw [shift={(289.94,150.37)}, rotate = 90] [color={rgb, 255:red, 155; green, 155; blue, 155 }  ,draw opacity=1 ][line width=0.75]    (4.37,-1.32) .. controls (2.78,-0.56) and (1.32,-0.12) .. (0,0) .. controls (1.32,0.12) and (2.78,0.56) .. (4.37,1.32)   ;
\draw [color={rgb, 255:red, 155; green, 155; blue, 155 }  ,draw opacity=1 ]   (269.6,160.04) -- (269.6,152.44) ;
\draw [shift={(269.6,150.44)}, rotate = 90] [color={rgb, 255:red, 155; green, 155; blue, 155 }  ,draw opacity=1 ][line width=0.75]    (4.37,-1.32) .. controls (2.78,-0.56) and (1.32,-0.12) .. (0,0) .. controls (1.32,0.12) and (2.78,0.56) .. (4.37,1.32)   ;
\draw [color={rgb, 255:red, 0; green, 0; blue, 0 }  ,draw opacity=1 ]   (300.9,70.24) -- (308.5,70.28) ;
\draw [shift={(310.5,70.29)}, rotate = 180.27] [color={rgb, 255:red, 0; green, 0; blue, 0 }  ,draw opacity=1 ][line width=0.75]    (4.37,-1.32) .. controls (2.78,-0.56) and (1.32,-0.12) .. (0,0) .. controls (1.32,0.12) and (2.78,0.56) .. (4.37,1.32)   ;
\draw [color={rgb, 255:red, 0; green, 0; blue, 0 }  ,draw opacity=1 ]   (300.81,89.98) -- (308.41,90.01) ;
\draw [shift={(310.41,90.02)}, rotate = 180.27] [color={rgb, 255:red, 0; green, 0; blue, 0 }  ,draw opacity=1 ][line width=0.75]    (4.37,-1.32) .. controls (2.78,-0.56) and (1.32,-0.12) .. (0,0) .. controls (1.32,0.12) and (2.78,0.56) .. (4.37,1.32)   ;
\draw [color={rgb, 255:red, 0; green, 0; blue, 0 }  ,draw opacity=1 ]   (300.63,110.18) -- (308.23,110.21) ;
\draw [shift={(310.23,110.22)}, rotate = 180.27] [color={rgb, 255:red, 0; green, 0; blue, 0 }  ,draw opacity=1 ][line width=0.75]    (4.37,-1.32) .. controls (2.78,-0.56) and (1.32,-0.12) .. (0,0) .. controls (1.32,0.12) and (2.78,0.56) .. (4.37,1.32)   ;
\draw [color={rgb, 255:red, 0; green, 0; blue, 0 }  ,draw opacity=1 ]   (300.99,79.84) -- (308.59,79.88) ;
\draw [shift={(310.59,79.89)}, rotate = 180.27] [color={rgb, 255:red, 0; green, 0; blue, 0 }  ,draw opacity=1 ][line width=0.75]    (4.37,-1.32) .. controls (2.78,-0.56) and (1.32,-0.12) .. (0,0) .. controls (1.32,0.12) and (2.78,0.56) .. (4.37,1.32)   ;
\draw [color={rgb, 255:red, 0; green, 0; blue, 0 }  ,draw opacity=1 ]   (300.38,119.91) -- (307.98,119.95) ;
\draw [shift={(309.98,119.96)}, rotate = 180.27] [color={rgb, 255:red, 0; green, 0; blue, 0 }  ,draw opacity=1 ][line width=0.75]    (4.37,-1.32) .. controls (2.78,-0.56) and (1.32,-0.12) .. (0,0) .. controls (1.32,0.12) and (2.78,0.56) .. (4.37,1.32)   ;
\draw [color={rgb, 255:red, 0; green, 0; blue, 0 }  ,draw opacity=1 ]   (300.39,130.24) -- (307.99,130.28) ;
\draw [shift={(309.99,130.29)}, rotate = 180.27] [color={rgb, 255:red, 0; green, 0; blue, 0 }  ,draw opacity=1 ][line width=0.75]    (4.37,-1.32) .. controls (2.78,-0.56) and (1.32,-0.12) .. (0,0) .. controls (1.32,0.12) and (2.78,0.56) .. (4.37,1.32)   ;
\draw [color={rgb, 255:red, 0; green, 0; blue, 0 }  ,draw opacity=1 ]   (300.16,149.98) -- (307.76,150.01) ;
\draw [shift={(309.76,150.02)}, rotate = 180.27] [color={rgb, 255:red, 0; green, 0; blue, 0 }  ,draw opacity=1 ][line width=0.75]    (4.37,-1.32) .. controls (2.78,-0.56) and (1.32,-0.12) .. (0,0) .. controls (1.32,0.12) and (2.78,0.56) .. (4.37,1.32)   ;
\draw [color={rgb, 255:red, 0; green, 0; blue, 0 }  ,draw opacity=1 ]   (300.34,139.84) -- (307.94,139.88) ;
\draw [shift={(309.94,139.89)}, rotate = 180.27] [color={rgb, 255:red, 0; green, 0; blue, 0 }  ,draw opacity=1 ][line width=0.75]    (4.37,-1.32) .. controls (2.78,-0.56) and (1.32,-0.12) .. (0,0) .. controls (1.32,0.12) and (2.78,0.56) .. (4.37,1.32)   ;
\draw [color={rgb, 255:red, 0; green, 0; blue, 0 }  ,draw opacity=1 ]   (300.6,99.91) -- (308.2,99.95) ;
\draw [shift={(310.2,99.95)}, rotate = 180.27] [color={rgb, 255:red, 0; green, 0; blue, 0 }  ,draw opacity=1 ][line width=0.75]    (4.37,-1.32) .. controls (2.78,-0.56) and (1.32,-0.12) .. (0,0) .. controls (1.32,0.12) and (2.78,0.56) .. (4.37,1.32)   ;
\draw   (421.36,181.28) -- (300.36,181.28) -- (300.36,50.52) ;
\draw  [fill={rgb, 255:red, 0; green, 0; blue, 0 }  ,fill opacity=1 ] (218.6,210.62) .. controls (218.6,209.64) and (219.4,208.84) .. (220.38,208.84) .. controls (221.36,208.84) and (222.16,209.64) .. (222.16,210.62) .. controls (222.16,211.6) and (221.36,212.4) .. (220.38,212.4) .. controls (219.4,212.4) and (218.6,211.6) .. (218.6,210.62) -- cycle ;
\draw  [fill={rgb, 255:red, 0; green, 0; blue, 0 }  ,fill opacity=1 ] (398.6,70.62) .. controls (398.6,69.64) and (399.4,68.84) .. (400.38,68.84) .. controls (401.36,68.84) and (402.16,69.64) .. (402.16,70.62) .. controls (402.16,71.6) and (401.36,72.4) .. (400.38,72.4) .. controls (399.4,72.4) and (398.6,71.6) .. (398.6,70.62) -- cycle ;
\draw  [fill={rgb, 255:red, 0; green, 0; blue, 0 }  ,fill opacity=1 ] (298.58,181.28) .. controls (298.58,180.3) and (299.37,179.5) .. (300.36,179.5) .. controls (301.34,179.5) and (302.14,180.3) .. (302.14,181.28) .. controls (302.14,182.26) and (301.34,183.06) .. (300.36,183.06) .. controls (299.37,183.06) and (298.58,182.26) .. (298.58,181.28) -- cycle ;
\draw  [fill={rgb, 255:red, 0; green, 0; blue, 0 }  ,fill opacity=1 ] (249.02,160.3) .. controls (249.02,159.32) and (249.82,158.52) .. (250.8,158.52) .. controls (251.78,158.52) and (252.58,159.32) .. (252.58,160.3) .. controls (252.58,161.28) and (251.78,162.08) .. (250.8,162.08) .. controls (249.82,162.08) and (249.02,161.28) .. (249.02,160.3) -- cycle ;
\draw  [dash pattern={on 0.84pt off 2.51pt}] (222.16,210.62) -- (280.18,210.62) -- (280.18,110.8) ;
\draw [color={rgb, 255:red, 155; green, 155; blue, 155 }  ,draw opacity=1 ]   (300.16,159.58) -- (300.16,151.98) ;
\draw [shift={(300.16,149.98)}, rotate = 90] [color={rgb, 255:red, 155; green, 155; blue, 155 }  ,draw opacity=1 ][line width=0.75]    (4.37,-1.32) .. controls (2.78,-0.56) and (1.32,-0.12) .. (0,0) .. controls (1.32,0.12) and (2.78,0.56) .. (4.37,1.32)   ;
\draw  [dash pattern={on 0.84pt off 2.51pt}] (280.18,110.4) -- (400.38,110.4) -- (400.38,70.62) ;
\draw [color={rgb, 255:red, 155; green, 155; blue, 155 }  ,draw opacity=1 ][fill={rgb, 255:red, 155; green, 155; blue, 155 }  ,fill opacity=1 ]   (250.9,120.24) -- (258.5,120.28) ;
\draw [shift={(260.5,120.29)}, rotate = 180.27] [color={rgb, 255:red, 155; green, 155; blue, 155 }  ,draw opacity=1 ][line width=0.75]    (4.37,-1.32) .. controls (2.78,-0.56) and (1.32,-0.12) .. (0,0) .. controls (1.32,0.12) and (2.78,0.56) .. (4.37,1.32)   ;
\draw [color={rgb, 255:red, 155; green, 155; blue, 155 }  ,draw opacity=1 ][fill={rgb, 255:red, 155; green, 155; blue, 155 }  ,fill opacity=1 ]   (250.81,139.98) -- (258.41,140.01) ;
\draw [shift={(260.41,140.02)}, rotate = 180.27] [color={rgb, 255:red, 155; green, 155; blue, 155 }  ,draw opacity=1 ][line width=0.75]    (4.37,-1.32) .. controls (2.78,-0.56) and (1.32,-0.12) .. (0,0) .. controls (1.32,0.12) and (2.78,0.56) .. (4.37,1.32)   ;
\draw [color={rgb, 255:red, 155; green, 155; blue, 155 }  ,draw opacity=1 ][fill={rgb, 255:red, 155; green, 155; blue, 155 }  ,fill opacity=1 ]   (250.99,129.84) -- (258.59,129.88) ;
\draw [shift={(260.59,129.89)}, rotate = 180.27] [color={rgb, 255:red, 155; green, 155; blue, 155 }  ,draw opacity=1 ][line width=0.75]    (4.37,-1.32) .. controls (2.78,-0.56) and (1.32,-0.12) .. (0,0) .. controls (1.32,0.12) and (2.78,0.56) .. (4.37,1.32)   ;
\draw [color={rgb, 255:red, 155; green, 155; blue, 155 }  ,draw opacity=1 ][fill={rgb, 255:red, 155; green, 155; blue, 155 }  ,fill opacity=1 ]   (250.6,149.91) -- (258.2,149.95) ;
\draw [shift={(260.2,149.95)}, rotate = 180.27] [color={rgb, 255:red, 155; green, 155; blue, 155 }  ,draw opacity=1 ][line width=0.75]    (4.37,-1.32) .. controls (2.78,-0.56) and (1.32,-0.12) .. (0,0) .. controls (1.32,0.12) and (2.78,0.56) .. (4.37,1.32)   ;
\draw [color={rgb, 255:red, 155; green, 155; blue, 155 }  ,draw opacity=1 ][fill={rgb, 255:red, 155; green, 155; blue, 155 }  ,fill opacity=1 ]   (250.75,79.94) -- (258.35,79.97) ;
\draw [shift={(260.35,79.98)}, rotate = 180.27] [color={rgb, 255:red, 155; green, 155; blue, 155 }  ,draw opacity=1 ][line width=0.75]    (4.37,-1.32) .. controls (2.78,-0.56) and (1.32,-0.12) .. (0,0) .. controls (1.32,0.12) and (2.78,0.56) .. (4.37,1.32)   ;
\draw [color={rgb, 255:red, 155; green, 155; blue, 155 }  ,draw opacity=1 ][fill={rgb, 255:red, 155; green, 155; blue, 155 }  ,fill opacity=1 ]   (250.81,99.51) -- (258.41,99.55) ;
\draw [shift={(260.41,99.56)}, rotate = 180.27] [color={rgb, 255:red, 155; green, 155; blue, 155 }  ,draw opacity=1 ][line width=0.75]    (4.37,-1.32) .. controls (2.78,-0.56) and (1.32,-0.12) .. (0,0) .. controls (1.32,0.12) and (2.78,0.56) .. (4.37,1.32)   ;
\draw [color={rgb, 255:red, 155; green, 155; blue, 155 }  ,draw opacity=1 ][fill={rgb, 255:red, 155; green, 155; blue, 155 }  ,fill opacity=1 ]   (250.99,89.38) -- (258.59,89.42) ;
\draw [shift={(260.59,89.43)}, rotate = 180.27] [color={rgb, 255:red, 155; green, 155; blue, 155 }  ,draw opacity=1 ][line width=0.75]    (4.37,-1.32) .. controls (2.78,-0.56) and (1.32,-0.12) .. (0,0) .. controls (1.32,0.12) and (2.78,0.56) .. (4.37,1.32)   ;
\draw [color={rgb, 255:red, 155; green, 155; blue, 155 }  ,draw opacity=1 ][fill={rgb, 255:red, 155; green, 155; blue, 155 }  ,fill opacity=1 ]   (250.6,109.45) -- (258.2,109.48) ;
\draw [shift={(260.2,109.49)}, rotate = 180.27] [color={rgb, 255:red, 155; green, 155; blue, 155 }  ,draw opacity=1 ][line width=0.75]    (4.37,-1.32) .. controls (2.78,-0.56) and (1.32,-0.12) .. (0,0) .. controls (1.32,0.12) and (2.78,0.56) .. (4.37,1.32)   ;
\draw [color={rgb, 255:red, 155; green, 155; blue, 155 }  ,draw opacity=1 ][fill={rgb, 255:red, 155; green, 155; blue, 155 }  ,fill opacity=1 ]   (250.44,70.09) -- (258.04,70.13) ;
\draw [shift={(260.04,70.13)}, rotate = 180.27] [color={rgb, 255:red, 155; green, 155; blue, 155 }  ,draw opacity=1 ][line width=0.75]    (4.37,-1.32) .. controls (2.78,-0.56) and (1.32,-0.12) .. (0,0) .. controls (1.32,0.12) and (2.78,0.56) .. (4.37,1.32)   ;

\draw (204.16,203.64) node [anchor=north west][inner sep=0.75pt]  [font=\footnotesize]  {$u$};
\draw (405.76,59.32) node [anchor=north west][inner sep=0.75pt]    {$v$};
\draw (211.36,154.32) node [anchor=north west][inner sep=0.75pt]  [font=\footnotesize]  {$( 0,0)$};
\draw (292.56,184.12) node [anchor=north west][inner sep=0.75pt]  [font=\footnotesize]  {$( m,\ -n)$};

\end{tikzpicture}
\qquad 
\tikzset{every picture/.style={line width=0.75pt}} 
\begin{tikzpicture}[x=0.75pt,y=0.75pt,yscale=-1,xscale=1]

\draw   (436.21,170.41) -- (289.15,170.18) -- (289.3,75) ;
\draw    (350.4,170.33) -- (350.34,162.31) ;
\draw [shift={(350.32,160.31)}, rotate = 89.54] [color={rgb, 255:red, 0; green, 0; blue, 0 }  ][line width=0.75]    (4.37,-1.32) .. controls (2.78,-0.56) and (1.32,-0.12) .. (0,0) .. controls (1.32,0.12) and (2.78,0.56) .. (4.37,1.32)   ;
\draw    (360.07,169.91) -- (360.07,162.31) ;
\draw [shift={(360.07,160.31)}, rotate = 90] [color={rgb, 255:red, 0; green, 0; blue, 0 }  ][line width=0.75]    (4.37,-1.32) .. controls (2.78,-0.56) and (1.32,-0.12) .. (0,0) .. controls (1.32,0.12) and (2.78,0.56) .. (4.37,1.32)   ;
\draw    (380.32,169.91) -- (380.32,162.31) ;
\draw [shift={(380.32,160.31)}, rotate = 90] [color={rgb, 255:red, 0; green, 0; blue, 0 }  ][line width=0.75]    (4.37,-1.32) .. controls (2.78,-0.56) and (1.32,-0.12) .. (0,0) .. controls (1.32,0.12) and (2.78,0.56) .. (4.37,1.32)   ;
\draw    (370.07,169.76) -- (370.07,162.16) ;
\draw [shift={(370.07,160.16)}, rotate = 90] [color={rgb, 255:red, 0; green, 0; blue, 0 }  ][line width=0.75]    (4.37,-1.32) .. controls (2.78,-0.56) and (1.32,-0.12) .. (0,0) .. controls (1.32,0.12) and (2.78,0.56) .. (4.37,1.32)   ;
\draw    (390.07,169.91) -- (390.07,162.31) ;
\draw [shift={(390.07,160.31)}, rotate = 90] [color={rgb, 255:red, 0; green, 0; blue, 0 }  ][line width=0.75]    (4.37,-1.32) .. controls (2.78,-0.56) and (1.32,-0.12) .. (0,0) .. controls (1.32,0.12) and (2.78,0.56) .. (4.37,1.32)   ;
\draw    (400.07,169.91) -- (400.07,162.31) ;
\draw [shift={(400.07,160.31)}, rotate = 90] [color={rgb, 255:red, 0; green, 0; blue, 0 }  ][line width=0.75]    (4.37,-1.32) .. controls (2.78,-0.56) and (1.32,-0.12) .. (0,0) .. controls (1.32,0.12) and (2.78,0.56) .. (4.37,1.32)   ;
\draw [color={rgb, 255:red, 155; green, 155; blue, 155 }  ,draw opacity=1 ]   (360.07,189.91) -- (360.07,182.31) ;
\draw [shift={(360.07,180.31)}, rotate = 90] [color={rgb, 255:red, 155; green, 155; blue, 155 }  ,draw opacity=1 ][line width=0.75]    (4.37,-1.32) .. controls (2.78,-0.56) and (1.32,-0.12) .. (0,0) .. controls (1.32,0.12) and (2.78,0.56) .. (4.37,1.32)   ;
\draw [color={rgb, 255:red, 155; green, 155; blue, 155 }  ,draw opacity=1 ]   (350.4,179.93) -- (350.4,172.33) ;
\draw [shift={(350.4,170.33)}, rotate = 90] [color={rgb, 255:red, 155; green, 155; blue, 155 }  ,draw opacity=1 ][line width=0.75]    (4.37,-1.32) .. controls (2.78,-0.56) and (1.32,-0.12) .. (0,0) .. controls (1.32,0.12) and (2.78,0.56) .. (4.37,1.32)   ;
\draw [color={rgb, 255:red, 155; green, 155; blue, 155 }  ,draw opacity=1 ]   (370.32,201.01) -- (370.32,193.41) ;
\draw [shift={(370.32,191.41)}, rotate = 90] [color={rgb, 255:red, 155; green, 155; blue, 155 }  ,draw opacity=1 ][line width=0.75]    (4.37,-1.32) .. controls (2.78,-0.56) and (1.32,-0.12) .. (0,0) .. controls (1.32,0.12) and (2.78,0.56) .. (4.37,1.32)   ;
\draw [color={rgb, 255:red, 155; green, 155; blue, 155 }  ,draw opacity=1 ]   (380.07,210.01) -- (380.07,202.41) ;
\draw [shift={(380.07,200.41)}, rotate = 90] [color={rgb, 255:red, 155; green, 155; blue, 155 }  ,draw opacity=1 ][line width=0.75]    (4.37,-1.32) .. controls (2.78,-0.56) and (1.32,-0.12) .. (0,0) .. controls (1.32,0.12) and (2.78,0.56) .. (4.37,1.32)   ;
\draw [color={rgb, 255:red, 155; green, 155; blue, 155 }  ,draw opacity=1 ]   (390.32,221.01) -- (390.32,213.41) ;
\draw [shift={(390.32,211.41)}, rotate = 90] [color={rgb, 255:red, 155; green, 155; blue, 155 }  ,draw opacity=1 ][line width=0.75]    (4.37,-1.32) .. controls (2.78,-0.56) and (1.32,-0.12) .. (0,0) .. controls (1.32,0.12) and (2.78,0.56) .. (4.37,1.32)   ;
\draw [color={rgb, 255:red, 155; green, 155; blue, 155 }  ,draw opacity=1 ]   (400.32,230.01) -- (400.32,222.41) ;
\draw [shift={(400.32,220.41)}, rotate = 90] [color={rgb, 255:red, 155; green, 155; blue, 155 }  ,draw opacity=1 ][line width=0.75]    (4.37,-1.32) .. controls (2.78,-0.56) and (1.32,-0.12) .. (0,0) .. controls (1.32,0.12) and (2.78,0.56) .. (4.37,1.32)   ;
\draw [color={rgb, 255:red, 155; green, 155; blue, 155 }  ,draw opacity=1 ]   (350.32,180.31) -- (358.07,180.31) ;
\draw [shift={(360.07,180.31)}, rotate = 180] [color={rgb, 255:red, 155; green, 155; blue, 155 }  ,draw opacity=1 ][line width=0.75]    (4.37,-1.32) .. controls (2.78,-0.56) and (1.32,-0.12) .. (0,0) .. controls (1.32,0.12) and (2.78,0.56) .. (4.37,1.32)   ;
\draw [color={rgb, 255:red, 155; green, 155; blue, 155 }  ,draw opacity=1 ]   (340.65,170.33) -- (348.4,170.33) ;
\draw [shift={(350.4,170.33)}, rotate = 180] [color={rgb, 255:red, 155; green, 155; blue, 155 }  ,draw opacity=1 ][line width=0.75]    (4.37,-1.32) .. controls (2.78,-0.56) and (1.32,-0.12) .. (0,0) .. controls (1.32,0.12) and (2.78,0.56) .. (4.37,1.32)   ;
\draw [color={rgb, 255:red, 155; green, 155; blue, 155 }  ,draw opacity=1 ]   (370.32,200.31) -- (378.07,200.31) ;
\draw [shift={(380.07,200.31)}, rotate = 180] [color={rgb, 255:red, 155; green, 155; blue, 155 }  ,draw opacity=1 ][line width=0.75]    (4.37,-1.32) .. controls (2.78,-0.56) and (1.32,-0.12) .. (0,0) .. controls (1.32,0.12) and (2.78,0.56) .. (4.37,1.32)   ;
\draw [color={rgb, 255:red, 155; green, 155; blue, 155 }  ,draw opacity=1 ]   (390.57,220.41) -- (398.32,220.41) ;
\draw [shift={(400.32,220.41)}, rotate = 180] [color={rgb, 255:red, 155; green, 155; blue, 155 }  ,draw opacity=1 ][line width=0.75]    (4.37,-1.32) .. controls (2.78,-0.56) and (1.32,-0.12) .. (0,0) .. controls (1.32,0.12) and (2.78,0.56) .. (4.37,1.32)   ;
\draw [color={rgb, 255:red, 155; green, 155; blue, 155 }  ,draw opacity=1 ]   (380.07,210.01) -- (387.82,210.01) ;
\draw [shift={(389.82,210.01)}, rotate = 180] [color={rgb, 255:red, 155; green, 155; blue, 155 }  ,draw opacity=1 ][line width=0.75]    (4.37,-1.32) .. controls (2.78,-0.56) and (1.32,-0.12) .. (0,0) .. controls (1.32,0.12) and (2.78,0.56) .. (4.37,1.32)   ;
\draw [color={rgb, 255:red, 155; green, 155; blue, 155 }  ,draw opacity=1 ]   (360.07,189.91) -- (368.15,190.04) ;
\draw [shift={(370.15,190.08)}, rotate = 180.94] [color={rgb, 255:red, 155; green, 155; blue, 155 }  ,draw opacity=1 ][line width=0.75]    (4.37,-1.32) .. controls (2.78,-0.56) and (1.32,-0.12) .. (0,0) .. controls (1.32,0.12) and (2.78,0.56) .. (4.37,1.32)   ;
\draw    (250,79.92) -- (410.4,240.32) ;
\draw  [dash pattern={on 0.84pt off 2.51pt}]  (250.91,209.66) -- (319.77,209.94) ;
\draw  [dash pattern={on 0.84pt off 2.51pt}]  (319.77,209.94) -- (320.06,190.23) ;
\draw  [dash pattern={on 0.84pt off 2.51pt}]  (380.34,189.94) -- (320.06,190.23) ;
\draw  [dash pattern={on 0.84pt off 2.51pt}]  (380.06,111.94) -- (380.34,189.94) ;
\draw  [dash pattern={on 0.84pt off 2.51pt}]  (380.06,110.23) -- (400.34,110.23) ;
\draw  [fill={rgb, 255:red, 0; green, 0; blue, 0 }  ,fill opacity=1 ] (249.26,209.66) .. controls (249.26,208.74) and (250,208) .. (250.91,208) .. controls (251.83,208) and (252.57,208.74) .. (252.57,209.66) .. controls (252.57,210.57) and (251.83,211.31) .. (250.91,211.31) .. controls (250,211.31) and (249.26,210.57) .. (249.26,209.66) -- cycle ;
\draw  [fill={rgb, 255:red, 0; green, 0; blue, 0 }  ,fill opacity=1 ] (398.69,110.23) .. controls (398.69,109.31) and (399.43,108.57) .. (400.34,108.57) .. controls (401.26,108.57) and (402,109.31) .. (402,110.23) .. controls (402,111.14) and (401.26,111.89) .. (400.34,111.89) .. controls (399.43,111.89) and (398.69,111.14) .. (398.69,110.23) -- cycle ;
\draw  [fill={rgb, 255:red, 0; green, 0; blue, 0 }  ,fill opacity=1 ] (287.49,170.18) .. controls (287.49,169.26) and (288.23,168.52) .. (289.15,168.52) .. controls (290.06,168.52) and (290.81,169.26) .. (290.81,170.18) .. controls (290.81,171.09) and (290.06,171.83) .. (289.15,171.83) .. controls (288.23,171.83) and (287.49,171.09) .. (287.49,170.18) -- cycle ;
\draw  [fill={rgb, 255:red, 0; green, 0; blue, 0 }  ,fill opacity=1 ] (313.26,145.09) .. controls (313.26,144.17) and (314,143.43) .. (314.91,143.43) .. controls (315.83,143.43) and (316.57,144.17) .. (316.57,145.09) .. controls (316.57,146) and (315.83,146.74) .. (314.91,146.74) .. controls (314,146.74) and (313.26,146) .. (313.26,145.09) -- cycle ;
\draw    (339.54,169.4) -- (339.48,161.38) ;
\draw [shift={(339.46,159.39)}, rotate = 89.54] [color={rgb, 255:red, 0; green, 0; blue, 0 }  ][line width=0.75]    (4.37,-1.32) .. controls (2.78,-0.56) and (1.32,-0.12) .. (0,0) .. controls (1.32,0.12) and (2.78,0.56) .. (4.37,1.32)   ;
\draw [color={rgb, 255:red, 155; green, 155; blue, 155 }  ,draw opacity=1 ]   (340.65,170.33) -- (340.66,165.37) -- (340.82,161.94) ;
\draw [shift={(340.91,159.94)}, rotate = 92.66] [color={rgb, 255:red, 155; green, 155; blue, 155 }  ,draw opacity=1 ][line width=0.75]    (4.37,-1.32) .. controls (2.78,-0.56) and (1.32,-0.12) .. (0,0) .. controls (1.32,0.12) and (2.78,0.56) .. (4.37,1.32)   ;

\draw (237.83,209.74) node [anchor=north west][inner sep=0.75pt]    {$u$};
\draw (404.97,101.54) node [anchor=north west][inner sep=0.75pt]    {$v$};
\draw (253.54,165.97) node [anchor=north west][inner sep=0.75pt]  [font=\footnotesize]  {$( 0,0)$};
\draw (313.83,128.74) node [anchor=north west][inner sep=0.75pt]  [font=\footnotesize]  {$( r,r)$};

\end{tikzpicture}

\end{center}
\caption{\small Top: Illustration of Lemma \ref{nestedpoly} in the special case when $\mathcal{Y}_u$ and $\mathcal{Z}_v$ are southwest boundaries. Bottom: Illustration of Lemma \ref{relatetau} and Lemma \ref{dia_vs_sw}. Note that any directed path between $u$ and $v$ goes through a gray edge/arrow if and only if it goes through a black edge/arrow.}
\label{sec3fig2}
\end{figure}

Fix $u\in \mathbb{Z}^2$, we will define a polymer with a general down-right boundary with the base at $u$. Let $\mathcal{Y}_u = \{y_i\}_{i\in \mathbb{Z}}$ be a bi-infinite downright path going through $u$. We use the convention that $y_0 = u$ and $y_i\cdot e_1 \leq y_j\cdot e_1$ if $i \leq j$. 

Next, let us place positive edge weights $\{S_{y_{i-1}, y_i}\}$ along $\mathcal{Y}_u$, and we will define the following function $H$. Let $H_{u, u} = 1$. For each $x_0 = y_m$ for some $m > 0$, define
$$H_{u, x_0} = \prod_{n=1}^m \wt{Y}_{y_{n-1}, y_n} \qquad \text{where } \wt{Y}_{y_{n-1}, y_n}  = \begin{cases}
S_{y_{n-1}, y_n} \quad & \text{if $y_n-y_{n-1} = e_1$,}\\
1/S_{y_{n-1}, y_n} \quad & \text{if $y_n-y_{n-1} = -e_2$.}
\end{cases}$$ 
For each $x_0 = y_{-m}$ for some $m > 0$, define
$$H_{u, x_0} = \prod_{n=0}^{-m+1} \wt{Y}_{y_n, y_{n-1}} \qquad \text{where } \wt{Y}_{y_n, y_{n-1}}  = \begin{cases}
1/S_{y_n, y_{n-1}} \quad & \text{if $y_n-y_{n-1} = e_1$,}\\
S_{y_n, y_{n-1}} \quad & \text{if $y_n-y_{n-1} = -e_2$.}
\end{cases}$$ 

Recall $\mathcal{Y}_u^{\geq} = \cup_n (y_n + \mathbb{Z}_{\geq 0})$ and $\mathcal{Y}_u^{>} = \cup_n (y_n + \mathbb{Z}_{> 0})$. For each $y\in\mathcal{Y}_u$ and $v \in \mathcal{Y}_u^{>}$, 
define the set of paths
    \[\mathbb{X}_{y, v}^{\mathcal{Y}_u}=\{x_{\bbullet} \in \mathbb{X}_{y, v}:x_1\in\mathcal{Y}_u^{>}\}.\]
This set is empty if both $y+e_i$, $i\in\{1,2\}$, are on $\mathcal{Y}_u$.
For $v\in\mathcal{Y}_u^{>}$, define the partition function
$$Z^{\mathcal{Y}_u}_{u, v} = \sum_{y\in\mathcal{Y}_u} \sum_{x_{\bbullet} \in \mathbb{X}_{y, v}^{\mathcal{Y}_u}} H_{u, y}\prod_{i=1}^{|y-v|_1} {Y}_{x_i},$$
where $\{Y_z\}$ are the bulk weights for $z\in \mathcal{Y}_u^{>}$. For $v\in\mathcal{Y}_u$ let $Z^{\mathcal{Y}_u}_{u,v}=H_{u,v}$. The corresponding quenched path measure will be denoted as $Q^{\mathcal{Y}_u}_{u, v}$. Note that these partition functions satisfy the following induction: for $w\in\mathcal{Y}_u^{>}$,
\begin{align}\label{ZY-ind}
Z^{\mathcal{Y}_u}_{u, w}=(Z^{\mathcal{Y}_u}_{u, w-e_1}+Z^{\mathcal{Y}_u}_{u, w-e_2})Y_w.
\end{align}

Given a polymer model defined on $\mathcal{Y}_u^{\geq}$. We fix another bi-infinite down-right path $\mathcal{Z}_v \subset \mathcal{Y}_u^{\geq} $ and define the following nested polymer model rooted at $v$. It has the same bulk weights, and on the new boundary $\mathcal{Z}_v =\{z_n\}$, the weights are given by 
$$S_{z_{n-1}, z_n} = \begin{cases}
\frac{Z^{\mathcal{Y}_u}_{u, z_{n}}}{Z^{\mathcal{Y}_u}_{u, z_{n-1}}} \quad & \text{if $z_n-z_{n-1} = e_1$,}\\[10pt]
\frac{Z^{\mathcal{Y}_u}_{u, z_{n-1}}}{Z^{\mathcal{Y}_u}_{u, z_{n}}} \quad & \text{if $z_n-z_{n-1} = -e_2$}.
\end{cases}$$
We will denote this nested polymer measure by $Q^{\mathcal{Z}_v, (\mathcal{Y}_u)}_{v, \bbullet}$.

\begin{lemma}\label{ratio_agrees}
Fix $u, v \in \mathbb{Z}^2$ and two down-right bi-infinite paths $\mathcal{Y}_u$ and $\mathcal{Z}_v$ with $\mathcal{Z}_v \subset \mathcal{Y}_u^{\geq}.$ Then for $w \in \mathcal{Z}^{\geq 0}_v$,
\begin{align}\label{tempZZY}
Z^{\mathcal{Z}_v, (\mathcal{Y}_u)}_{v, w}=\frac{Z^{ \mathcal{Y}_u}_{u, w}}{Z^{ \mathcal{Y}_u}_{u, v}}\,.
\end{align}
Consequently, for each $w \in \mathcal{Z}^{\geq 0}_v$ and $i\in\{1,2\}$, 
\begin{equation}\label{ratio_same_eq}
\frac{Z^{\mathcal{Y}_u}_{u, w+e_i}}{Z^{ \mathcal{Y}_u}_{u, w}} = \frac{Z^{\mathcal{Z}_v, (\mathcal{Y}_u)}_{v, w+e_i}}{Z^{\mathcal{Z}_v, (\mathcal{Y}_u)}_{v, w}}.
\end{equation}
\end{lemma}
\begin{proof}
When $w\in\mathcal{Z}_v$ the equality \eqref{tempZZY} comes straight from the definitions. Then it follows for $w\in\mathcal{Z}_v^>$ because the two sides satisfy the same induction \eqref{ZY-ind}.
\end{proof}

\begin{lemma} \label{nestedpoly}
Fix $u, v \in \mathbb{Z}^2$ and two down-right bi-infinite paths $\mathcal{Y}_u$ and $\mathcal{Z}_v$ with $\mathcal{Z}_v \subset \mathcal{Y}_u^{\geq}.$ Let $i\in\{1,2\}$ and $z \in \mathcal{Z}_v$ be such that $z+e_i$ is inside $\mathcal{Z}_v^{>0}$.
Then, for each $w\in \mathcal{Z}_v^{>0}$.
$$Q^{\mathcal{Y}_u}_{u, w}\{\text{path goes through $[\![z, z+e_i]\!]$}\} = Q^{\mathcal{Z}_v, (\mathcal{Y}_u)}_{v, w}\{\text{path goes through $[\![z, z+e_{i}]\!]$}\}.$$
\end{lemma}
\begin{proof}
We prove the case with $i=2$, the other case being symmetric. Then 
\begin{align*}
&Q^{\mathcal{Y}_u}_{u,w}\{\text{path goes through the edge $[\![z, z+e_2]\!]$}\}
=\frac{Z^{\mathcal{Y}_u}_{u,z} \cdot Z_{z+e_2, w}}{Z^{\mathcal{Y}_u}_{u,w}} = \frac{\frac{Z^{\mathcal{Y}_u}_{u,z}}{Z^{\mathcal{Y}_u}_{u,v}} \cdot Z_{z+e_2, z}}{\frac{Z^{\mathcal{Y}_u}_{u,w}}{Z^{\mathcal{Y}_u}_{u,v}}}\\
&\qquad\qquad= \frac{Z^{\mathcal{Z}_v, (\mathcal{Y}_u)}_{v, z} \cdot Z_{z+e_2, z}}{Z^{\mathcal{Z}_v, (\mathcal{Y}_u)}_{v, w}} \qquad \text{ by Lemma \ref{ratio_agrees}}\\
&\qquad\qquad=  Q^{\mathcal{Z}_v, (\mathcal{Y}_u)}_{v, w}\{\text{path goes through the edge $[\![z, z+e_2]\!]$}\}.
\end{align*}
See the top panel in Figure \ref{sec3fig2} for an illustration.
\end{proof}

Next, we restrict attention to stationary polymers with southwest and antidiagonal boundaries. 
To simplify the notation, we will denote the respective partition functions by $Z_{u, \bbullet}$ and $Z_{u, \bbullet}^\textup{dia}$. The corresponding polymer measures are denoted by $Q_{u, \bbullet}$ and $Q_{u, \bbullet}^\textup{dia}$. For the antidiagonal boundaries, the bi-infinite paths are given by $\mathcal{S}_{u}=u+\mathcal{S}_{(0,0)}$, where $\mathcal{S}_{(0,0)}$ is given in \eqref{stair_S}. For the nested polymers, we will always assume the outer polymer has an antidiagonal boundary, and the nested partition functions with antidiagonal and southwest boundaries are denoted, respectively, by $Z_{v, \bbullet}^{(u), \textup{dia}}$ and $Z_{v, \bbullet}^{(u)}$. The corresponding polymer measures are denoted by $Q^{(u)}_{v,\bbullet}$ and $Q^{(u), \textup{dia}}_{v,\bbullet}$.

The following two lemmas  relate the exit times of two polymer processes with different starting points. They are illustrated on the bottom of Figure \ref{sec3fig2}. 

\begin{lemma}\label{relatetau}  Fix two base points $(0,0)$ and $(m,-n)$ with $m,n > 0$.  Take $u$ with $u\le (0,0)$ and $u\le (m,-n)$.
Let $Z^{(u)}_{0,\,\babullet}$ and $Z^{(u)}_{(m,-n),\,\babullet}$  be the partition functions of the polymers with southwest boundaries, rooted at $(0,0)$ and $(m,-n)$, respectively, nested inside a polymer rooted at $u$ and having antidiagonal boundary $\mathcal{S}_{u}$.   
Then for $v\in ((0,0)+\Z_{>0}^2)\cap((m,-n)+\Z_{>0}^2)$,  
$$Q^{(u)}_{0, v}\{\tau \leq m \}  = Q^{(u)}_{(m,-n), v}\{\tau < -n\}.$$
\end{lemma}
\begin{proof}
This lemma follows from Lemma \ref{nestedpoly} as we have the equalities
\begin{align*}
&Q^{(u)}_{0, v}\{\tau \leq m  \}\\ 
&\quad= Q^\textup{dia}_{u, v}\{\{\text{path goes through edges $\{[\![a, a+e_2]\!] : 0 < a\cdot e_1 \leq m \text{ and } a\cdot e_2 = 0$}\}\} \cup\\
& \qquad \qquad \qquad  \{\text{path goes through edges $\{[\![a, a+e_1]\!] : 0 < a\cdot e_2 \leq v\cdot e_2 \text{ and } a\cdot e_1 = 0 $}\}\}\\
&\quad= Q^\textup{dia}_{u, v}\{\text{path goes through edges $\{[\![b, b+e_1]\!] : 0 < b\cdot e_2 \leq v\cdot e_2 \text{ and } b\cdot e_1 = m $}\}\\
&\quad=Q^{(u)}_{(m,-n), v}\{\tau < -n\}.\qedhere
\end{align*}
\end{proof}

Recall the exit time from the antidiagonal boundary, defined above \eqref{dia_exit}.

\begin{lemma}\label{dia_vs_sw}
 Fix two base points $(0,0)$ and $(r,r)$ with $r \in \mathbb{Z}_{>0}$. Take $u \in -\mathbb{Z}^2_{> 0}$.
 Let $Z^{(u)}_{0,\,\babullet}$ and $Z^{(u), \textup{dia}}_{(r,r),\,\babullet}$ be the partition functions of the polymers with southwest and antidiagonal boundaries, rooted at $(0,0)$ and $(r,r)$, respectively, nested inside a polymer rooted at $u$ and having  antidiagonal boundary $\mathcal{S}_{u}$.   Then for $v\in (r,r)+\Z_{>0}^2$, 
$$Q^{(u)}_{0, v}\{\tau \geq 2r \}  = Q^{(u), \textup{dia}}_{(r,r), v}\{\tau^{\textup{dia}} \geq r\}.$$
\end{lemma}

\begin{proof}
This lemma again follows from Lemma \ref{nestedpoly} as we have the equalities
\begin{align*}
&Q^{(u)}_{0, v}\{\tau \ge 2r \}  \\ 
&\quad= Q^\textup{dia}_{u, z}\{\text{path goes through edges $\{[\![a, a+e_2]\!] : 2r \leq a\cdot e_1 \leq v\cdot e_1 \text{ and } a\cdot e_2 = 0 $}\}\}\\
&\quad= Q^\textup{dia}_{u, z}\Big\{\{\text{path goes through edges $\{[\![(b,2r-b), (b,2r-b)+e_1]\!] : 2r \leq b < v\cdot e_1$}\}\}\cup\\
& \qquad \qquad \qquad \{\text{path goes through edges $\{[\![(b,2r-b), (b,2r-b)+e_2]\!] : 2r \leq b \leq v\cdot e_1$}\}\}\Big\}\\
&\quad=Q^{(u),\textup{dia}}_{(r,r), v}\{\tau^{\textup{dia}} \geq r\}.\qedhere
\end{align*}
\end{proof}

\subsection{Radon-Nikodym derivative calculation}\label{RNsec}

Given $a>0$, $N\in\mathbb Z_{>0}$, and $\rho>0$, let $P^{{\rho}}$ denote the probability distribution on the product space $\Omega = \mathbb{R}^{\floor{aN^{2/3}}}$ under which the coordinates $X_i(\omega) = \omega_i$ are i.i.d.\ $\mathrm{Ga}^{-1}({{\rho}})$ random variables. 

\begin{proposition}\label{RNest}
Fix $\mu>0$ and $\varepsilon\in(0,\mu/2)$. There exists a positive constant $C$ that only depends on $\varepsilon$ and $\mu$ and such that 
the following holds. Take any $a>0$, $b\in \mathbb{R}$, and $N \in \mathbb{Z}_{>0}$, and any  ${{\rho}} \in [\varepsilon, \mu-\varepsilon]$.
Take $|b|\le\frac14\varepsilon N^{1/3}$ and 
let $f$ denote the Radon-Nikodym derivative 
$$f  =  \frac{dP^{{{\rho}}+bN^{-1/3}}}{dP^{{\rho}}}\,.$$
Then 
    $$E^{P^{{\rho}}}[f^2] \leq e^{C ab^2}.$$
\end{proposition}

\begin{proof}
Let us denote $\lambda = {{\rho}}+ bN^{-1/3}$.
From a direct computation, we obtain
\begin{align}\label{rnd1}
E^{P^{{\rho}}}[f^2] & = \int \Big(\prod_{i=1}^{\floor{aN^{2/3}}}\frac{\frac{1}{\Gamma(\lambda)}\frac{1}{\omega_i^{\lambda+1}} e^{-\frac{1}{\omega_i}}}{\frac{1}{\Gamma({{\rho}})}\frac{1}{\omega_i^{{{\rho}}+1}} e^{-\frac{1}{\omega_i}}} \Big)^2 P(d\omega) \nonumber \\
&= \Big(\frac{\Gamma({{\rho}})^2}{\Gamma(\lambda)^2}\frac{1}{\Gamma({{\rho}})} \int_0^\infty \frac{1}{x^{2\lambda - {{\rho}} + 1}} e^{-\frac{1}{x}} dx\Big)^{\floor{aN^{2/3}}}\nonumber \\
& =  \Big(\frac{\Gamma({{\rho}})\Gamma(2\lambda - {{\rho}})}{\Gamma(\lambda)^2}\Big)^{\floor{aN^{2/3}}}\,.
\end{align}

We continue by taking the logarithm of \eqref{rnd1},
$$\log \eqref{rnd1} =  \floor{aN^{2/3}}\Big( \log\Gamma({{\rho}}) +  \log\Gamma(2\lambda-{{\rho}}) -2\log\Gamma(\lambda) \Big).$$
Note that ${{\rho}} = \lambda - bN^{-1/3}$ and $2\lambda - {{\rho}} = \lambda + bN^{-1/3}$. We can thus assume that $b>0$, the other case being symmetric.
Next, note that if we Taylor expand 
\begin{equation}\label{3_log_gamma}\log\Gamma({{\rho}}) +  \log\Gamma(2\lambda-{{\rho}}) -2\log\Gamma(\lambda),
\end{equation} 
then both the zeroth and the first derivative terms cancel out. 

The assumption $0<b\le \frac14\varepsilon N^{1/3}$ implies that
$$0<\varepsilon\le {{\rho}} < \lambda < 2\lambda - {{\rho}} \le \mu-\frac\varepsilon2<\mu.$$
In addition, $\log\Gamma(\abullet)$ is a smooth function on $\mathbb{R}_{>0}$.
Thus, the second derivative term and the remainder from the expansion can be upper bounded using a constant $C'$ depending only on $\varepsilon$ and $\mu$ and we get
$$\eqref{3_log_gamma} \leq  C' b^2N^{-2/3} + C' b^3N^{-1}.$$
Again, by the assumption on $b$, $C' b^2N^{-2/3} + C' b^3N^{-1} \leq (1+\varepsilon/4)C' b^2N^{-2/3}$. The claim follows with $C=(1+\varepsilon/4)C'$.
\end{proof}

\subsection{Sub-exponential random variables}

Let $\{X_i\}$ be a sequence of i.i.d.\ sub-exponential random variables with parameters $K_0>0$ and $\lambda_0>0$. This means 
\begin{align}\label{subexp}
\mathbb{E}[e^{\lambda (X_1-\mathbb{E}[X_1])}] \leq e^{K_0 \lambda^2} \quad \textup{ for all $\lambda \in [0, \lambda_0]$}.
\end{align}

Define $S_0 = 0$ and $S_k = X_1 + \dots + X_k - k\mathbb{E}[X_1]$ for $k\geq 1$. The following theorem captures the right tail behavior of the running maximum.

\begin{theorem}\label{max_sub_exp}
Assume \eqref{subexp}. Then
$$\mathbb{P} \Big(\max_{0\leq k \leq n} S_k \geq t\sqrt{n}\Big) \leq 
\begin{cases}
e^{-t^2/(4K_0)} \quad  & \textup{if $t \leq 2\lambda_0 K_0 \sqrt n$}\,, \\
e^{-\frac{1}{2}\lambda_0 t\sqrt{n}} \quad  & \textup{if $t \geq 2\lambda_0 K_0 \sqrt n$\,.}
\end{cases}
$$ 
\end{theorem}

\begin{proof}
Since $S_k$ is a mean zero random walk,  $e^{\lambda S_k}$ is a non-negative sub-martingale for $\lambda \geq 0$. By Doob's maximal inequality, 
\begin{align*}
\mathbb{P} \Big(\max_{0\leq k \leq n} S_k \geq t\sqrt{n}\Big)  &= \mathbb{P} \Big(\max_{0\leq k \leq n} e^{\lambda S_k} \geq e^{\lambda t\sqrt{n}}\Big)
\leq \frac{\mathbb{E}[e^{\lambda S_n}]}{e^{\lambda t\sqrt{n}}}
 = \frac{\mathbb{E}[e^{\lambda(X_1-\mathbb E[X_1])}]^n}{e^{\lambda t\sqrt{n}}}\le e^{nK_0\lambda^2-\lambda t\sqrt n},
\end{align*}
where in the last inequality we applied \eqref{subexp}, for which we now assume $\lambda \in [0, \lambda_0]$. On this interval, 
the exponent $h(\lambda) = nK_0 \lambda^2 -  \lambda t\sqrt{n}$ is minimized at 
$\lambda_t = \min\{\lambda_0, \frac{t}{2K_0 \sqrt{n}}\}$ and 
$$h(\lambda_t) =  
\begin{cases}
-\frac{t^2}{4K_0} \quad  & \textup{if $t \leq 2\lambda_0K_0 \sqrt{n}$\,,}\\
nK_0\lambda_0^2 - \lambda_0 t\sqrt{n}  \leq -\frac{1}{2}\lambda_0 t\sqrt{n} \quad  & \textup{if $t \geq 2\lambda_0K_0 \sqrt{n}$\,.}
\end{cases}$$ 
The proof is complete.
\end{proof}

Next, we verify that log gamma and log inverse gamma random variables are sub-exponential. Recall that if $X\sim\textup{Ga}(\alpha)$, then $\mathbb E[\log X]=\Psi_0(\alpha)$, where $\Psi_0$ is the digamma function, i.e.\ $\Psi_0(\alpha)=(\log\Gamma(\alpha))'$.

\begin{proposition}\label{Ga_sub_exp}
Fix $\varepsilon \in (0, \mu/2)$. There exist positive constants $K_0, \lambda_0$ depending on $\varepsilon$ such that for each $\alpha \in [\varepsilon, \mu-\varepsilon]$ and $X\sim \textup{Ga}(\alpha)$, we have
$$
\mathbb{E}[e^{\lambda(\log X - \Psi_0(\alpha))}]\leq e^{K_0 \lambda^2} \qquad \textup{ for all $\lambda \in [-\lambda_0, \lambda_0]$}.
$$
\end{proposition}

\begin{proof}
First, note that $\mathbb{E}[X^\lambda] = \frac{\Gamma(\alpha+\lambda)}{\Gamma(\alpha)}$, provided that $\alpha+\lambda > 0$. This last condition can be guaranteed for all $\alpha>\varepsilon$ by taking $\lambda_0$ small enough (depending on $\varepsilon$).
Then, by Taylor's theorem,
\begin{align*}
\log\mathbb{E}[e^{\lambda(\log X - \Psi_0(\alpha))}] & = \log(\mathbb{E} [X^{\lambda}]e^{-\lambda \Psi_0(\alpha)})
  = \log\Gamma(\alpha +\lambda) -\log\Gamma(\alpha) - \lambda \Psi_0(\alpha)\\
& =\Psi_1(\alpha)\frac{\lambda^2}2 + o(\lambda^2)
 \leq K_0 \lambda^2,
\end{align*}
provided $\lambda_0$ is taken sufficiently small depending on $\varepsilon$. The constant $K_0$ can be chosen to not depend on $\alpha\in[\varepsilon,\mu-\varepsilon]$ because $\Psi_1$ is a smooth function on $\mathbb{R}_{>0}$.
\end{proof}

\subsection{Random walk estimates}\label{sec_rw}
Let $\{X_i\}_{i\in \mathbb{Z}_{>0}}$ be an i.i.d.~sequence of random variables with
$$\mathbb{E}[X_i] = \mu, \quad \Var[X_i] = 1 \quad \text{and} \quad \mathbb{E}[|X_i-\mu|^3] = c_3 < \infty.$$ Define $S_k = \sum_{i=1}^k X_i$ for $k \geq 1$. 
We have the following proposition which bounds the probability that the running maximum of a random walk is small. 
\begin{proposition}\label{rwest} There exists a positive  constant $C$ such that for any $l >0$, we have
\beq\mathbb{P}\Big(\max_{1\leq k \leq N} S_k < l\Big) \leq C(c_3l+c_3^2)(|\mu| + 1/\sqrt N).\eeq
\end{proposition}

This result follows directly from the following two results from  \cite{Nag-70}.
 
\begin{lemma}[\cite{Nag-70} Lemma 5]
There exists an absolute constant $C$ such that for any $l>0$
\beq\mathbb{P}\Big(\max_{1\leq k \leq N} S_k < l\Big) - \mathbb{P}\Big(\max_{1\leq k \leq N} S_k < 0\Big) \leq C(c_3l+c_3^2)(|\mu| + 1/\sqrt N).\eeq
\end{lemma}

\begin{lemma}[\cite{Nag-70} Lemma 7]
There exists an absolute constant $C$ such that 
\beq\mathbb{P}\Big(\max_{1\leq k \leq N} S_k < 0\Big) \leq Cc^2_3(|\mu| + 1/\sqrt N).\eeq
\end{lemma}

\setcounter{secnumdepth}{0}
\bibliographystyle{aop-no-url}
\bibliography{Exit}

\begin{thebibliography}{38}
\expandafter\ifx\csname natexlab\endcsname\relax\def\natexlab#1{#1}\fi
\expandafter\ifx\csname url\endcsname\relax
  \def\url#1{\texttt{#1}}\fi
\expandafter\ifx\csname urlprefix\endcsname\relax\def\urlprefix{URL }\fi

\bibitem{Bak-Li-19}
\textsc{Bakhtin, Y.} and \textsc{Li, L.} (2019).
\newblock Thermodynamic limit for directed polymers and stationary solutions of
  the {B}urgers equation.
\newblock \textit{Comm. Pure Appl. Math.} \textbf{72} 536--619.

\bibitem{Bal-Bus-Sep-20}
\textsc{Bal\'{a}zs, M.}, \textsc{Busani, O.} and \textsc{Sepp\"{a}l\"{a}inen,
  T.} (2020).
\newblock Non-existence of bi-infinite geodesics in the exponential corner
  growth model.
\newblock \textit{Forum Math.\ Sigma} \textbf{8} Paper No. e46, 34.

\bibitem{Bal-Sep-10}
\textsc{Bal{{\'a}}zs, M.} and \textsc{Sepp{{\"a}}l{{\"a}}inen, T.} (2010).
\newblock Order of current variance and diffusivity in the asymmetric simple
  exclusion process.
\newblock \textit{Ann. of Math. (2)} \textbf{171} 1237--1265.

\bibitem{Bar-COr-Dim-21}
\textsc{Barraquand, G.}, \textsc{Corwin, I.} and \textsc{Dimitrov, E.} (2021).
\newblock Fluctuations of the log-gamma polymer free energy with general
  parameters and slopes.
\newblock \textit{Probab. Theory Related Fields} \textbf{181} 113--195.

\bibitem{Bas-Sar-Sly-19}
\textsc{Basu, R.}, \textsc{Sarkar, S.} and \textsc{Sly, A.} (2019).
\newblock Coalescence of geodesics in exactly solvable models of last passage
  percolation.
\newblock \textit{J. Math. Phys.} \textbf{60} 093301, 22.

\bibitem{Bas-Sep-She-23-}
\textsc{Basu, R.}, \textsc{Sepp\"al\"ainen, T.} and \textsc{Shen, X.} (2023).
\newblock Temporal correlation in the inverse-gamma polymer. Forthcoming.

\bibitem{Bas-Sid-Sly-14-}
\textsc{Basu, R.}, \textsc{Sidoravicius, V.} and \textsc{Sly, A.} (2014).
\newblock Last passage percolation with a defect line and the solution of the
  slow bond problem. Preprint (\href{https://arxiv.org/abs/1408.3464}{\tt arXiv
  1408.3464}).

\bibitem{Bha-20}
\textsc{Bhatia, M.} (2020).
\newblock Moderate deviation and exit time estimates for stationary last
  passage percolation.
\newblock \textit{J. Stat. Phys.} \textbf{181} 1410--1432.

\bibitem{Bus-Sep-22-ejp}
\textsc{Busani, O.} and \textsc{Sepp\"{a}l\"{a}inen, T.} (2022).
\newblock Non-existence of bi-infinite polymers.
\newblock \textit{Electron. J. Probab.} \textbf{27} Paper No. 14, 40.

\bibitem{Cat-Gro-06}
\textsc{Cator, E.} and \textsc{Groeneboom, P.} (2006).
\newblock Second class particles and cube root asymptotics for {H}ammersley's
  process.
\newblock \textit{Ann. Probab.} \textbf{34} 1273--1295.

\bibitem{Com-17}
\textsc{Comets, F.} (2017).
\newblock \textit{Directed polymers in random environments}, vol. 2175 of
  \textit{Lecture Notes in Mathematics}.
\newblock Springer, Cham.
\newblock Lecture notes from the 46th Probability Summer School held in
  Saint-Flour, 2016.

\bibitem{Cor-12}
\textsc{Corwin, I.} (2012).
\newblock The {K}ardar-{P}arisi-{Z}hang equation and universality class.
\newblock \textit{Random Matrices Theory Appl.} \textbf{1} 1130001, 76.

\bibitem{Cor-16}
\textsc{Corwin, I.} (2016).
\newblock Kardar-{P}arisi-{Z}hang universality.
\newblock \textit{Notices Amer. Math. Soc.} \textbf{63} 230--239.

\bibitem{Hol-09}
\textsc{den Hollander, F.} (2009).
\newblock \textit{Random polymers}, vol. 1974 of \textit{Lecture Notes in
  Mathematics}.
\newblock Springer-Verlag, Berlin.
\newblock Lectures from the 37th Probability Summer School held in Saint-Flour,
  2007.

\bibitem{Emr-Jan-Xie-23-}
\textsc{Emrah, E.}, \textsc{Janjigian, C.} and \textsc{Xie, Y.} (2023).
\newblock Moderate deviation and exit point estimates for integrabledirected
  polymer models. Forthcoming.

\bibitem{Geo-Ras-Sep-17-ptrf-2}
\textsc{Georgiou, N.}, \textsc{Rassoul-Agha, F.} and
  \textsc{Sepp{\"a}l{\"a}inen, T.} (2017{\natexlab{a}}).
\newblock Geodesics and the competition interface for the corner growth model.
\newblock \textit{Probab. Theory Related Fields} \textbf{169} 223--255.

\bibitem{Geo-Ras-Sep-17-ptrf-1}
\textsc{Georgiou, N.}, \textsc{Rassoul-Agha, F.} and
  \textsc{Sepp{\"a}l{\"a}inen, T.} (2017{\natexlab{b}}).
\newblock Stationary cocycles and {B}usemann functions for the corner growth
  model.
\newblock \textit{Probab. Theory Related Fields} \textbf{169} 177--222.

\bibitem{Geo-etal-15}
\textsc{Georgiou, N.}, \textsc{Rassoul-Agha, F.},
  \textsc{Sepp{{\"a}}l{{\"a}}inen, T.} and \textsc{Yilmaz, A.} (2015).
\newblock Ratios of partition functions for the log-gamma polymer.
\newblock \textit{Ann. Probab.} \textbf{43} 2282--2331.

\bibitem{Geo-Sep-13}
\textsc{Georgiou, N.} and \textsc{Sepp{{\"a}}l{{\"a}}inen, T.} (2013).
\newblock Large deviation rate functions for the partition function in a
  log-gamma distributed random potential.
\newblock \textit{Ann. Probab.} \textbf{41} 4248--4286.

\bibitem{Gia-07}
\textsc{Giacomin, G.} (2007).
\newblock \textit{Random polymer models}.
\newblock Imperial College Press, London.

\bibitem{Hal-Tak-15}
\textsc{Halpin-Healy, T.} and \textsc{Takeuchi, K.~A.} (2015).
\newblock A {KPZ} cocktail---shaken, not stirred \dots toasting 30 years of
  kinetically roughened surfaces.
\newblock \textit{J. Stat. Phys.} \textbf{160} 794--814.

\bibitem{Hus-Hen-85}
\textsc{Huse, D.~A.} and \textsc{Henley, C.~L.} (1985).
\newblock Pinning and roughening of domain walls in ising systems due to random
  impurities.
\newblock \textit{Phys. Rev. Lett.} \textbf{54} 2708--2711.

\bibitem{Hus-Hen-Fis-85}
\textsc{Huse, D.~A.}, \textsc{Henley, C.~L.} and \textsc{Fisher, D.~S.} (1985).
\newblock {H}use, {H}enley, and {F}isher respond.
\newblock \textit{Physical Review Letters} \textbf{55} 2924--2924.

\bibitem{Imb-Spe-88}
\textsc{Imbrie, J.~Z.} and \textsc{Spencer, T.} (1988).
\newblock Diffusion of directed polymers in a random environment.
\newblock \textit{J. Statist. Phys.} \textbf{52} 609--626.

\bibitem{Jan-Ras-20-aop}
\textsc{Janjigian, C.} and \textsc{Rassoul-Agha, F.} (2020).
\newblock Busemann functions and {G}ibbs measures in directed polymer models on
  {$\mathbb Z^2$}.
\newblock \textit{Ann. Probab.} \textbf{48} 778--816.

\bibitem{Jan-Ras-Sep-22-1F1S-}
\textsc{Janjigian, C.}, \textsc{Rassoul-Agha, F.} and \textsc{Sepp\"al\"ainen,
  T.} (2022{\natexlab{a}}).
\newblock Ergodicity and synchronization of the {K}ardar-{P}arisi-{Z}hang
  equation. Preprint (\href{https://arxiv.org/abs/2211.06779}{\tt arXiv
  2211.06779}).

\bibitem{Jan-Ras-Sep-22-}
\textsc{Janjigian, C.}, \textsc{Rassoul-Agha, F.} and \textsc{Sepp\"al\"ainen,
  T.} (2022{\natexlab{b}}).
\newblock Geometry of geodesics through {B}usemann measures in directed
  last-passage percolation.
\newblock \textit{J. Eur. Math. Soc.} Online first
  (\href{https://arxiv.org/abs/1908.09040}{\tt arXiv 1908.09040}).

\bibitem{Lan-Sos-22-b-}
\textsc{Landon, B.} and \textsc{Sosoe, P.} (2022{\natexlab{a}}).
\newblock Tail bounds for the {O}'{C}onnell-{Y}or polymer. Preprint
  (\href{https://arxiv.org/abs/2209.12704}{\tt arXiv 2209.12704}).

\bibitem{Lan-Sos-22-a-}
\textsc{Landon, B.} and \textsc{Sosoe, P.} (2022{\natexlab{b}}).
\newblock Upper tail bounds for stationary {KPZ} models. Preprint
  (\href{https://arxiv.org/abs/2208.01507}{\tt arXiv 2208.015077}).

\bibitem{Nag-70}
\textsc{Nagaev, S.~V.} (1970).
\newblock On the speed of convergence in a boundary problem. {I}.
\newblock \textit{Theory of Probability \& Its Applications} \textbf{15}
  163--186.

\bibitem{Qua-12}
\textsc{Quastel, J.} (2012).
\newblock Introduction to {KPZ}.
\newblock In \textit{Current developments in mathematics, 2011}. Int. Press,
  Somerville, MA, 125--194.

\bibitem{Qua-Spo-15}
\textsc{Quastel, J.} and \textsc{Spohn, H.} (2015).
\newblock The one-dimensional {KPZ} equation and its universality class.
\newblock \textit{J. Stat. Phys.} \textbf{160} 965--984.

\bibitem{Sep-12-corr}
\textsc{Sepp{{\"a}}l{{\"a}}inen, T.} (2012).
\newblock Scaling for a one-dimensional directed polymer with boundary
  conditions.
\newblock \textit{Ann. Probab.} \textbf{40} 19--73.
\newblock Corrected version available at
  \href{https://arxiv.org/abs/0911.2446}{\tt arXiv:0911.2446}.

\bibitem{Sep-18}
\textsc{Sepp\"{a}l\"{a}inen, T.} (2018).
\newblock The corner growth model with exponential weights.
\newblock In \textit{Random growth models}, vol.~75 of \textit{Proc. Sympos.
  Appl. Math.} Amer. Math. Soc., Providence, RI, 133--201.

\bibitem{Sep-She-20}
\textsc{Sepp\"{a}l\"{a}inen, T.} and \textsc{Shen, X.} (2020).
\newblock Coalescence estimates for the corner growth model with exponential
  weights.
\newblock \textit{Electron. J. Probab.} \textbf{25} Paper No. 85, 31.
\newblock Corrected version available at
  \href{https://arxiv.org/abs/1911.03792}{\tt arXiv:1911.03792}.

\bibitem{Tho-00}
\textsc{Thorisson, H.} (2000).
\newblock \textit{Coupling, stationarity, and regeneration}.
\newblock Probability and its Applications (New York), Springer-Verlag, New
  York.

\bibitem{Xie-22}
\textsc{Xie, Y.} (2022).
\newblock \textit{Limiting Distributions and Deviation Estimates of Random
  Walks in Dynamic Random Environments}.
\newblock Thesis (Ph.D.)--Purdue University Graduate School.
\newblock
  \href{https://hammer.purdue.edu/articles/thesis/LIMITING_DISTRIBUTIONS_AND_DEVIATION_ESTIMATES_OF_RANDOM_WALKS_IN_DYNAMIC_RANDOM_ENVIRONMENTS/19643079/1}{DOI
  10.25394/PGS.19643079.V1}.

\bibitem{Zyg-22}
\textsc{Zygouras, N.} (2022).
\newblock Some algebraic structures in {KPZ} universality.
\newblock \textit{Probab. Surv.} \textbf{19} 590--700.

\end{thebibliography}

\end{document}